\documentclass[10pt]{amsart}
\usepackage{geometry}                
\usepackage{graphicx}
\usepackage{dsfont}
\usepackage{amssymb}
\usepackage{amsmath}
\usepackage{epstopdf}
\usepackage{fullpage}
\usepackage{enumerate}
\usepackage{color,subcaption}
\usepackage{amsthm}
\usepackage{placeins}
\usepackage[abs]{overpic}
\usepackage{array}
\usepackage{blkarray}
\usepackage{multirow}
\usepackage{float}
\usepackage{morefloats}
\usepackage{algorithm,algorithmic}
\usepackage{pgfplots}
\input{macros.sty}
\usepackage{tikz}

\usetikzlibrary{calc,positioning} 

\usepackage{amsopn}
\let\oldtocsection=\tocsection

\let\oldtocsubsection=\tocsubsection

\let\oldtocsubsubsection=\tocsubsubsection

\renewcommand{\tocsection}[2]{\hspace{0em}\oldtocsection{#1}{#2}\textbf}
\renewcommand{\tocsubsection}[2]{\hspace{1em}\oldtocsubsection{#1}{#2}}
\renewcommand{\tocsubsubsection}[2]{\hspace{2em}\oldtocsubsubsection{#1}{#2}}

\newcommand{\dv}{\text{\rm div}}
\renewcommand{\o}{\text{\rm o}}
\renewcommand{\d}{\text{\rm d}}

\newcommand{\zext}{z_{\text{\rm ext}}}
\newcommand{\Drot}{D_{\text{\rm rot}}}
\newcommand{\Dstat}{D_{\text{\rm stat}}}
\newcommand{\Dgap}{D_{\text{\rm gap}}}
\newcommand{\Dmag}{D_{\text{\rm mag}}}
\newcommand{\Tad}{\Theta_{\text{\rm ad}}}
\newcommand{\Wper}{W_{\text{\rm per}}}
\newcommand{\Tor}{\text{\rm Tor}}
\newcommand{\curl}{\textbf{curl}}
\newcommand{\com}{\text{\rm com}}

\newcommand{\tr}{\text{\rm tr}}

\newcommand{\I}{\text{\rm I}}
\newcommand{\Id}{\text{\rm Id}}

\newcommand{\calA}{{\mathcal A}}
\newcommand{\calB}{{\mathcal B}}

\newcommand{\calK}{{\mathcal K}}
\newcommand{\calF}{{\mathcal F}}
\newcommand{\calC}{{\mathcal C}}

\newcommand{\calL}{{\mathcal L}}

\newcommand{\calO}{{\mathcal O}}
\newcommand{\calT}{{\mathcal T}}

\newcommand{\Winfty}{W^{1,\infty}(\mathbb{R}^d,\mathbb{R}^d)}
\newcommand{\R}{{\mathbb R}}

\usepackage{soul,xcolor}
\setstcolor{violet}

\newcommand{\black}{\color{black}}

\usepackage{comment}
\usepackage{hyperref}
\usepackage{xargs}
\usepackage{xcolor}
\usepackage[colorinlistoftodos,prependcaption,textsize=tiny]{todonotes}
\usepackage{pdfpages} 

\newcommandx{\pcomment}[2][1=]{\todo[linecolor=red,backgroundcolor=red!25,bordercolor=red,#1]{#2}}
\newcommandx{\ccomment}[2][1=]{\todo[linecolor=olive,backgroundcolor=olive!25,bordercolor=olive,#1]{#2}}
\newcommandx{\acomment}[2][1=]{\todo[linecolor=blue,backgroundcolor=blue!25,bordercolor=blue,#1]{#2}}

\newif\ifuzero
\newif\ifacanum

\begin{document}
\newtheorem{theorem}{Theorem}[section]
\newtheorem{problem}{Problem}[section]
\newtheorem{remark}{Remark}[section]
\newtheorem{example}{Example}[section]
\newtheorem{definition}{Definition}[section]
\newtheorem{lemma}{Lemma}[section]
\newtheorem{corollary}{Corollary}[section]
\newtheorem{proposition}{Proposition}[section]
\numberwithin{equation}{section}

\title{Space-time shape optimization of rotating electric machines}
\author{
A. Cesarano\textsuperscript{1}, C. Dapogny\textsuperscript{2}, P. Gangl\textsuperscript{1}
}

\maketitle
\begin{center}
\emph{\textsuperscript{1} Johann Radon Institute of Computational and Applied Mathematics, Altenberger Stra\ss{}e 69, 4040 Linz}\\
\emph{\textsuperscript{2} Univ. Grenoble Alpes, CNRS, Grenoble INP\footnote{Institute of Engineering Univ. Grenoble Alpes}, LJK, 38000 Grenoble, France.}
\end{center}
 
 \begin{abstract}
 This article is devoted to the shape optimization of the internal structure of an electric motor, 
 and more precisely of the arrangement of air and ferromagnetic material inside the rotor part with the aim to increase the torque of the machine. 
 The governing physical problem is the time-dependent, non linear magneto-quasi-static version of Maxwell's equations. 
 This multiphase problem can be reformulated on a 2d section of the real cylindrical 3d configuration; however, 
 due to the rotation of the machine, the geometry of the various material phases at play (the ferromagnetic material, the permanent magnets, air, etc.) undergoes a prescribed motion over the considered time period. 
 This original setting raises a number of issues. From the theoretical viewpoint, we prove the well-posedness of this unusual non linear evolution problem featuring a moving geometry. 
 We then calculate the shape derivative of a performance criterion depending on the shape of the ferromagnetic phase via the corresponding magneto-quasi-static potential. 
 Our numerical framework to address this problem is based on a shape gradient algorithm. 
 The non linear time periodic evolution problems for the magneto-quasi-static potential is solved in the time domain, with a Newton-Raphson method. 
 The discretization features a space-time finite element method, applied on a precise, meshed representation of the space-time region of interest, 
 which encloses a body-fitted representation of the various material phases of the motor at all the considered stages of the time period. 
 After appraising the efficiency of our numerical framework on an academic problem, we present a quite realistic example of optimal design of the ferromagnetic phase of the rotor of an electric machine. 
  \end{abstract}
\bigskip
\bigskip
\hrule
\tableofcontents
\vspace{-0.5cm}
\hrule
\bigskip
\bigskip


\section{Introduction}
\noindent Electric motors are devices meant to convert electric energy into mechanical energy. 
For multiple reasons that are related to burning scientific and societal challenges,
they have recently aroused a tremendous enthusiasm in the academic and industrial communities.
Notably, as they can achieve much superior yield to thermal engines, their systematic deployment would allow for decisive energy savings,
not depending on expensive and increasingly scarce fossil fuels. 
They are also regarded as a promising energy production means in the perspective of the current environmental crisis, as their carbon footprint can be inferior to that of thermal motors by up to 80$\%$.
We refer to classical reference books such as \cite{fitzgerald2003electric,krause2013analysis} for a more exhaustive presentation of electric machines.

Anticipating a little on the more complete description of \cref{sec.setting}, an electric motor is composed of an external, static part called stator, and an internal, rotating part -- the rotor. An alternating electric current is injected through the coils placed
in the stator, thus generating a magnetic field which sets the rotor in motion. The resulting mechanical work is transmitted to a shaft attached to its core and can be used directly or stored, see \cref{fig_machine} below for an illustration.
The efficiency of this workflow strongly depends on the arrangement of the costly components of the rotor, in particular the ferromagnetic material and the
permanent magnets (usually made of rare-earth elements), and the optimization of their layout presages tremendous performance improvements, see \cite{neittaanmaki1996inverse} about this issue. 

Over the past decades, optimal design has been a thriving field of research at the crossroads of mathematics, physics and engineering; 
multiple frameworks are available to address such issues, with competing assets and drawbacks. 
Among them, shape optimization aims to improve an initial guess through smooth variations of the boundaries or material interfaces at play; this relies on the information contained in the ``shape gradients'' of the objective and constraint functionals of the problem -- see e.g. \cite{allaire2020survey,allaire2004structural,delfour2011shapes,henrot2018shape,murat1976controle,sokolowski1992introduction} and \cite{dapogny2022shape,GanglSturmNeunteufelSchoeberl2020,Blauth2021,Paganini2021} for numerical implementations. 
On the other hand, topology optimization methods affect the connectivity of a design, 
for instance by leveraging topological derivatives, indicating where holes can be fruitfully nucleated inside a shape \cite{amstutz2022introduction,amstutz2006new,novotny2012topological,novotny2020introduction}, 
or via a suitable reformulation of the physical situation and the optimal design problem under scrutiny in terms of density functions \cite{bendsoe2013topology,sigmund2013topology}.
Surprisingly enough, such techniques have only been applied recently in the physical context of electric motors.
Among the contributions in this direction and without aiming exhaustivity, let us mention the articles \cite{guo2020simultaneous,guo2020multimaterial,ma2019topology} using density-based topology optimization strategies, and those \cite{brun2023level,lee2012topological,ren2019topology} relying on the level set method;
see also \cite{kuci2021level} about a coupling between the level set method and a body-fitted meshed representation of the machine enabling more accurate finite element simulations. 
We refer to \cite{gangl2016sensitivity,lucchini2022topology} for recent surveys of optimal design techniques in the field of electromagnetism, and to \cite{bramerdorfer2018modern,wang2022topology} for a particular focus on their application to electric machines.

The specific setting of electric motors raises several challenges. At first, although the boundary value problem governing the behavior of the motor 
is a reduction of the full 3d Maxwell's equations to its 2d cross-section, it features a non linear material law which is essential to capture the decisive saturation effect of the ferromagnetic constituent of the rotor \cite{cullity2011introduction}. 
The physical problem at hand is also time-dependent; moreover, as the rotor is in rotation, the optimal design criterion depends on all the stages of this periodic motion; this feature is fairly awkward in view of optimal design because of the tremendous computational cost incurred, see e.g. \cite{KaliseKunischSturm2018,SchulzSiebenbornWelker2015}.
To the best of our knowledge and with the exception of the recent contribution \cite{GanglKrenn2022}, all the existing studies about the shape or topology optimization of electric motors resort to an additional simplification referred to as the magneto-static approximation of Maxwell's equations: the physical problem is thereby reduced to a series of decoupled static (elliptic and non linear) problems posed on the various rotated configurations of the motor, which can be solved independently from one another. 
Unfortunately, this convenient reformulation does not allow to account for subtle and realistic temporal effects such as eddy currents.

The present article is a natural continuation to the series \cite{GanglLangerLaurainMeftahiSturm2015,MerkelGanglSchoeps2021}.
We wish to optimize the repartition of the constituent materials of an electric motor; yet, contrary to the prevailing practice in the literature,
we rely on the genuine magneto-quasi-static description of the physical situation. 
The boundary value problem at stake can no longer be reduced to a collection of independent stationary problems: it is a non linear, time-dependent problem of mixed elliptic-parabolic type \cite{bachinger2005numerical}, 
equipped with time periodicity conditions, in which the optimization criterion brings into play all the rotated geometric configurations of the internal structure of the motor. 

From the theoretical viewpoint, our first contribution is to prove the well-posedness of this physical evolution problem, which leverages techniques from the theory of non linear partial differential equations. 
We next analyze shape optimization problems where the performance functionals at play depend on the solution to this problem. 
From the numerical viewpoint, the specific nature of the considered time-dependent problem raises the need for an adapted strategy. Our framework relies on 
 a space-time finite element discretization, see e.g. \cite{LangerMooreNeumueller2016,steinbach2015} in the context of linear problems, \cite{toulopoulos2022} for a quasilinear problem, or more recently \cite{GanglGobrialSteinbach2023} for an application in the context of electric machines. 
In a nutshell, the time variable is treated as if it were an additional space variable, thus converting a time-dependent problem posed in a $d$-dimensional domain into a static problem on a $(d+1)$-dimensional space time cylinder; 
the latter is solved by applying the finite element method on a mesh of the latter.
While this approach suffers from an increase in the dimension of the problem, 
it enjoys a number of unique assets over classical time-stepping methods (see e.g. \cite{Thomee2006}):
\begin{itemize}
\item The adaptive refinement of the mesh and the parallelization of the solution of the evolution problem can be realized in space and time, jointly; 
\item The adjoint problem involved in the shape derivative of the optimization criterion can be solved concurrently with the state equation, thus allowing for a further level of parallelization;
\item Shapes evolving over the considered time period are static in the space-time domain; hence, evolution problems involving such moving domains can be treated by ``standard'', time independent numerical methods once a space-time mesh adapted to the moving geometry is available; 
\item Time periodicity conditions are straightforward to enforce by direct identification of the degrees of freedom on the bottom and top sections of the space-time cylinder. 
\end{itemize}

The remainder of this article is organized as follows. In \cref{sec.setting}, we present the physical and mathematical settings of the boundary value problem governing the physics of motors, 
as well as the shape optimization problem considered in this setting. 
The next \cref{sec:genshape} is devoted to an academic, albeit instructive preliminary situation, that of the calculation of the shape derivative of a ``simple'' functional, depending on a domain $\Omega$ via the integral of a given function over a collection of deformed versions of $\Omega$. In \cref{sec_calcsdmqs}, we outline a few mathematical features of the considered magneto-quasi-static evolution problems, and we  
detail the calculation of a shape derivative in this context. 
After discussing a few details of our numerical implementation in \cref{sec.num}, we show in \cref{sec.numex} two numerical examples illustrating the previous developments;
a few conclusions and perspectives of our work are then given in \cref{sec.concl}. 
The article ends with two \cref{app.tech,app.varft}, where a few ``classical'' technical results are recalled, and the main steps of the proof of the well-posedness of the considered magneto-quasi-static problem are sketched.

\section{Modeling of the physical behavior of an electric motor}\label{sec.setting}

\noindent This section introduces the physical and mathematical aspects of the optimal design of electric motors. 
After a few generalities about these devices and their modeling in terms of the Maxwell's equations in \cref{sec.physmod}, we discuss the approximate 2d magneto-quasi-static setting in \cref{sec.2dmagneto}.
The internal structure of the considered motors is presented in \cref{sec.2dgeomotor}, and 
we eventually set the mathematical framework and the main notations used throughout the article in the next \cref{sec.not}.  

\subsection{The eddy current problem for the 3d Maxwell's equations} \label{sec.physmod}

\noindent 
This section describes, in an informal manner, the general operation of an electric motor and its mathematical formulation via the eddy current problem stemming from the 3d Maxwell's equations; 
our main references about this topic are Chap. 2 in \cite{gangl2016sensitivity}, \cite{rodriguez2010eddy} and \cite{Touzani2014}. \par\medskip

The motor under scrutiny has a cylindrical structure and its transverse section is sketched in \cref{fig_machine}. 
It is made of an inner rotating part called rotor, 
and a static outer part, the stator; these regions are separated by a thin air gap and the core of the rotor is connected to a transmission shaft. 
The stator is an arrangement of ferromagnetic material and coils made of thin copper wires;
the rotor is composed of ferromagnetic material, air inclusions and several permanent magnets.

In the classical language of electromagnetism, for which we refer to \cite{griffiths2017introduction,jackson2007classical} and whose main concepts and notations are used in this presentation, 
the device is activated when a time-dependent electric current $\mathbf J_i$ is powered into the coils of the stator.
According to Amp\`{e}re's law, this generates a time-dependent magnetic field $\mathbf H$ in the whole region of interest, 
which in turn induces an electric field $\mathbf E$ by Faraday's law. The resulting Lorentz force sets in motion the charge carriers contained in the rotor and thereby the attached shaft; this 
mechanical work is eventually used or stored. 

In order to put this rough sketch into equations, let us first note that, as opposed to light waves, the present application falls in the regime of low-frequency electromagnetism, so that the displacement currents, i.e. the time derivative of the magnetic induction $\mathbf D$, can be neglected, 
see \cite{rodriguez2010eddy} for an intuitive explanation of this fact and \cite{buffa2000justification} for a mathematical justification. 
The magnetic field ${\mathbf H}$ generated around the motor is then related to the total density of current ${\mathbf J}$ via the following magneto-quasi-static version of Amp\`ere's law:
\begin{equation} \label{eq_ampere}
    \curl (\mathbf H) = \mathbf J.
\end{equation}
On the other hand, this field is related to the magnetic induction (or magnetic flux density) ${\mathbf B}$ via the following constitutive law: 
\begin{align} \label{eq_constitutive}
    \mathbf H = \nu(|\mathbf B|)\mathbf B - \mathbf M.
\end{align}
Here, the coefficient $\nu$ represents the magnetic reluctivity, i.e. the inverse of the magnetic permeability; 
it expresses how the material develops a magnetic field in the presence of a magnetization force. 
The vector field ${\mathbf M}$ denotes the permanent magnetization, i.e. the density of permanent magnetic dipoles within the motor; this field vanishes outside the permanent magnets of the rotor. 
The reluctivity $\nu$ takes different values depending on the material phase -- a dependence with respect to the spatial location which is omitted in \cref{eq_constitutive} for simplicity. 
It assumes constant values $\nu_a$, $\nu_m$, $\nu_c$ in air, permanent magnets and copper, respectively.
By contrast, in ferromagnetic materials, it is a function of the intensity $\lvert \mathbf B \lvert$ of the magnetic induction: for small values of $\mathbf H$, the material amplifies the magnetic flux density $\mathbf B$, whereas for high values of $\mathbf H$, the relation between both quantities is similar to that occurring in void. We refer to \cite{pechstein2006monotonicity} about the calibration of such a reluctivity function from physical measurements.

According to Gauss's law, the magnetic induction ${\mathbf B}$ is solenoidal, i.e.
\begin{align} \label{eq_gauss}
    \mbox{div}(\mathbf B)=0.
\end{align}
In turn, the time variations of $\mathbf B$ induce an electric field ${\mathbf E}$ according to Faraday's law:
\begin{equation} \label{eq_faraday}
    \curl(\mathbf E) = - \frac{\partial \mathbf B}{\partial t}.
\end{equation}

This model is completed by Ohm's law, which expresses the density $\mathbf J$ of current inside the region of interest as 
the superposition of the density $\mathbf J_i$ of current impressed in the coils and that produced by the Lorentz force acting on the charge carriers: 
\begin{equation}\label{eq.Ohm}
\mathbf J = \mathbf J_i + \sigma ({\mathbf v} \times \mathbf B + \mathbf E) ,
\end{equation}
where ${ \mathbf v}$ is the velocity field of the rotating part and $\sigma$ is the electric conductivity, which takes a different, constant value inside each material phase.\par\medskip

The combination of these relations allows to express all the above quantities in terms of a single vector field -- the vector potential $\mathbf A$ -- and to characterize the latter by a partial differential equation. 
Indeed, at first,
 \cref{eq_gauss} implies the existence of a vector field $\mathbf A$ such that $\mathbf B = \curl( \mathbf A)$, which is determined up to the addition of the gradient of a scalar field. 
Judging from \cref{eq_faraday}, the vector potential $\mathbf A$ can be selected as the unique such vector field satisfying $\mathbf E = - \frac{\partial \mathbf A}{\partial t}$.
Now injecting the constitutive relation \cref{eq_constitutive} and the expression \cref{eq.Ohm} of the current density into Amp\`{e}re's law \cref{eq_ampere}, 
we obtain the three-dimensional magneto-quasi-static version of Maxwell's equation, also referred to as the eddy current problem:
\begin{align} \label{eq_mqs3d}
    \sigma \left( \frac{\partial \mathbf A}{\partial t} -   \mathbf v \times  \curl (\mathbf A) \right) +  \curl\Big( \nu(\lvert\curl( \mathbf A) \lvert)\curl( \mathbf A)\Big) = \mathbf J_i + \curl(\mathbf M).
\end{align}

The partial differential equation \cref{eq_mqs3d} is complemented by suitable time and spatial boundary conditions. 
As far as the former are considered, our study concerns the permanent regime of the motor, where the material distribution retrieves its initial configuration and all physical quantities return to their initial values after one rotation. 
The equation \cref{eq_mqs3d} is then equipped with a time periodic condition over the considered time period $[0,T]$. 
As regards boundary conditions, at each time $t \in [0,T]$, the usual transmission conditions are assumed at the interfaces between the various material phases, 
and no magnetic flux leaves the computational domain:
\begin{equation}\label{eq.noflux}
\mathbf B \cdot \mathbf n = 0 \text{ on the outer boundary of the stator},
\end{equation}
where $\mathbf n$ is the unit normal vector field to this boundary, pointing outward the device.

\subsection{The 2d magneto-quasi-static equation}\label{sec.2dmagneto}

\noindent As we have mentioned, the full 3d motor under scrutiny has a cylindrical structure. 
Taking advantage of its relatively large size along the shaft axis, that we suppose without loss of generality to be the line passing through the origin, oriented along the third coordinate vector ${\mathbf e}_3$, and 
of the invariance of its geometry in this direction, 
it is customary to approximate the three-dimensional problem \cref{eq_mqs3d} with a simpler two-dimensional version, posed in the cross-section of the device \cite{kolota2011analysis,torkaman2008comprehensive}.

This reduction is justified since the impressed current density $\mathbf J_{\text{i}}$ and the magnetization $\mathbf M$ are of the form
$$\mathbf J_i(t,x_1, x_2, x_3) = (0,0, f(t,x_1, x_2)), \:\: \mathbf M(t,x_1, x_2, x_3) = (M_1(t,x_1, x_2), M_2(t,x_1, x_2), 0)$$ 
and the velocity $\mathbf v$ representing the rotation of the rotor reads
$$ \mathbf v(t,x_1, x_2, x_3) = ( v_1(t,x_1,x_2),  v_2(t,x_1,x_2),0).$$
In turn, the vector potential $\mathbf A$ is oriented along the axis $\mathbf e_3$ and its values depend only on the spatial position in the cross-section $D$, i.e. 
$$\mathbf A(t,x_1, x_2, x_3) = (0,0, u(t,x_1,x_2)),$$
for a suitable function $u(t,x_1,x_2)$. 

Under these assumptions, the eddy current equation \cref{eq_mqs3d} reduces to a 2d partial differential equation for the scalar field $u$ which is posed in the cross-section $D$ of the machine; 
an elementary calculation indeed yields: 
\begin{align} \label{eq_mqs2d}
    \sigma \left( \frac{\partial u}{\partial t} +  v \cdot \nabla u \right) -  \mbox{div}( \nu(|\nabla u|) \nabla u) = f - \dv M^\perp, \qquad x \in D, \; t \in (0, T).
\end{align}
Here, we have dropped the bold font for vector fields when they lie inside the cross-section plane of the motor; we have set 
$$M(t,x_1, x_2) := (M_1(t,x_1, x_2), M_2(t,x_1, x_2)), \:\: v(t,x_1, x_2) := (v_1(t,x_1, x_2), v_2(t,x_1, x_2) ),$$ 
and we use the notation $y^\perp = (-y_2, y_1)$ for the 90$^\circ$ clockwise rotate of a two-dimensional vector $y = (y_1, y_2)$. \par\medskip

\begin{figure}
\centering
 \includegraphics[width=1.0\textwidth]{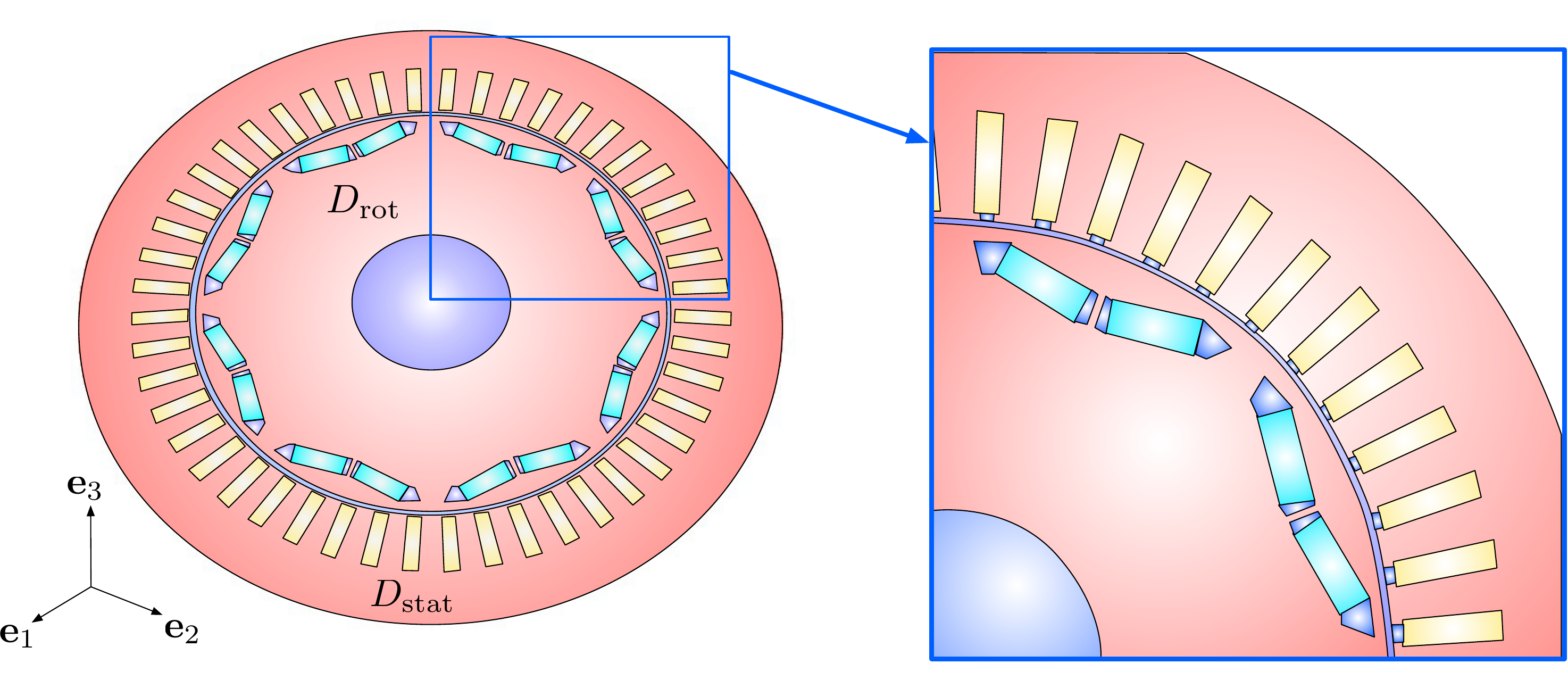}
    \caption{\it Two-dimensional cross-section $D$ of an electric machine of the form considered in \cref{sec.setting}: the ferromagnetic material is depicted in red, air is in dark blue, the coils are in yellow, and the permanent magnets are in light blue.}
    \label{fig_machine}
\end{figure}

Eventually, the 2d equation \cref{eq_mqs2d} is complemented with the homogeneous Dirichlet boundary condition
$$ u = 0 \text{ on } \partial D,$$
inherited from the no-flux assumption \cref{eq.noflux} made in the complete 3d model of \cref{sec.physmod}.  
As far as temporal boundary conditions are concerned, the time periodic setting implies that:
$$ \forall x = (x_1,x_2) \in D, \quad u(0,x) = u(T,x). $$

\begin{remark}
The mathematical analyses conducted in this article can be adapted straightforwardly to handle other time and boundary conditions about $u$, and notably the more classical initial condition $u(0,\cdot) = u_0$. 
Since the time periodic setting is the most relevant one for our purpose, we do not discuss those alternative instances for brevity.

\end{remark}

In engineering applications, the performance of the motor is often measured in terms of the average torque induced by the rotation of the motor shaft over the time period $(0,T)$. 
This quantity $\Tor(u)$ is evaluated in terms of the potential $u$ by the so-called Arkkio's method \cite{sadowski1992torque}:
\begin{equation}\label{eqn:torque}
   \Tor(u) := \frac{1}{T} \frac{L \, \nu_a}{(r_s-r_r)}  \int_0^T \int_\Sigma \; Q(x_1,x_2)  \nabla u \cdot \nabla u \: \d x \d t
\end{equation}
where $L$ is the length of the true, 3d electric machine in the ${\mathbf e}_3$ direction, $\Sigma$ is an annulus with inner and outer radii $r_r < r_s$ lying in the air gap between the rotor and the stator, and the $2 \times 2$ matrix $Q(x)$ is defined by:
\begin{align}
\forall x= (x_1,x_2) \in D, \quad Q(x_1,x_2) = \frac{1}{\sqrt{x_1^2 + x_2^2}}
\left( \begin{array}{cc} x_1 x_2 & \frac{x_2^2 - x_1^2}{2} \\[0.5em] \frac{x_2^2 - x_1^2}{2} & -x_1 x_2 \end{array} \right).
\end{align}

\subsection{Description of the geometric structure of the motor}\label{sec.2dgeomotor}

\noindent In this section, we get into more specifics about the constituent material phases of the motor and we
present the corresponding model governing the expressions of the conductivity and reluctivity coefficients $\sigma$ and $\nu$ in the formulation of the evolution problem \cref{eq_mqs2d}.
In the first \cref{sec.phaserest}, we describe the internal structure of the considered motor when it is at rest;
we then discuss its time evolution due to the rotation of the rotor in \cref{sec.phasemov}. 

\subsubsection{Structure of the motor at rest}\label{sec.phaserest}

\noindent
The structure of the 2d cross-section $D$ reflects the three-dimensional arrangement of the motor introduced in the previous \cref{sec.physmod}; it features three disjoint annulus-shaped regions:
$$ \overline D = \overline{\Drot} \cup \overline{\Dgap} \cup \overline{\Dstat},$$
where the open sets $D_{\text{rot}}$, $\Dgap$ and $D_{\text{stat}}$ stand for the rotor, the separating air gap and the stator, respectively. 
The stator $\Dstat$ consists of 
\begin{itemize}
\item A region $D_{\text{stat,f}}$ filled with ferromagnetic material;
\item A region $D_{\text{stat,a}}$ made of air;
\item The coils $D_{\text{stat,c}}$ featuring copper wires.
\end{itemize} 
The part $D_{\text{rot}}$ is composed of:
\begin{itemize}
\item A ferromagnetic core, occupying the region $\Omega$; 
\item Air, occupying the region $\Omega_{\text{a}}$;
\item Permanent magnets represented by the domain $\Dmag$.
\end{itemize}
In this study, the interface $\Gamma := \partial \Omega \cap \partial \Omega_a$ between the phases of the rotor made of ferromagnetic material and air is the center of attention when it comes to optimal design, 
while the region $\Dmag$ occupied by the permanent magnets and the stator $\Dstat$ are not subject to optimization, 
see \cref{fig_machine} for an illustration of this structure.
 
The conductivity and reluctivity coefficients $\sigma$ and $\nu$ have different expressions in the various phases featured in these decompositions. 
In order to emphasize their dependence on the actual shape $\Omega \subset D_{\text{rot}}$ 
of the optimized ferromagnetic core, they are labeled with an $_\Omega$ subscript.
\begin{align} \label{eq_sigmaNu}
    \sigma_\Omega(x) = \begin{cases}
                    \sigma_f & x \in \Omega \cup D_{\text{stat,f}}, \\
                    \sigma_a & x \in \Omega_{\text{a}} \cup \Dgap \cup D_{\text{stat,a}}, \\
                    \sigma_m & x \in \Dmag, \\
                    \sigma_c & x \in D_{\text{stat,c}},
                \end{cases} \:\:\: \text{ and }
    \nu_\Omega(x, s) = \begin{cases}
                    \hat \nu(s) & x \in \Omega \cup D_{\text{stat,f}}, \\
                    \nu_a & x \in \Omega_{\text{a}} \cup \Dgap \cup  D_{\text{stat,a}},, \\
                    \nu_m & x \in \Dmag, \\
                    \nu_c & x \in  D_{\text{stat,c}}.
                \end{cases}
\end{align}
In practice, the value of the electric conductivity $\sigma_a$ in the air region $\Omega_{\text{a}} \cup \Dgap \cup D_{\text{stat,a}}$ is negligible; 
moreover, due to the laminated structure of the part filled with ferromagneric material, which is made of iron sheets, 
and since the copper wires typically used in electric machines are insulated, the values of $\sigma$ are also negligible in $\Omega \cup  D_{\text{stat,f}}$ and $D_{\text{stat,c}}$: 
$$\sigma_f = \sigma_a = \sigma_c =0.$$
The reluctivity $\nu$ takes constant, positive values $\nu_a$, $\nu_m$ and $\nu_c$ in air, in the magnets, and in the wires, respectively.  
However, it is a function $\hat\nu: \R_+ \to \R_+$ of the amplitude of the gradient of the potential $u$ in the regions $\Omega$ and $D_{\text{stat,f}}$ occupied by ferromagnetic material. 
As evidenced in \cite{pechstein2006monotonicity}, the physical properties of the latter imply that
the mapping $s \mapsto \hat \nu(s) s$ in \cref{eq_sigmaNu} is strongly monotone and Lipschitz continuous, i.e.
\begin{equation}\label{eq.hypnu}
\text{There exists } 0< \underline \nu \leq \overline \nu < \infty \:\: \text{ s.t. } \:\: \forall s_1, \: s_2 \in \R_+, \quad  \left\{
\begin{array}{l}
  (\hat \nu(s_1)s_1 - \hat \nu(s_2)s_2)(s_1 - s_2) \geq    \underline \nu (s_1 - s_2)^2 ,\\[0.5em]
   |\hat \nu(s_1)s_1 - \hat \nu(s_2)s_2| \leq \; \overline \nu |s_1 - s_2|.
    \end{array}\right.
\end{equation}
In particular, \cref{eq.hypnu} implies that $\hat\nu$ is uniformly bounded away from $0$ and $\infty$: 
$$ \forall s \in \R_+, \quad \underline\nu \leq \hat\nu(s) \leq \overline \nu.$$

\subsubsection{Description of the motion of the rotor}\label{sec.phasemov}

\noindent The operation of the motor is analyzed through a representative time period $(0,T)$; during the latter,
$\Drot$ moves according to the smooth velocity field $v : (0,T) \times \overline \Drot \to \R^d$. 
The corresponding flow $\varphi : (0,T) \times \Drot \to \R^d$ is the solution to the following ordinary differential equation:
\begin{equation}\label{eq.defvel}
\forall  x \in \Drot,\quad \left\{
\begin{array}{cl}
 \frac{\partial \varphi}{\partial t}(t,x) =  v(t,\varphi(t,x)) & \text{on } (0,T), \\
 \varphi(0,x) = x;
 \end{array}\right.
 \end{equation}
we also let the shorthand $\varphi_t \equiv \varphi(t,\cdot)$. 
It follows from the standard theory of ordinary differential equations that: 
 \begin{equation}\label{eq.hypvphi1}
 \hspace{-1.5cm}
\begin{minipage}{0.9\textwidth}
\begin{itemize}
\item $\varphi$ is smooth over $[0,T] \times \overline\Drot$;
\item For each $t \in [0,T]$, the mapping $x \mapsto \varphi(t,x)$ is a diffeomorphism from $\overline\Drot$ onto itself;
\end{itemize} 
\end{minipage}
 \end{equation}
We also assume that the considered motion of the rotor is $T$-periodic, and so it holds: 
\begin{equation}\label{eq.hypvphi2}
 \forall x \in \Drot, \quad \varphi(0,x) = \varphi(T,x) = x.
 \end{equation}
Eventually, for convenience in the mathematical analysis, we extend $\varphi_t$ to a smooth diffeomorphism of $D$ such that:
 \begin{equation}\label{eq.hypvphi3}
\varphi_t(\Dgap) = \Dgap, \text{ and } \varphi_t \equiv \Id \text{ on } \Dstat.
 \end{equation}
 
This motion of the rotor induces an evolution in time of the constituent material phases of $\Drot$. 
Let $ \Omega(t) := \varphi_t (\Omega)$ and $\Gamma(t):= \varphi_t(\Gamma)$ (resp. $\Dmag(t) = \varphi_t(\Dmag)$, etc.) denote the deformed versions of $\Omega$ and $\Gamma$ (resp. $\Dmag$, etc.) induced by $\varphi_t$.
From the mathematical viewpoint, the material coefficients $\sigma$ and $\nu$ in \cref{eq.uOm} depend on time, as:
\begin{equation}\label{eq.snut}
\forall x \in D, \:\: s \in \R_+,  \quad \sigma_{\Omega(t)} (\varphi_t(x))= \sigma_\Omega(x)  \text{ and } \nu_{\Omega(t)} (\varphi_t(x),s)= \nu_\Omega(x,s).
\end{equation}\par\medskip

In practice, the cross-section $D$ of the motor is a two-dimensional disk with center $0$ and $\varphi_t$ corresponds to the rotation with origin $0$ and angle $\alpha(t) := \frac{2\pi t}{T}$: 
\begin{equation}\label{eq.phirot}
 \varphi_t(x) = R_{\alpha(t)} \left( 
\begin{array}{c}
x_1 \\
x_2
\end{array}
\right), \text{ where } R_{\alpha(t)} :=  \left( 
\begin{array}{cc}
\cos \alpha(t) & -\sin \alpha(t) \\ 
\sin\alpha(t) & \cos \alpha(t) 
\end{array}
\right).
 \end{equation}
With this definition, a simple calculation reveals that the velocity field $v$ is divergence-free and that the Jacobian determinant of $\varphi_t$ identically equals $1$:
\begin{equation}\label{eq.divvzero}
 \forall t \in [0,T], \:\: x \in D, \:\: \dv \:v(t,x) = 0, \text{ and }   \det \nabla\varphi_t(x) = 1.
 \end{equation}
 
\begin{remark}
Even though this very specific setting where the space dimension $d$ equals $2$ and the motion $\varphi(t,x)$ of the rotor accounts for a rotation is the relevant one in practice, 
the general case where $d$ is arbitrary  and $\varphi_t$ is an arbitrary smooth diffeomorphism of $D$ satisfying \cref{eq.hypvphi1,eq.hypvphi2,eq.hypvphi3}, with associated velocity field $v(t,x) = \frac{\partial \varphi}{\partial t}(t,\varphi_t^{-1}(x))$, does not pose any additional difficulty in the analysis. 
For this reason, in the subsequent developments, we retain generic notations for $d$ and $\varphi$ insofar as possible. 
\end{remark} \par\medskip

\subsection{Notations and mathematical formulation of the magneto-quasi-static problem}\label{sec.not}

\noindent 
The following notations are adopted throughout the remainder of this article.
\begin{itemize}
\item We denote by $\Id: \R^d \to \R^d$ the identity mapping of $\R^d$ and by $\I \in \R^{d\times d}$ the identity matrix with size $d$.
\item For any time $t \in [0,T]$, $n_{\Omega(t)}$ (resp. $n$) stands for the unit normal vector to the deformed interface $\Gamma(t)$ (resp. to the interface at rest $\Gamma$), pointing outward $\Omega(t)$ (resp. outward $\Omega$).
\item We denote by 
$$ Q = (0,T) \times D, \text{ and }  Q_\Omega := \Big\{ (t,x) \in Q, \:\: x \in \Omega(t) \Big\}$$
the (open) space-time cylinders induced by $D$ and $\Omega$ over the time period $(0,T)$.
\item Let $\alpha: D \to \R$ be a smooth quantity on $\overline\Omega$ and $D \setminus \Omega$, which is possibly discontinuous across the interface $\Gamma$; we denote by 
$$ \alpha^\pm(x) = \lim\limits_{t > 0 \atop t \to 0} \alpha(x\pm tn(x))$$
 the one-sided limits of $\alpha$ at $x \in \Gamma$, from outside and inside $\Omega$, respectively.
 \item In particular, let $u : D \to \R$ be a smooth function on $\overline \Omega$ and $D \setminus\Omega$, which is possibly discontinuous at the interface $\Gamma$. 
 For any $x \in \Gamma$, we denote by 
 $$ \frac{\partial u^\pm}{\partial n}(x) = \lim\limits_{t \to 0 \atop t > 0} \nabla u (x \pm t n(x)) \cdot n(x) \text{ and } \nabla_\Gamma u^\pm(x) = \nabla u(x \pm tn(x)) -  \frac{\partial u^\pm}{\partial n}(x)  n(x),$$
 the one-sided normal derivatives and tangential gradients of $u$ at $x$.
 \item Let $w : D \to \R^d$ be a smooth enough vector field. We denote by 
 $$  w_\Gamma : \Gamma \to \R^d, \quad w_\Gamma(x) = w(x) - \Big( w(x) \cdot n(x) \Big) n(x)$$
 the tangential part of $w$ on $\Gamma$.
\item  Let $u : Q \to \R$ be a smooth enough scalar function of time and space; then, 
 \begin{itemize}
\item $\nabla u (t,x)$ denotes the gradient of $u$ with respect to the space variable $x$ only; 
\item $ \frac{\partial u}{\partial t}(t,x)$ is the partial derivative of $u$ with respect to the time variable, i.e. the derivative of the partial mapping $t \mapsto u(t,x)$ for fixed $x \in D$; 
\item $\frac{\d u}{\d t}(t,x)$ is the total (or material) derivative of $u$ according to the velocity field $v(t,x)$, that is: 
\begin{equation}\label{eq.totder}
 \frac{\d u}{\d t} (t,x) = \frac{\partial u}{\partial t}(t,x) + v(t,x) \cdot \nabla u (t,x);
 \end{equation}
in other terms, $\frac{\d u}{\d t}(t,x)$ is the derivative of the mapping $t \mapsto u(t,\varphi_t(y))$ evaluated at time $t$ and point $y =\varphi_t^{-1}(x)$.
\end{itemize} 
\item Let $w = (w_1,\ldots,w_d) :Q \to \R^d$ be a (smooth enough) vector-valued function of space and time. 
We denote by $\nabla w$, or often $\left[\nabla w \right]$, the derivative of $w$ with respect to the spatial variable, 
i.e. the $d\times d$ matrix field with entries:
$$ \left[\nabla w  \right]_{ij} = \frac{\partial w_i}{\partial x_j}, \quad i,j=1,\ldots,d.$$
  \end{itemize}

Summarizing the previous discussions and taking these notations into account, the transverse component of the vector potential inside the motor, which we denote by $u_\Omega: Q \to \R$ to emphasize its dependence on the phase $\Omega$ to be optimized, satisfies the following evolution problem:
\begin{equation}\label{eq.uOm}
\left\{
\begin{array}{cl}
\sigma_{\Omega(t)} \frac{\d u_\Omega}{\d t}  - \dv\left( \nu_{\Omega(t)} (x, \lvert \nabla u_\Omega \lvert) \nabla u_\Omega\right) = f & \text{in } (0,T) \times D, \\
u_\Omega(t,x) =0 & \text{for } t \in (0,T), \: x \in \partial D, \\ 
u_\Omega(0, x) = u_\Omega(T, x)& \text{for } x \in D_{\text{mag}},
\end{array}
\right.
\end{equation}
where the coefficients $\sigma_{\Omega(t)}$ and $\nu_{\Omega(t)}$ are given by \cref{eq_sigmaNu}. 
The precise mathematical setting of this problem and its well-posedness are discussed in \cref{sec.mathmag} below. 

\begin{remark}\label{rem.linear}
\noindent \begin{itemize}
\item In passing from the formulation \cref{eq_mqs2d} to \cref{eq.uOm}, we have absorbed the term $\dv(M^\perp)$ at the right-hand side of the problem \cref{eq_mqs2d} into the source $f$. 
This convenient notational simplification is slightly abusive though, as both terms $f$ and $\dv(M^\perp)$ are not exactly of the same mathematical nature: while $f(t,\cdot)$ will typically belong to $L^2(D)$ in our mathematical analysis, 
the quantity $M^\perp(t,\cdot)$ will be smooth on $\Dmag$ and vanish outside this set, so that $\dv(M^\perp)(t,\cdot) \in H^{-1}(D)$. Nevertheless, the subsequent developments are easily adapted to handle this term, see \cref{eq.usedvM} about this point. 
\item The subsequent mathematical discussions rely on the fact that the domain $\Dmag$ is non empty. If it were to be the case, \cref{eq.uOm} would simply boil down to decoupled elliptic problems.
\item As we have mentioned, the non linearity of the boundary value problem \cref{eq.uOm} (notably, the dependence of the reluctivity on $\lvert\nabla u\lvert$) is key in the correct modeling of the physical problem. However, the linear setting, where
$$ \hat\nu(s) \equiv \nu_f \text{ for some } \nu_f > 0,$$
will sometimes be considered in the mathematical setting, as it allows for a less technical exposition of the main ideas.
\end{itemize}
\end{remark}

\section{Generalities about shape optimization and treatment of an academic problem} \label{sec:genshape}

\noindent In this section, we enter into the optimal design framework of this article. 
After introducing the model shape optimization problem under scrutiny in \cref{sec.sopb}, 
we detail in \cref{sec.simplefunc} the calculation of a ``simple'' shape derivative as a handful preliminary to our subsequent analyses. 

\subsection{Shape optimization in a nutshell}\label{sec.sopb}

\noindent In this article, we aim to optimize the shape of the phase $\Omega$ made of ferromagnetic material
with respect to a physical criterion, that is
\begin{equation}\label{eq.sopb}
 \min\limits_{\Omega \subset \Drot} J(\Omega).
\end{equation}
The performance criterion under scrutiny $J(\Omega)$ depends on $\Omega$ via the transverse component of vector potential $u_\Omega(t,x)$ of the magnetic flux density, solution to the evolution problem \cref{eq.uOm}.
We consider a model objective function of the form
\begin{equation}\label{eq.defJOmpde}
 J(\Omega) = \int_Q j(u_\Omega(t,x)) \:\d x \d t,
 \end{equation}
where $j: \R \to \R$ is a smooth function, satisfying the growth conditions: 
$$\exists\: C >0, \:\: \forall u \in \R, \quad \lvert j(u) \lvert \leq C (1+\lvert u \lvert)^2, \:\: \lvert j^\prime(u) \lvert \leq C (1+\lvert u \lvert), \text{ and } \lvert j^{\prime\prime}(u) \lvert \leq C.$$ 
Note that different quantities of interest could be considered without much changes to our analysis.
Notably, $J(\Omega)$ could also depend on the gradient of $u_\Omega$, as in the torque functional $\Tor(u_\Omega)$ defined in \cref{eqn:torque}, see \cref{sec.numex}. 
Besides, constraints (e.g. on the volume of $\Omega$) could be added to the problem \cref{eq.sopb} without much change to the subsequent theoretical developments. \par\medskip

The numerical solution of the problem \cref{eq.sopb} usually calls for the derivative of the objective function $J(\Omega)$ with respect to the 
domain -- a notion which can be defined in various ways. 
Here, we rely on Hadamard's boundary variation method, see e.g. \cite{henrot2018shape,murat1976controle,sokolowski1992introduction} or \cite{allaire2020survey} for a recent overview.
In a nutshell, variations of a shape $\Omega \subset \Drot$ are considered under the form
$$ \Omega_\theta := (\Id + \theta) (\Omega) , \quad \theta \in \Winfty, \:\: || \theta ||_{\Winfty} < 1,$$
that is, $\Omega_\theta$ is a version of $\Omega$ whose points have been displaced according to the ``small' vector field $\theta$, see \cref{fig.hadamard}. 

\begin{figure}[!ht]
\centering
\includegraphics[width=0.5\textwidth]{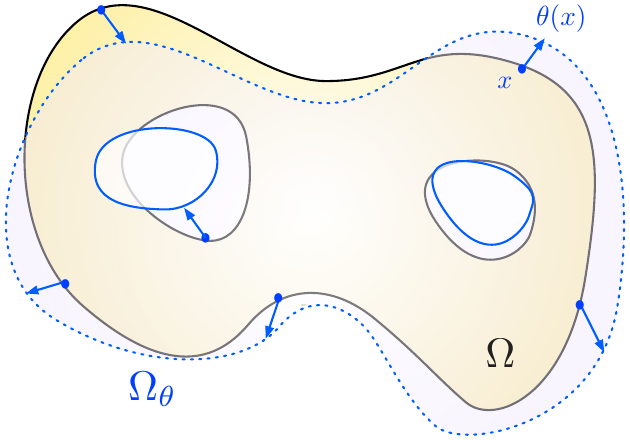}
\caption{\it The variation $\Omega_\theta$ of $\Omega$ featured in the method of Hadamard is obtained by displacing its points according to the ``small'' vector field $\theta$.}
\label{fig.hadamard}
\end{figure}

A function $F(\Omega)$ of the domain is then called shape differentiable at a particular domain $\Omega$ provided the underlying mapping $\theta \mapsto F(\Omega_\theta)$, 
from $\Winfty$ into $\R$, is Fr\'echet differentiable at $\theta=0$. The corresponding shape derivative $F^\prime(\Omega)(\theta)$ gives rise to the following expansion: 
\begin{equation}\label{eq.expFOmt}
 F(\Omega_\theta) = F(\Omega) + F^\prime(\Omega)(\theta) + \o(\theta), \text{ where } \frac{\lvert \o (\theta) \lvert }{\lvert\lvert \theta \lvert\lvert_{\Winfty}} \xrightarrow{\theta \to 0} 0.
 \end{equation}
 In particular, a deformation $\theta \in \Winfty$ such that $F^\prime(\Omega)(\theta) < 0$ is a descent direction for $F$ from $\Omega$, since for a small enough descent step $\tau >0$,
 $$ F(\Omega_{\tau \theta}) = F(\Omega) + \tau F^\prime(\Omega)(\theta) + \o(\tau),$$
 i.e. a ``small'' deformation of $\Omega$ according to $\theta$ produces a design $\Omega_{\tau \theta}$ which is ``better'' with respect to the criterion $F(\Omega)$.
 
In the present application, we only optimize the interface $\Gamma$ between the phases $\Omega$, $\Omega_{\text{a}}$ of the rotor, respectively made of ferromagnetic material and air. 
Accordingly, the deformations $\theta$ involved in Hadamard's method ought to be restricted to the subspace
$$\Tad := \Big\{ \theta \in \Winfty, \quad \theta = 0 \text{ on } \overline{\Dstat}, \:\: \theta \cdot n =0 \text{ on } \partial \Drot \cup \partial \Dmag \Big\}.$$

Let us end this section with a few informal words about the general structure of the shape derivative of a ``smooth enough'' function of the domain $F(\Omega)$, 
such as those considered in the sequel; we refer to e.g. \S 5.9 in \cite{henrot2018shape} for more details about this subject.  
Usually, several equivalent formulas are available for $F^\prime(\Omega)(\theta)$. 
A rigorous calculation usually yields the so-called volume (or distributed) expression of this derivative, which is of the form:
\begin{equation}\label{eq.thstructsd}
 F^\prime(\Omega)(\theta) = \int_D \Big( M_\Omega : \nabla \theta + s_\Omega \cdot \theta \Big) \:\d x,  
 \end{equation}
where $M_\Omega : D \to \R^{d\times d}$ and $s_\Omega : D \to \R^d$ are matrix and vector fields depending on $\Omega$, $u_\Omega$ and the function $F(\Omega)$, 
notably via a so-called adjoint state $p_\Omega$, solution to a problem similar to \cref{eq.uOm}, with a different right-hand side. 
Often, \cref{eq.thstructsd} can be turned into an equivalent, surface expression: 
\begin{equation}\label{eq.surfstruct}
 F^\prime(\Omega)(\theta) = \int_{\Gamma} v_\Omega (\theta \cdot n) \:\d s,
 \end{equation}
featuring a scalar field $v_\Omega :\partial \Omega \to \R$. 
This alternative structure encodes the intuitive fact that $F^\prime(\Omega)(\theta)$ solely depends on the normal component of the deformation $\theta$ -- roughly speaking, vector fields taking only tangential directions on the boundary $\partial \Omega$ only account for a ``reparametrization'' of $\Omega$. The surface form \cref{eq.surfstruct} is also helpful for numerical purpose, as it immediately reveals a descent direction for $F(\Omega)$. 
Indeed, letting $\theta = -v_\Omega n$ on $\partial \Omega$ immediately yields $F^\prime(\Omega)(\theta) < 0$. 

Mathematically, the surface formula \cref{eq.surfstruct} for $F^\prime(\Omega)(\theta)$ is often obtained from its volume counterpart \cref{eq.thstructsd} after integration by parts, 
assuming some regularity from the boundary $\partial \Omega$ and the functions $u_\Omega$, $p_\Omega$, see e.g. the proofs of \cref{lem.sdacad,prop.sdJOmmag} below. 
%
\subsection{Calculation of the shape derivative of a ``simple'' functional of the domain}\label{sec.simplefunc}

\noindent We now present the calculation of the derivative of a ``simple'' functional $J(\Omega)$, depending on the shape $\Omega$
via the integration of a fixed, smooth function $f(t,x)$ on the transformed versions $\Omega(t)= \varphi_t(\Omega)$ of $\Omega$ for $t \in (0,T)$. 
We seize this opportunity to introduce a few technical preliminaries useful throughout this article in \cref{sec.techprel}, before proceeding to the calculation, properly speaking, in \cref{sec.acadsd}.

\subsubsection{Technical preliminaries}\label{sec.techprel}

\noindent 
For any vector field $\theta \in \Tad$, let us introduce the corresponding space-time deformation mapping $\Theta : Q \to (0,T) \times \R^d$ 
\begin{equation}\label{eq.stdef}
 \Theta(t,x) := (\Theta_t(t,x) , \Theta_x(t,x)) =  (t, \varphi_t \circ (\Id + \theta) \circ \varphi^{-1}_t (x) ),
 \end{equation}
whose notation omits its dependence on $\theta$ for simplicity.
The space-time cylinder $Q_{\Omega_\theta}$ induced by $\Omega_\theta$ then reads:
$$ Q_{\Omega_\theta} := \Theta(Q_\Omega) = \Big\{ (t,x) \in Q, \:\: x \in \varphi_t (\Omega_\theta) \Big\}.$$
\begin{remark}
Loosely speaking, the space-time deformation mapping $\Theta$ transmits the effect of $\theta$ on the configuration at rest $\Omega$ (which is deformed into $\Omega_\theta$) to the rotated configurations $\Omega(t) = \varphi_t(\Omega)$ (which become $\varphi_t(\Omega_\theta)$). 
Contrary to the works \cite{BrueggerHarbrecht2022,BrueggerHarbrechtTausch2021}, we do not aim to find a time-dependent shape, but rather optimize the design of its configuration at rest with respect to a measure of performance depending on its known deformations $\varphi_t(\Omega)$.   
The action of $\Theta$ on the space-time cylinder $Q_\Omega$ is exemplified in \cref{fig_Qomega}. 
\end{remark}

\begin{figure}
\centering
 \includegraphics[width=.8\textwidth]{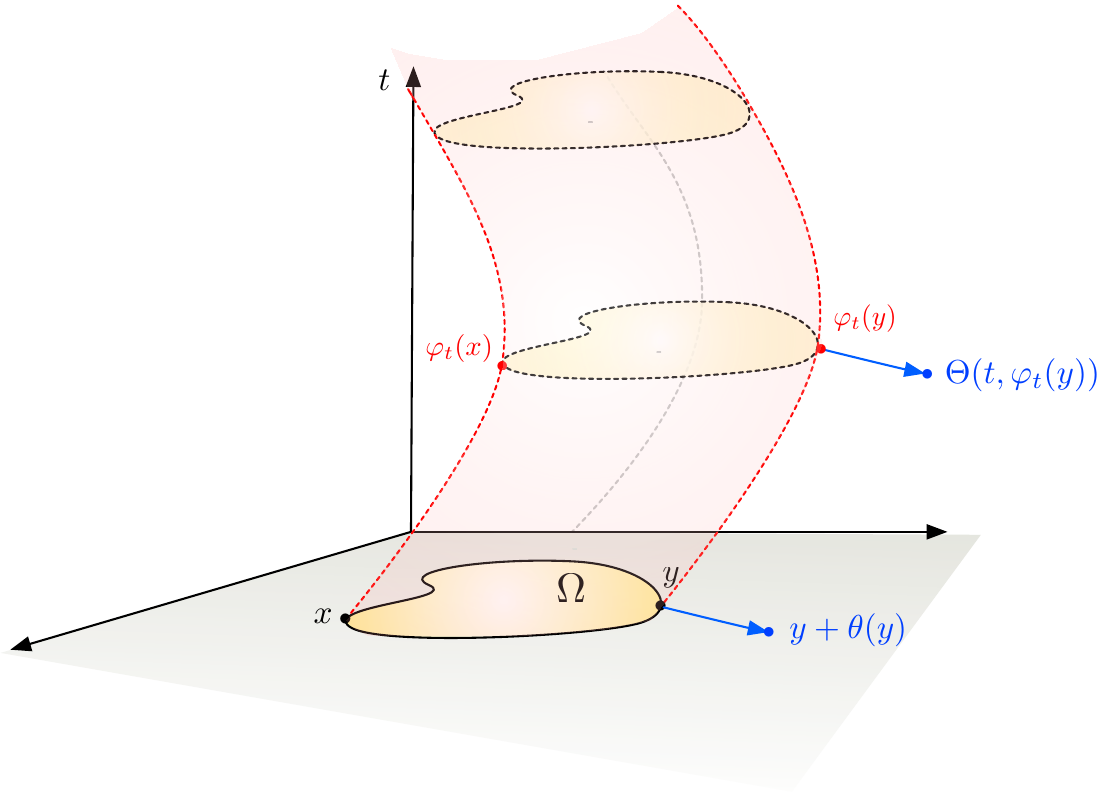}
  \caption{\it One deformation $\theta$ of the shape $\Omega$, with the associated deformation $\Theta$ of the space-time cylinder $Q_\Omega$.}
    \label{fig_Qomega}
\end{figure}

The derivative of $\Theta(t,x)$ with respect to the variables $(t,x)$ is the following $(d+1) \times (d+1)$ matrix field $F(\theta) : Q \to \R^{(d+1) \times (d+1)}$:
$$ F(\theta)= \left( \begin{array}{cc}
F_{tt}(\theta) & F_{tx}(\theta) \\
F_{xt} (\theta) & F_{xx} (\theta)
\end{array}
\right),$$
whose $1 \times 1$ and $1 \times d$ blocks $F_{tt}(\theta)$ and $F_{tx}(\theta)$ read:
$$ F_{tt}(\theta)(t,x) = \frac{\partial \Theta_t}{\partial t}(t,x) = 1, \quad F_{tx}(\theta)(t,x) = \nabla \Theta_t(t,x) = 0,$$
and whose $d \times 1$ and $d \times d$ blocks $F_{xt}(\theta)$ and $F_{xx}(\theta)$ are given by:
$$
\begin{array}{l}
\begin{array}{>{\displaystyle}cc>{\displaystyle}l}
F_{xt}(\theta)(t,x) &=& \frac{\partial \Theta_x}{\partial t}(t,x)\\[0.7em]
& =& \frac{\partial \varphi_t}{\partial t}( (\Id + \theta)(\varphi^{-1}_t(x)) ) + \Big[\nabla \varphi_t ((\Id +\theta)(\varphi^{-1}_t(x) ))\Big] \Big[(\I +\nabla \theta)(\varphi_t^{-1}(x)) \Big] \frac{\partial}{\partial t}(\varphi_t^{-1}(x)),
\end{array} \\[3em]
 \begin{array}{>{\displaystyle}cc>{\displaystyle}l}
F_{xx}(\theta)(t,x) &=& \nabla \Theta_x(t,x) \\[0.5em]
&=& \Big[\nabla \varphi_t ((\Id + \theta)(\varphi_t^{-1}(x)) )\Big] \:\Big[ (\I + \nabla \theta)(\varphi^{-1}_t(x)) \Big] \: \Big[\nabla \varphi_t^{-1}(x)\Big].
\end{array}
\end{array}$$
A simple calculation immediately yields the expression of the matrix field $F(\theta)^{-T}$:
\begin{equation}\label{eq.FthetamT} F(\theta)^{-T} =   \left( \begin{array}{cc}
1 &  b(\theta)^T \\
0 & F_{xx}(\theta)^{-T}
\end{array}
\right), \text{ where }
 b(\theta): Q \to \R^d \text{ is defined by } b(\theta) := - \Big[F_{xx}(\theta)\Big]^{-1} F_{xt}(\theta).
\end{equation}

Let now $u : Q \to \R$ be a smooth enough function; the derivative of the composite function $v := u \circ \Theta : Q \to \R$ satisfies:
$$ \left( 
\begin{array}{c}
\frac{\partial u}{\partial t}\\[0.25em]
\nabla u
\end{array}
\right) \circ \Theta  = F(\theta)^{-T} \left( 
\begin{array}{c}
\frac{\partial v}{\partial t}\\[0.25em]
\nabla v
\end{array}
\right),$$
and so
\begin{equation}\label{eq.chgvNtNx}
 \left(\frac{\partial u}{\partial t }\right) \circ \Theta = \frac{\partial v}{\partial t} + b(\theta) \cdot \nabla v, \: \text{ and } \:
(\nabla u )\circ \Theta = \Big[F_{xx}(\theta)\Big]^{-T} \nabla v.
 \end{equation}
 
We now list a series of useful technical identities in the calculation of shape derivatives.
These can be rigorously established along the lines of Chap. 5 in \cite{henrot2018shape}; for brevity, we limit ourselves to formal statements and computations which in particular omit the functional setting.
\begin{itemize}
\item The vector field $b(0) : Q \to \R^d$ equals: 
$$
 \begin{array}{>{\displaystyle}cc>{\displaystyle}l}
 b(0)(t,x) &=& -F_{xt}(0)(t,x)  \\[0.75em]
 &=& - \frac{\partial \varphi_t}{\partial t} (\varphi_t^{-1}(x)) - \Big[\nabla \varphi_t(\varphi_t^{-1}(x)) \Big] \frac{\partial }{\partial t} (\varphi_t^{-1}(x)) \\[0.75em]
 &=& -\frac{\partial }{\partial t} (\varphi_t(\varphi_t^{-1}(x)))  \\[0.75em]
 &=& 0,
 \end{array}$$
 where we have used the fact that 
 $$F_{xx}(0)(t,x) = \Big[\nabla \varphi_t(\varphi_t^{-1}(x))\Big]  \Big[\nabla \varphi_t^{-1}(x)\Big] = \nabla( \varphi_t \circ \varphi_t^{-1}(x)) = \I.$$
\item It follows easily from the above two relations that the mapping $\theta \mapsto b(\theta)$ is Fr\'echet differentiable at $\theta =0$, with derivative:
\begin{equation}\label{eq.bthetap}
b^\prime(0)(\theta)= -F_{xt}^\prime(0)(\theta).
\end{equation}
\item The derivatives of the mappings $\theta \mapsto F_{xt}(\theta)$ and $\theta \mapsto F_{xx}(\theta)$ at $\theta=0$ read: 
\begin{multline}\label{eq.Fxtt}
 F_{xt}^\prime(0)(\theta)(t,x) = \left[\nabla \left( \frac{\partial \varphi_t}{\partial t}\right)(\varphi_t^{-1}(x)) \right] \theta(\varphi_t^{-1}(x)) +\Big[ \nabla^2\varphi_t (\varphi_t^{-1}(x)) \theta(\varphi_t^{-1}(x)) \Big]\frac{\partial}{\partial t}(\varphi_t^{-1}(x)) \\
 +  \Big[ \nabla \varphi_t (\varphi_t^{-1}(x)) \Big] \Big[\nabla \theta(\varphi_t^{-1}(x)) \Big] \frac{\partial}{\partial t}(\varphi_t^{-1}(x)), 
 \end{multline}
and  
\begin{equation}\label{eq.Fxxt}
 F_{xx}^\prime(0)(\theta)(t,x) = \Big[ \nabla^2 \varphi_t (\varphi_t^{-1}(x)) \theta(\varphi_t^{-1}(x)) \Big] \Big[\nabla \varphi_t^{-1}(x)\Big] + \Big[\nabla \varphi_t(\varphi_t^{-1}(x))\Big] \Big[ \nabla \theta(\varphi_t^{-1}(x)) \Big] \Big[\nabla \varphi_t^{-1}(x)\Big].
 \end{equation}
\item The determinant $m(\theta): Q \to \R$ of the space-time change of variables induced by $\Theta$ in \cref{eq.stdef}, 
\begin{equation}\label{eq.defmtheta}
m(\theta) :=   \left\lvert \det F(\theta) \right\lvert ,
\end{equation} 
is a Fr\'echet differentiable function of $\theta$ at $0$ and:
 \begin{equation}\label{eq.expmtheta}
m^\prime(0)(\theta)(t,x) :=  \tr \left(\Big[ \nabla^2 \varphi_t( \varphi_t^{-1}(x) ) \theta(\varphi_t^{-1}(x))\Big] \Big[\nabla \varphi_t^{-1}(x)\Big] \right) + (\dv \theta)(\varphi_t^{-1}(x)).
  \end{equation}
\item Let $f: Q\to \R$ be a smooth scalar function. 
The mapping $\theta \mapsto f \circ \Theta$ is Fr\'echet differentiable at $\theta=0$, and its derivative $f_1(\theta)$ reads: 
   \begin{equation}\label{eq.fthetap}
f_1(\theta)(t,x) := \left(\Big[\nabla \varphi_t(\varphi_t^{-1}(x)) \Big]^T \nabla f (t,x) \right)\cdot \theta(\varphi_t^{-1}(x)) .
  \end{equation}
\item  Likewise, let $w: Q \to \R^d$ be a smooth vector field; then the mapping $\theta \mapsto w \circ \Theta$ is Fr\'echet differentiable at $\theta=0$, and its derivative $w_1(\theta)$ reads
   \begin{equation}\label{eq.vthetap}
w_1(\theta)(t,x) := \Big[\nabla w(t,x) \Big] \Big[\nabla \varphi_t(\varphi_t^{-1}(x)) \Big] \theta(\varphi_t^{-1}(x)) .
  \end{equation}
\item The matrix-valued mapping $A(\theta) : Q \to \R^{d\times d}$ defined by
\begin{equation}\label{eq.defAtheta} 
A( \theta) := m(\theta) \: F_{xx}^{-1}(\theta) F_{xx}^{-T}(\theta);
\end{equation}
is Fr\'echet differentiable at $\theta =0$, with derivative:
\begin{equation}\label{eq.Athetap}
 A^\prime(0)(\theta) = m^\prime(0)(\theta)\I - F_{xx}^\prime(0)(\theta) - F_{xx}^\prime(0)(\theta)^T.
\end{equation}
\end{itemize}

\begin{remark}
When the motion $\varphi_t$ is induced by a rotation as in \cref{eq.phirot}, some of the above expressions can be conveniently simplified. 
Indeed, it holds in this case:
$$        \frac{\d R_{\alpha(t)}}{\d t} = \alpha^\prime(t) \left( \begin{array}{cc}
                                                       - \sin(\alpha(t)) & - \cos(\alpha(t))\\
                                                       \cos(\alpha(t)) & - \sin(\alpha(t))
                                                      \end{array} \right)
        \qquad \mbox{ and } \qquad
        \Big[\nabla v(t,x) \Big]= \alpha'(t) \left( \begin{array}{cc}
        0 & - 1\\
        1 & 0
        \end{array} \right),
$$
and so:
$$
    \begin{array}{ccl}
        F_{xx}'(0)(\theta) &= & R_{\alpha(t)} \Big[ \nabla \theta(\varphi_t^{-1}(x)) \Big] R_{-\alpha(t)}, \\[0.5em]
        F_{xt}'(0)(\theta) &= & \left(\frac{\d R_{\alpha(t)} }{\d t}\right) \theta(\varphi_t^{-1}(x)) + R_{\alpha(t)} \Big[ \nabla \theta(\varphi_t^{-1}(x)) \Big] \left(\frac{\d R_{-\alpha(t)} }{\d t}\right) ,  
    \end{array}
      $$
 and for any smooth enough vector field $w : Q \to \R^d$, 
 $$
 w_1(\theta) = \Big[\nabla w(t,x) \Big]R_{\alpha(t)} \theta(\varphi_t^{-1}(x)).
 $$
\end{remark}
\subsubsection{An academic shape optimization problem}\label{sec.acadsd}

\noindent In this section, we calculate the derivative of the following shape functional:
\begin{equation}\label{eq.JOmacad}
J(\Omega) = \int_{Q_\Omega} f(t,x) \:\d x \d t,
\end{equation}
depending on the domain $\Omega$ via the integral of a smooth function $f \in \calC^\infty([0,T] \times \R^d)$ over the space-time cylinder $Q_\Omega$, which involves all the deformed version $\Omega(t)$ of $\Omega$, $t\in (0,T)$. 
Again, this quite academic question is a useful preliminary to the treatment of the more ``physical'' shape functionals described in \cref{sec_calcsdmqs}. 

\begin{lemma}\label{lem.sdacad}
The functional $J(\Omega)$ in \cref{eq.JOmacad} is shape differentiable at any bounded Lipschitz shape $\Omega \subset \R^d$, and its shape derivative reads, in volume form: 
\begin{multline}\label{eq.volumeadsd}
J^\prime(\Omega)(\theta) = \int_{Q_\Omega} \tr \left(\Big[ \nabla^2 \varphi_t( \varphi_t^{-1}(x) ) \theta(\varphi_t^{-1}(x))\Big] \Big[\nabla \varphi_t^{-1}(x)\Big]  \right) f(t,x) \:\d x \d t \\
 + \int_{Q_\Omega} (\dv \theta)(\varphi_t^{-1}(x))  f(t,x) \:\d x\d t + \int_{Q_\Omega}   \Big[\nabla \varphi_t(\varphi_t^{-1}(x)) \Big]^T \nabla f (t,x) \cdot \theta(\varphi_t^{-1}(x))  \:\d x \d t,
\end{multline}
or equivalently, in surface form: 
\begin{equation}\label{eq.surfacadsd}
J^\prime(\Omega)(\theta) = \int_{\partial \Omega} v_\Omega  \: \theta \cdot n \:\d s , \text{ where } v_\Omega(x) := \int_0^T |\det \nabla \varphi_t(x)| f(t,\varphi_t(x)) \:\d t. 
\end{equation}
\end{lemma}
\begin{proof}
For any small enough deformation $\theta \in \Tad$, the definition of $J(\Omega_\theta)$ reads:
 $$ J(\Omega_\theta) = \int_{Q_{\Omega_\theta}} f (t,x) \:\d x \d t.$$
 Using a change of variables based on the mapping $\Theta$ in \cref{eq.stdef} and the fact that, by definition $\Theta^{-1}(Q_{\Omega_\theta}) = Q_\Omega$, we obtain:
 $$ J(\Omega_\theta) = \int_{Q_{\Omega}} m(\theta) ( f\circ \Theta) (t,x) \:\d x \d t,$$
 where we recall the notation $m(\theta) = \left\lvert \det F(\theta)\right\lvert$.
Combining the expansions \cref{eq.expmtheta,eq.fthetap} of the two terms in the above integrand, 
we see that the mapping $\theta \mapsto J(\Omega_\theta)$ is differentiable at $\theta=0$, with derivative 
\begin{multline*}
J^\prime(\Omega)(\theta) = \int_{Q_\Omega} \tr \left( \Big[ \nabla^2 \varphi_t( \varphi_t^{-1}(x) ) \theta(\varphi_t^{-1}(x))\Big] \Big[\nabla \varphi_t^{-1}(x)\Big]  \right) f(t,x) \:\d x \d t \\
 + \int_{Q_\Omega} (\dv \theta)(\varphi_t^{-1}(x))  f(t,x) \:\d x\d t + \int_{Q_\Omega}   \Big[\nabla \varphi_t(\varphi_t^{-1}(x)) \Big]^T \nabla f (t,x) \cdot \theta(\varphi_t^{-1}(x))  \:\d x \d t,
\end{multline*}
which is the announced formula \cref{eq.volumeadsd}.

To infer the surface expression \cref{eq.surfacadsd}, we first express \cref{eq.volumeadsd} as nested integrals over $(0,T)$ and $\varphi_t(\Omega)$ and we use a change of variables in terms of the mapping $\varphi_t$ in the spatial integral to obtain
\begin{multline}\label{eq.vformacadchgv}
J^\prime(\Omega)(\theta) = \int_{0}^T \int_\Omega \lvert \det \nabla \varphi_t(x) \lvert \tr \left( \Big[ \nabla^2 \varphi_t( x ) \theta(x)\Big] \Big[\nabla \varphi_t^{-1}(\varphi_t(x))\Big]  \right) f(t, \varphi_t(x)) \:\d x \d t \\
 + \int_{0}^T \int_\Omega  |\det \nabla \varphi_t(x)| (\dv \theta)(x)  f(t,\varphi_t(x)) \:\d x\d t 
  + \int_{0}^T \int_\Omega  |\det \nabla \varphi_t(x)|  \Big[\nabla \varphi_t(x) \Big]^T \nabla f (t,\varphi_t(x)) \cdot \theta(x)  \:\d x \d t.
\end{multline}
We now proceed according to the strategy summarized in e.g. \cite{allaire2020survey}: integrating by parts in all the terms of the above expression containing derivatives of $\theta$ yields an expression of the form:
$$ J^\prime(\Omega)(\theta) = \int_\Omega M_\Omega \cdot \theta \:\d x + \int_\Gamma v_\Omega (\theta \cdot n) \: \d s + \int_\Gamma t_\Omega \cdot \theta_\Gamma \:\d s,$$
where $M_\Omega: D \to \R^d$, $v_\Omega : \Gamma \to \R$ and $t_\Omega : \Gamma \to \R^d$ depend on the function $f$ and on $\Omega$. 
If we expect that $J^\prime(\Omega)(\theta)$ admits an expression with the desirable structure \cref{eq.surfstruct}, the vector fields $M_\Omega$ and $t_\Omega$ must vanish. 
That this indeed holds is eventually verified by an elementary, albeit tedious calculation. 

In the present case, an integration by parts in the second integral on the right-hand side of \cref{eq.vformacadchgv} (which is the only one involving derivatives of $\theta$)
produces the following expression:
$$ J^\prime(\Omega)(\theta) = \int_0^T \int_{\Gamma} \lvert \det \nabla \varphi_t(x) \lvert f(t,\varphi_t(x)) \: \theta(x) \cdot n(x) \:\d s(x) \d t + r(\theta), $$
where $r(\theta)$ is a collection of integrals posed on the domain $\Omega$ involving only $\theta$ (and not its derivatives). 
A simple, albeit tedious calculation which is omitted for brevity reveals that $r(\theta)$ actually vanishes, which allows to conclude.
\end{proof}

\section{Calculation of the shape derivative in the magneto-quasi-static context} \label{sec_calcsdmqs}

\noindent In this section, we turn to the calculation of the shape derivative of a ``physical'' objective function $J(\Omega)$
which depends on the shape $\Omega$ via the potential $u_\Omega$, 
solution to the magneto-quasi-static problem \cref{eq.uOm} where  $\Omega$ represents the phase occupied by ferromagnetic material. 
After sketching some mathematical preliminaries about this evolution problem in \cref{sec.mathmag} and a few technical features about its solution in \cref{sec.propsuOm}, 
we proceed to the calculation of the shape derivative $J^\prime(\Omega)(\theta)$ in \cref{sec.magsd}, under the simplifying assumption that the reluctivity $\nu$ of the ferromagnetic material is constant, in which case \cref{eq.uOm} becomes linear. We eventually provide the formal extension of this result in the general situation of a non linear material in \cref{sec.extder}.

\subsection{Mathematical preliminaries about the magneto quasi-static evolution problem}\label{sec.mathmag}

\noindent
The analysis of the evolution problem \cref{eq.uOm} brings into play suitable functional spaces for functions depending on the time and space variables. For the sake of convenience, 
a few basic facts about these are recalled in \cref{app.vdistrib};  we refer to e.g. \S 1 in Chap. XVIII of \cite{dautray1992evolution}, or Chap. 23 in \cite{zeidler2013linear} for more exhaustive presentations. \par\medskip

Let us consider the functional space $L^2(0,T; H^1_0(D))$, whose dual $L^2(0,T; H^1_0(D))^*$ can be identified with $L^2(0,T;H^{-1}(D))$, see \cref{eq.iddualstspace}.
We also introduce
 $$ W := \left\{ u \in L^2(0,T; H^1_0(D) ) \text{ s.t. } \sigma_{\Omega(t)} \frac{\partial u}{\partial t} \in L^2(0,T; H^{-1}(D)) \right\},$$
which is a Hilbert space when equipped with the norm
 $$ || u ||_{W}^2 := \int_0^T \lvert\lvert u(t,\cdot) \lvert\lvert^2_{H^1_0(D)}\:\d t + \int_0^T \left \lvert \left\lvert  \sigma_{\Omega(t)} \frac{\partial u}{\partial t}(t,\cdot)\right\lvert\right\lvert^2_{H^{-1}(D)}  \:\d t.$$
Intuitively, $W$ is a variant of the more classical space $W(0,T;H^1_0(D),L^2(D))$ defined in \cref{eq.WOTVH}, 
which is adapted to handle the fact that the coefficient $\sigma_{\Omega(t)}$ vanishes outside the region $\Dmag(t) \subset D$, see \cref{eq_sigmaNu}.
The properties of $W(0,T;H^1_0(D), L^2(D))$ summarized in \cref{lem.densW} are also verified by $W$, as the proof of this result can be straightforwardly adapted to the present case. 
In particular, any element $u \in W$ induces a continuous mapping $[0,T] \ni t \mapsto u(t,\cdot) \in L^2(\Dmag)$, which allows to introduce the closed subspace $\Wper$ of $W$ defined by:
 \begin{equation}\label{eq.Xper}
  \Wper := \Big\{ u \in W \text{ s.t. } u(t=0, \cdot) = u(t=T,\cdot)  \text{ on } \Dmag\Big\}.
  \end{equation}
 
 \begin{remark}
\noindent \begin{itemize}
\item As the conductivity $\sigma_{\Omega(t)}$ equals $\sigma_m > 0$ on $\Dmag(t)$ and vanishes outside $\Dmag(t)$,
the time derivative $\frac{\partial}{\partial t}(u(t,\varphi_t(\cdot)))$ of a function $u \in W$ belongs to $L^2(0,T; H^1(\Dmag)^*)$. 
\item The time periodicity condition satisfied by functions $u \in \Wper$ only holds in $\Dmag$, as an application of \cref{lem.densW} (ii). 
\end{itemize}
\end{remark}
 
Using this language, the magneto-quasi-static problem \cref{eq.uOm} can be expressed in variational form:
 \begin{multline}\label{eq.pbuOm}
\text{Search for } u_\Omega \in \Wper, \text{ s.t. } \forall w \in L^2(0,T;H^1_0(D)),\\ 
  \int_Q \sigma_{\Omega(t)} \left( \frac{\partial u_\Omega}{\partial t} + v \cdot \nabla u_\Omega \right) w \: \d x \d t + \int_{Q} \nu_{\Omega(t)}(x,\lvert\nabla u_\Omega \lvert) \nabla u_\Omega \cdot \nabla w \:\d x \d t  = \int_Q f w \:\d x\d t,
 \end{multline}
or equivalently:
\begin{multline}\label{eq.pbuOmae}
\text{Search for } u_\Omega \in \Wper, \text{ s.t. for a.e. } t \in (0,T), \\ 
\forall w \in H^1_0(D), \quad \int_D \sigma_{\Omega(t)} \left( \frac{\partial u_\Omega}{\partial t} + v \cdot \nabla u_\Omega \right) w \: \d x + \int_{D} \nu_{\Omega(t)}(x,\lvert\nabla u_\Omega \lvert) \nabla u_\Omega \cdot \nabla w \:\d x  = \int_D f w \:\d x.  
 \end{multline}

 The following theorem deals with the well-posedness of the variational problem \cref{eq.pbuOm}. 
 \begin{theorem}\label{th.wellposed}
 Let $\varphi_t$ be a time-dependent diffeomorphism of $\overline D$ satisfying \cref{eq.hypvphi1,eq.hypvphi2,eq.hypvphi3}. 
 For any source $f \in L^2(0,T;H^{-1}(D))$, the evolution problem \cref{eq.pbuOm} has a unique solution $u_\Omega \in \Wper$, which has a Lipschitz dependence on $f$: 
 there exists a constant $C >0$ such that the solutions $u_1$, $u_2 \in \Wper$ associated to the respective right-hand sides $f_1$, $f_2 \in L^2(0,T; H^{-1}(D))$ satisfy: 
 $$ \lvert\lvert u_1 - u_2 \lvert\lvert_{W} \leq C \lvert\lvert f_1 - f_2 \lvert\lvert_{L^2(0,T; H^{-1}(D))}.$$
 \end{theorem}
The proof of this result is sketched in \cref{app.varft}. 
 Let us point out that the analysis of a similar statement about a linear version of \cref{eq.pbuOm} was presented recently in \cite{GanglGobrialSteinbach2023}.

 \subsection{A few properties of the solution to the magneto-quasi-static problem}\label{sec.propsuOm}
 
 \noindent 
This section is devoted to a few useful technical facts about functions in the space $W$, and more precisely about the solution $u_\Omega$ to \cref{eq.pbuOm}.
Throughout, $\varphi_t$ is a smooth time-dependent diffeomorphism of $\overline D$ satisfying \cref{eq.hypvphi1,eq.hypvphi2,eq.hypvphi3} 
and the coefficients $\sigma_{\Omega(t)}$ and $\nu_{\Omega(t)}$ are given by \cref{eq_sigmaNu}. \par\medskip

The first remark is an avatar of the so-called Reynolds (or transport) theorem: 
\begin{lemma}\label{lem.reynolds}
Let $u ,p \in W$; the following identity holds in the sense of distributions on $(0,T)$:
\begin{multline*} \frac{\d}{\d t} \left( \int_D \sigma_{\Omega(t)}(x) u(t,x) p(t,x) \:\d x\right) = 
\int_D \sigma_{\Omega(t)}(x) \frac{\d u}{\d t}(t,x) p(t,x) \:\d x \\
+ \int_D \sigma_{\Omega(t)}(x) u(t,x) \frac{\d p}{\d t}(t,x) \:\d x +  \int_Q \sigma_{\Omega(t)} \dv v(t,x) u(t,x) p(t,x) \:\d x \d t,
\end{multline*}
where we recall the notation $\frac{\d u}{\d t}$, $\frac{\d p}{\d t}$ in \cref{eq.totder} for the total time derivatives of $u$ and $p$.
\end{lemma}

\begin{proof}
By \cref{lem.densW} (i), it is enough to prove that the desired identity holds at every $t \in (0,T)$ when $u$ and $p$ belong to the space $\calC^\infty([0,T],H^1_0(D))$.
To achieve this, we first use a change of variables based on the mapping $\varphi_t$:
$$ 
\begin{array}{>{\displaystyle}cc>{\displaystyle}l}
\int_D \sigma_{\Omega(t)}(x) u(t,x) p(t,x) \:\d x &=& \int_{D}  \lvert \det \nabla \varphi_t(x) \lvert \sigma_{\Omega(t)}(\varphi_t(x)) u(t,\varphi_t(x)) p(t,\varphi_t(x)) \:\d x \\[1em]
&=&   \int_{D}  \lvert \det \nabla \varphi_t(x) \lvert \sigma_\Omega(x) \: u(t,\varphi_t(x)) p(t,\varphi_t(x)) \:\d x,
\end{array}
$$
where the second line follows from \cref{eq.snut}.
Taking derivatives with respect to time, we obtain:
\begin{multline*} \frac{\d}{\d t} \left( \int_D \sigma_{\Omega(t)}(x) u(t,x) p(t,x) \:\d x\right) = 
\int_D |\det \nabla \varphi_t(x)| \sigma_{\Omega}(x) \frac{\d u}{\d t}(t,\varphi_t(x)) p(t,\varphi_t(x)) \:\d x \\
+ \int_D  |\det \nabla \varphi_t(x)| \sigma_{\Omega}(x) u(t,\varphi_t(x)) \frac{\d p}{\d t}(t,\varphi_t(x)) \:\d x  \\
+ \int_D |\det \nabla \varphi_t(x)| \tr \left( \nabla \varphi_t(x)^{-1} \nabla \left( \frac{\partial \varphi_t}{\partial t}\right) (x)\right) \sigma_\Omega(x) u(t,\varphi_t(x)) p(t,\varphi_t(x)) \:\d x.
\end{multline*}
Here, we have used the usual formula for the derivative of the determinant mapping.
Another change of variables now yields:
\begin{multline*} \frac{\d}{\d t} \left( \int_D \sigma_{\Omega(t)}(x) u(t,x) p(t,x) \:\d x\right) = 
\int_D \sigma_{\Omega(t)}(x) \frac{\d u}{\d t}(t,x) p(t,x) \:\d x 
+ \int_D  \sigma_{\Omega(t)}(x) u(t,x) \frac{\d p}{\d t}(t,x) \:\d x  \\
+ \int_D  \tr \left( \nabla \varphi_t(\varphi_t^{-1}(x))^{-1} \nabla \left(\frac{\partial \varphi_t}{\partial t}\right)( \varphi_t^{-1}(x))\right) \sigma_{\Omega(t)}(x) u(t,x) p(t,x) \:\d x.
\end{multline*}
The desired identity eventually results from the following elementary calculation and the definition \cref{eq.defvel} of the velocity field $v(t,x)$: 
$$ 
\begin{array}{>{\displaystyle}cc>{\displaystyle}l}
\nabla \left(\frac{\partial \varphi_t}{\partial t}\right)( \varphi_t^{-1}(x))  \nabla \varphi_t(\varphi_t^{-1}(x))^{-1} &=& \nabla \left(\frac{\partial \varphi_t}{\partial t}\right)( \varphi_t^{-1}(x)) \nabla(\varphi_t^{-1}(x)) \\[0.75em]
 &=& \nabla \left( \frac{\partial \varphi_t}{\partial t}(\varphi_t^{-1}(x)) \right)\\[0.75em]
 &=& \nabla v (t,x).
 \end{array}$$
\end{proof}

\begin{corollary}\label{lem.intdudtp}
The following relation holds, for all $u,p \in W$:
\begin{multline*}
 \int_Q \sigma_{\Omega(t)} \frac{\d u}{\d t} p \:\d x \d t =  \int_D \left(\sigma_{\Omega(T)} u(T,x) p(T,x) - \sigma_{\Omega(0)} u(0,x) p(0,x) \right)\:\d x  \\
 - \int_Q \sigma_{\Omega(t)} \frac{\d p}{\d t} u \:\d x \d t  - \int_Q \sigma_{\Omega(t)} (\dv v) u p \:\d x \d t. 
 \end{multline*}
\end{corollary}
\begin{proof}
For a.e. $t \in (0,T)$, \cref{lem.reynolds} states that:
\begin{multline*}
\int_D \sigma_{\Omega(t)}(x) \frac{\d u}{\d t}(t,x) p(t,x) \:\d x =
 \frac{\d}{\d t} \left( \int_D \sigma_{\Omega(t)}(x) u(t,x) p(t,x) \:\d x\right) \\
- \int_D \sigma_{\Omega(t)}(x) u(t,x) \frac{\d p}{\d t}(t,x) \:\d x  
- \int_D \sigma_{\Omega(t)} \dv v (t,x) u(t,x) p(t,x) \:\d x.
\end{multline*}
In particular, the mapping $t \mapsto \int_D \sigma_{\Omega(t)} u(t,x) p(t,x) \:\d x$ is absolutely continuous on $(0,T)$.
The result follows immediately from integration of this identity over the time period $(0,T)$.
\end{proof}

The variational problem \cref{eq.pbuOm} implicitly encompasses jump conditions for $u_\Omega(t,\cdot)$ at the interface $\Gamma(t)$
where the conductivity and the reluctivity are discontinuous; these are the subject of the next lemma:

\begin{lemma}\label{lem.jumpu}
Let the right-hand side $f$ belong to the space $L^2(0,T,L^2(D))$; the corresponding solution $u_\Omega(t,\cdot)$ to \cref{eq.pbuOm} is continuous across $\Gamma(t)$: 
\begin{equation}\label{eq.jumpu1}
\forall t \in (0,T), \:\: x \in \Gamma, \quad u_\Omega^-(t,\varphi_t(x)) = u_\Omega^+(t,\varphi_t(x)), 
\end{equation}
and so are its normal flux and tangential gradient:
\begin{equation}\label{eq.jumpu2}
\forall  t \in (0,T), \:\: x \in \Gamma, \quad \Big( \nu_\Omega \nabla u_\Omega(t,\varphi_t(x))\Big)^- \cdot n_{\Omega(t)}(\varphi_t(x))  = \Big( \nu_\Omega \nabla u_\Omega(t,\varphi_t(x))\Big)^+ \cdot n_{\Omega(t)}(\varphi_t(x)),
\end{equation}
\begin{equation}\label{eq.jumpu3}
\forall t \in (0,T), \:\: x \in \Gamma, \quad \nabla_{\Gamma(t)} u_\Omega^-(t,\varphi_t(x))  =  \nabla_{\Gamma(t)} u_\Omega^+(t,\varphi_t(x)).
\end{equation}
Finally, it holds:
\begin{equation}\label{eq.jumpu4}
\forall t \in (0,T), \:\: x \in \Gamma, \quad \frac{\d u_\Omega^-}{\d t}(t,\varphi_t(x)) = \frac{\d u_\Omega^+}{\d t}(t,\varphi_t(x)). 
\end{equation}
\end{lemma}
\begin{proof}
Since $u_\Omega \in L^2(0,T;H^1_0(D))$, the function $u_\Omega(t,\cdot)$ belongs to $H^1_0(D)$ for a.e. $t \in (0,T)$, and so \cref{eq.jumpu1} holds true. 

Let now $t \in (0,T)$ be given; by considering smooth test functions $w \in \calC^\infty_c(D)$  with compact support inside $\Omega(t_0)$ or $D \setminus \overline{\Omega(t_0)}$ in \cref{eq.pbuOmae}, then integrating by parts, we see that:
\begin{equation}\label{eq.vanvolPDEu}
 \sigma_{\Omega(t)} \frac{\d u_\Omega}{\d t}(t,x) - \dv\left( \nu_{\Omega(t)}(x,\lvert\nabla u_\Omega(t,x) \lvert)\nabla u_\Omega(t,x) \right) = f(t,x)
 \end{equation}
in the sense of distributions in $\Omega(t)$ or $D \setminus \overline{\Omega(t)}$. 
Repeating the calculation with arbitrary smooth test functions $w \in \calC^\infty_c(\overline D)$ and using \cref{eq.vanvolPDEu}, we now obtain: 
 $$\int_{\Gamma(t)}\left( \nu_{\Omega(t)} \frac{\partial  u_\Omega}{\partial n} \right)^- w \:\d s -\int_{\Gamma(t)} \left(\nu_{\Omega(t)} \frac{\partial  u_\Omega}{\partial n} \right)^+ w \:\d s  =0 ,$$
and \cref{eq.jumpu2} follows immediately.

Eventually, \cref{eq.jumpu3} is obtained by taking the tangential gradient in \cref{eq.jumpu1}, and taking time derivatives in \cref{eq.jumpu1} yields the last jump relation \cref{eq.jumpu4}. 
\end{proof}
\begin{remark}
In the following, we denote by $\nabla_{\Gamma(t)} u_\Omega(t,\varphi_t(x))$ and $\nu_{\Omega(t)} \nabla u_\Omega(t,\varphi_t(x))\cdot n_{\Omega(t)}(\varphi_t(x))$ the common values of the one-sided limits of the tangential gradient and normal flux of $u_\Omega(t,\cdot)$ through $\Gamma(t)$.
\end{remark}
%

\subsection{Calculation of the shape derivative of a physical objective function $J(\Omega)$ in the linear magneto-quasi-static setting}\label{sec.magsd}

\noindent In this section, we detail the calculation of the shape derivative of the physical objective function $J(\Omega)$ in \cref{eq.defJOmpde} involving the magneto-quasi-static potential $u_\Omega$,
with a particular emphasis on the arguments which are quite specific to the present physical context. 
To keep technicality at a minimum, our analysis is conducted in the linear case where the reluctivity $\hat\nu$ of the ferromagnetic material is a constant function $\hat \nu \equiv \nu_f$, see \cref{rem.linear}.
The needed adaptations to the general case of a non constant function $\hat\nu$ are elementary, albeit tedious; they are discussed in the next \cref{sec.extder}.
 Throughout the remainder of this section, we rely on the following shorthand in notation: we denote by $\Omega_1 := \Omega$ the optimized phase of the rotor filled by ferromagnetic material, 
and by $\Omega_2 := D \setminus \overline{\Omega}$ its complement in $D$. 
Accordingly, we denote by $\nu_1$, $\sigma_1$ the reluctivity and conductivity inside $\Omega_1$, and by $\nu_2$, $\sigma_2$ the piecewise constant functions matching the repartition \cref{eq_sigmaNu}. 

The evolution problem \cref{eq.uOm} then rewrites: 
\begin{equation}\label{eq.uOmSimpl}
\left\{
\begin{array}{cl}
\sigma_{\Omega(t)} \frac{\d u_\Omega}{\d t}  - \dv\left( \nu_{\Omega(t)} \nabla u_\Omega\right) = f & \text{in } Q, \\
u_\Omega(t,x) =0 & \text{for } t \in (0,T), \: x \in \partial D, \\ 
u_\Omega(0, x) = u_\Omega(T, x)& \text{for } x \in D_{\text{mag}},
\end{array}
\right.
\end{equation}
where 
$$\sigma_{\Omega(t)}(x) = \left\{
\begin{array}{cl}
\sigma_1 & \text{if } x \in \Omega(t), \\ 
\sigma_2 & \text{otherwise},
\end{array}
\right.
\text{ and } 
\nu_{\Omega(t)}(x) = \left\{
\begin{array}{cl}
\nu_1 & \text{if } x \in \Omega(t), \\ 
\nu_2 & \text{otherwise}.
\end{array}
\right.
$$ \par\medskip

\begin{proposition}\label{prop.sdJOmmag}
Let the source term $f: Q \to \R$ be smooth; the functional $J(\Omega)$ in \cref{eq.defJOmpde} is shape differentiable at any bounded and Lipschitz shape $\Omega \subset \Drot$, and its shape derivative reads, under volume form:
\begin{multline}\label{eq.Jpvol}
\forall \theta \in \Tad, \quad J^\prime(\Omega)(\theta)   =  \int_Q  m^\prime(0)(\theta) j(u_\Omega) \:\d x \d t  \\
+  \int_Q \sigma_{\Omega(t)}  \left(  m^\prime(0)(\theta) \frac{\d u_{\Omega}}{\d t}  - \Big[ F_{xx}^{\prime}(0)(\theta) \Big]v \cdot \nabla u_\Omega  + v_1(\theta) \cdot \nabla u_\Omega +  b^\prime(0)(\theta) \cdot \nabla u_\Omega \right)   \: p_\Omega \:\d x \d t  
\\
+ \int_Q \nu_{\Omega(t)} \Big[A^\prime(0)(\theta) \Big]\nabla u_\Omega \cdot \nabla p_\Omega \:\d x \d t
-  \int_Q  \Big( m^\prime(0)(\theta) f + f_1(\theta)\Big) p_\Omega \:\d x \d t,
\end{multline}
where the quantities $A^\prime(0)(\theta)$, $m^\prime(0)(\theta)$, $F_{xx}^\prime(0)(\theta)$ and $b^\prime(0)(\theta)$ are respectively defined by \cref{eq.Athetap,eq.expmtheta,eq.Fxxt,eq.bthetap},
and $f_1(\theta)$ and $v_1(\theta)$ are the derivatives of the mappings $\theta \mapsto f \circ \Theta$ and $\theta \mapsto v \circ \Theta$, see \cref{eq.fthetap,eq.vthetap}. 

The adjoint state $p_\Omega$ is the unique solution in $\Wper$ to the evolution problem: 
\begin{equation} \label{eq.pOm}
\left\{
\begin{array}{cl}
- \sigma_{\Omega(t)} \frac{\d p_\Omega}{\d t}  - \dv\left( \nu_{\Omega(t)} \nabla p_\Omega\right) - \sigma_{\Omega(t)} (\dv v) p_\Omega = -j^\prime(u_\Omega)& \text{in } Q, \\
p_\Omega(t,x) =0 & \text{for } t \in (0,T), \: x \in \partial D, \\ 
p_\Omega(0,x )  = p_\Omega(T,x ) & \text{for } x \in D.
\end{array}
\right.
\end{equation}
It is equivalently characterized as the solution to the following variational problem:
\begin{multline}\label{eq.varfp}
\text{Search for } p_\Omega \in \Wper \text{ s.t. } \forall w \in L^2(0,T;H^1_0(D)),  \\
 \int_Q  \Big( - \sigma_{\Omega(t)} \frac{\d  p_\Omega}{\d t } w + \nu_{\Omega(t)}\nabla p_\Omega \cdot \nabla w - \sigma_{\Omega(t)} (\dv v) p_\Omega w \Big)\:\d x \d t  = -\int_Q j^\prime(u_\Omega) w \:\d x \d t.
\end{multline}

Assuming $\Omega$ to be smooth, the shape derivative $J^\prime(\Omega)(\theta)$ admits the surface form: 
\begin{multline}\label{eq.Jpsurf}
\forall \theta \in \Tad, \quad J^\prime(\Omega)(\theta) = \int_{\Gamma} v_\Omega \: (\theta\cdot n) \:\d s, \text{ where } \\
v_\Omega(x) := -(\sigma_2-\sigma_1) \int_0^T   \lvert \det\nabla \varphi_t(x)\lvert \frac{\d u_{\Omega}}{\d t}(t,\varphi_t(x))  p_\Omega(t,\varphi_t(x)) \: \d t  \\
  - (\nu_2-\nu_1) \int_0^T \lvert\det \nabla \varphi_t(x)\lvert \nabla_{\Gamma(t)} u_\Omega(t,\varphi_t(x)) \cdot \nabla_{\Gamma(t)} p_\Omega (t,\varphi_t(x))  \: \d t \\
   + \left(\frac{1}{\nu_2} - \frac{1}{\nu_1} \right) \int_0^T  \lvert \det \nabla \varphi_t(x)\lvert \Big( \nu_{\Omega(t)} \nabla u_\Omega(t,\varphi_t(x)) \cdot n_{\Omega(t)} (\varphi_t(x))\Big)  \Big( \nu_{\Omega(t)} \nabla p_\Omega(t,\varphi_t(x)) \cdot n_{\Omega(t)}(\varphi_t(x)) \Big) \:  \d t.
\end{multline}
\end{proposition}

\begin{remark}\label{rem.jumpp}
\noindent \begin{itemize}
\item The well-posedness of the (linear) adjoint problem \cref{eq.pOm} follows from similar, although simpler arguments to those involved in the proof of \cref{th.wellposed} about the (non linear) state problem \cref{eq.uOm}.
\item Like $u_\Omega(t,\cdot)$, the adjoint state $p_\Omega(t,\cdot)$ satisfies jump relations at the interface $\Gamma(t)$; a similar analysis to that in the proof of \cref{lem.jumpu} indeed yields:
\begin{equation*}
\forall t \in (0,T), \:\: x \in \Gamma, \quad p_\Omega^-(t,\varphi_t(x)) = p_\Omega^+(t,\varphi_t(x)), 
\end{equation*}
and
\begin{equation*}
\forall t \in (0,T), \:\: x \in \Gamma, \quad (\nu_\Omega \nabla p_\Omega(t,\varphi_t(x)))^- \cdot n_{\Omega(t)}(\varphi_t(x))  = (\nu_\Omega \nabla p_\Omega(t,\varphi_t(x)))^+ \cdot n_{\Omega(t)}(\varphi_t(x)).
\end{equation*}
\end{itemize}
\end{remark}

\begin{remark}\label{eq.usedvM}
\noindent
\begin{itemize}
\item The smoothness assumption about the source $f$ in the statement of \cref{prop.sdJOmmag} can be weakened in various ways. For instance, the statement holds mutatis mutandis if $f$ only belongs to $L^2(0,T;H^{-1}(D))$, but for each time $t \in (0,T)$,
$f(t,\cdot)$ has support in the region $\Dstat$ where deformations $\theta \in \Tad$ vanish.
\item As we have evoked in \cref{rem.linear}, in realistic applications,
the right-hand side $f$ contains a contribution of the form $\dv(L \mathds{1}_{\Dmag(t)})$ where the smooth vector field $L(t,x)$ stands for the $90^\circ$ clockwise rotate $M^\perp$ of the magnetization $M$.
In such case, the last integral in the right-hand side of \cref{eq.Jpvol} should be supplemented with: 
\begin{equation} \label{eq.suppMagn}
-\int_{Q_{\Dmag}} \left( m^\prime(0)(\theta) L  + L_1(\theta) - \Big[F_{xx}^\prime(0)(\theta) \Big] L \right) \cdot \nabla p_\Omega \:\d x \d t,
\end{equation}
where $L_1(\theta)$ is the first-order term in the asymptotic expansion of the vector field $L \circ \Theta$ at $\theta=0$, see \cref{eq.vthetap}. 
\end{itemize}
\end{remark}

\begin{proof}[Proof of \cref{prop.sdJOmmag}]
However technically tedious, the proof follows a quite classical trail, as
exposed for instance in \cite{allaire2020survey,murat1976controle,henrot2018shape}. We proceed in five steps. \par\medskip

\noindent \textit{Step 1: We prove that the mapping $\Omega \mapsto u_\Omega$ has a ``Lagrangian derivative''.}\par\medskip

\noindent More precisely, we prove that the transported potential $\theta \mapsto \overline{u_\Omega}(\theta):= u_{\Omega_\theta} \circ \Theta$ is Fr\'echet differentiable from a neighborhood of $0$ in $\Tad$ into $\Wper$. 
This starts from the variational problem satisfied by the magneto-quasi-static field $u_{\Omega_\theta} \in \Wper$:
\begin{equation}
\forall w \in L^2(0,T;H^1_0(D)), \quad \int_Q \sigma_{\varphi_t(\Omega_\theta)}(x) \frac{\d u_{\Omega_\theta}}{\d t} w \:\d x \d t + \int_Q \nu_{\varphi_t(\Omega_\theta)}(x) \nabla u_{\Omega_\theta} \cdot \nabla w \:\d x \d t = \int_Q f w \:\d x \d t.
\end{equation}
A change of variables via $\Theta$ in the above integrals yields:
 \begin{multline}\label{eq.varfuOmchgvar}
\int_Q (\sigma_{\varphi_t(\Omega_\theta)} \circ \Theta_x)  m(\theta) \left(\left( \frac{\partial u_{\Omega_\theta}}{\partial t}\right) \circ \Theta + (v \circ \Theta) \cdot (\nabla u_{\Omega_\theta} \circ \Theta)  \right) \:( w \circ \Theta) \:\d x \d t \\
 + \int_Q (\nu_{\varphi_t(\Omega_\theta)} \circ \Theta_x)  m(\theta) \Big( \nabla u_{\Omega_\theta} \circ \Theta \Big) \cdot \Big( \nabla w  \circ \Theta \Big) \:\d x \d t = \int_Q  m(\theta) (f \circ \Theta) (w\circ \Theta) \:\d x \d t,
 \end{multline}
 where we recall the expression \cref{eq.defmtheta} of $m(\theta)$. 
 Now, from the definitions \cref{eq.snut} of $\sigma_\Omega$ and $\nu_\Omega$, it holds:
 $$ \sigma_{\varphi_t(\Omega_\theta)} \circ \Theta_x = \sigma_{\Omega(t)}, \text{ and } \nu_{\varphi_t(\Omega_\theta)} \circ \Theta_x = \nu_{\Omega(t)},$$
so that, by using the relations \cref{eq.chgvNtNx} and taking test functions of the form  $w \circ \Theta^{-1}$, $w\in L^2(0,T;H^1_0(D))$, 
we see that the transported function $\overline{u_\Omega}(\theta)$ satisfies the variational problem:
 \begin{multline}\label{eq.varfubar}
\forall w \in L^2(0,T;H^1_0(D)), \quad \int_Q \sigma_{\Omega(t)} m(\theta) \left( \frac{\partial \overline{u_{\Omega}}(\theta)}{\partial t} + b( \theta) \cdot \nabla \overline{u_{\Omega}}(\theta) + \Big[F_{xx}^{-1}(\theta)\Big] (v \circ \Theta) \cdot \nabla \overline{u_\Omega}(\theta)\right)   \: w \:\d x \d t \\
 + \int_Q \nu_{\Omega(t)}  \Big[A(\theta) \Big]  \nabla \overline{u_{\Omega}}(\theta) \cdot   \nabla w  \:\d x \d t = \int_Q  m(\theta) (f \circ \Theta)  w\:\d x \d t,
\end{multline}
where $A(\theta)$ and $b(\theta)$ are respectively defined by \cref{eq.defAtheta} and \cref{eq.FthetamT}. 

This problem can be reformulated as a more abstract equation for $\overline{u_\Omega}(\theta)$:
$$ \text{Search for } u \in \Wper \text{ s.t. } \calF(\theta,u) = 0,$$
where we have defined the mapping $\calF : \Tad \times \Wper \mapsto L^2(0,T;H^{-1}(D))$ by:
\begin{equation}\label{eq.defcalF}
 \calF(\theta,u) = \left[
\begin{array}{>{\displaystyle}r}
w\mapsto  \int_Q \sigma_{\Omega(t)} m(\theta) \left( \frac{\partial u}{\partial t} + b( \theta) \cdot \nabla u + \Big[F_{xx}^{-1}(\theta)\Big] (v \circ \Theta) \cdot \nabla u\right)   \: w \:\d x \d t \\
 + \int_Q \nu_{\Omega(t)} \Big[A(\theta) \Big]  \nabla u \cdot   \nabla w  \:\d x \d t
  - \int_Q  m(\theta) (f \circ \Theta)  w\:\d x \d t
 \end{array}
\right].
\end{equation}
An elementary verification shows that this equation fulfills the assumptions of  the implicit function theorem at the point $(0,u_\Omega)$, see e.g. \cite{lang2012fundamentals} Chap. I, Th. 5.9. 
Hence, the mapping $\theta \mapsto \overline{u_\Omega}(\theta)$ is Fr\'echet differentiable at $\theta = 0$. Its derivative, denoted by $\mathring{u_\Omega}(\theta)$, is the so-called Lagrangian derivative of $u_\Omega$.
\par\medskip

\noindent \textit{Step 2: We characterize the Lagrangian derivative $\mathring{u_\Omega}(\theta)\in \Wper$.}\par\medskip

\noindent Relying on the information that the mapping $\theta \mapsto \overline{u_\Omega}(\theta)$ is differentiable at $\theta=0$, taking derivatives in \cref{eq.varfubar} shows that
the function $ \mathring{u_{\Omega}}(\theta) \in \Wper$ satisfies, for all $ w \in L^2(0,T;H^1_0(D))$:
 \begin{multline}\label{eq.varderlag}
 \int_Q \sigma_{\Omega(t)}  \frac{\d \mathring{u_{\Omega}}(\theta)}{\d t}    \: w \:\d x \d t 
 + \int_Q \nu_{\Omega(t)}  \nabla \mathring{u_{\Omega}}(\theta) \cdot   \nabla w  \:\d x \d t   = 
  -  \int_Q \sigma_{\Omega(t)}  m^\prime(0)(\theta) \frac{\d u_{\Omega}}{\d t}    \: w \:\d x \d t  \\
 +\int_Q \sigma_{\Omega(t)} \Big[F_{xx}^{\prime}(0)(\theta)\Big] v \cdot \nabla u_\Omega  w\:\d x \d t - \int_Q \sigma_{\Omega(t)} v_1(\theta) \cdot \nabla u_\Omega w\:\d x \d t 
 -  \int_Q \sigma_{\Omega(t)}  b^\prime(0)(\theta) \cdot \nabla u_\Omega  w  \:\d x \d t  \\
 - \int_Q \nu_{\Omega(t)}\Big[A^\prime(0)(\theta)\Big] \nabla u_\Omega \cdot \nabla w \:\d x \d t
+  \int_Q  \Big( m^\prime(0)(\theta) f + f_1(\theta)\Big) w \:\d x \d t.
\end{multline}
This is a well-posed evolution problem for $ \mathring{u_{\Omega}}(\theta) \in \Wper$.\par\medskip

\noindent \textit{Step 3: We calculate the derivative of the functional $J(\Omega)$ in \cref{eq.defJOmpde}.}\par\medskip

\noindent A change of variables based on the mapping $\Theta$ in the definition of $J(\Omega_\theta)$ yields the following expression:
$$ J(\Omega_\theta) = \int_Q m(\theta) j(\overline{u_{\Omega}}(\theta)) \:\d x \d t,$$
which naturally brings into play the transported function $\overline{u_\Omega}(\theta)$.
As we have proved the differentiability of the latter with respect to the deformation $\theta$, we now obtain, by taking derivatives in the above expression:
\begin{equation}\label{eq.Jpudot}
 J^\prime(\Omega)(\theta)= \int_Q \Big( m^\prime(0)(\theta) j(u_\Omega) + j^\prime(u_\Omega) \mathring{u_\Omega}(\theta) \Big)\:\d x \d t.  
 \end{equation}
 \par\medskip
 
 \noindent \textit{Step 4: We apply the adjoint method to derive the volume form of $J^\prime(\Omega)(\theta)$.}\par\medskip
 
\noindent The expression \cref{eq.Jpudot} is not suitable for practical calculations, as it depends in a complicated way on $\theta$, via the Lagrangian derivative $\mathring{u_\Omega}(\theta)$, 
which is only known implicitly, as the solution to \cref{eq.varderlag}. In particular, \cref{eq.Jpudot} does not allow for the simple identification of a descent direction for $J(\Omega)$.

To overcome this classical issue, we rely on the adjoint method from optimal control theory, see e.g. \cite{lions1971optimal} or \cite{plessix2006review} for an intuitive presentation. 
We introduce the adjoint state $p_\Omega \in \Wper$ as the solution to the variational problem \cref{eq.varfp}, and take $w = \mathring{u_\Omega}(\theta)$ as test function in there to obtain: 
$$-\int_Q j^\prime(u_\Omega) \mathring{u_\Omega}(\theta) \:\d x\d t = \int_Q \left( -\sigma_{\Omega(t)} \frac{\d p_\Omega}{\d t} \mathring{u_\Omega}(\theta) + \nu_{\Omega(t)} \nabla p_\Omega \cdot \nabla \mathring{u_\Omega}(\theta) - \sigma_{\Omega(t)} (\dv v) p_\Omega \mathring{u_\Omega}(\theta) \right)\:\d x \d t.$$ 
Using \cref{lem.intdudtp} and the fact that $\mathring{u_\Omega}(\theta) \in \Wper$ together with $\sigma_{\Omega(T)} = \sigma_{\Omega(0)}$ due to \cref{eq.hypvphi2,eq.hypvphi3}, it follows:
$$\int_Q \Big(\sigma_{\Omega(t)} \frac{\d \mathring{u_\Omega}(\theta)}{\d t } p_\Omega + \nu_{\Omega(t)}\nabla p_\Omega \cdot \nabla \mathring{u_\Omega}(\theta) \Big)\:\d x \d t  = -\int_Q j^\prime(u_\Omega) \mathring{u_\Omega}(\theta) \:\d x \d t.$$
Now introducing the variational characterization \cref{eq.varderlag} of $\mathring{u_\Omega}(\theta)$, we obtain: 
\begin{equation}\label{eq.refjuringu}
\begin{array}{>{\displaystyle}cc>{\displaystyle}l}
\int_Q j^\prime(u_\Omega)  \mathring{u_\Omega}(\theta) \:\d x \d t &=&    \int_Q \sigma_{\Omega(t)}  m^\prime(0)(\theta) \frac{\d u_{\Omega}}{\d t}    \: p_\Omega \:\d x \d t  
 -\int_Q \sigma_{\Omega(t)} \Big[ F_{xx}^{\prime}(0)(\theta) \Big] v \cdot \nabla u_\Omega \: p_\Omega \:\d x \d t \\[1em]
&& + \int_Q \sigma_{\Omega(t)} v_1(\theta) \cdot \nabla u_\Omega \: p_\Omega \:\d x \d t 
+  \int_Q \sigma_{\Omega(t)}  b^\prime(0)(\theta) \cdot \nabla u_\Omega \:  p_\Omega  \:\d x \d t  \\[1em]
&& + \int_Q \nu_{\Omega(t)}\Big[A^\prime(0)(\theta)\Big] \nabla u_\Omega \cdot \nabla p_\Omega \:\d x \d t
-  \int_Q  \Big( m^\prime(0)(\theta) f + f_1(\theta)\Big) p_\Omega \:\d x \d t.
\end{array}
\end{equation}
Injecting this expression into \cref{eq.Jpudot}, we obtain the volume form \cref{eq.Jpvol} of the shape derivative of $J(\Omega)$.\par\medskip

 \noindent \textit{Step 5: We infer the surface form \cref{eq.Jpsurf} of $J^\prime(\Omega)(\theta)$ from the volume form \cref{eq.Jpvol}.}\par\medskip
 
\noindent Let us write \cref{eq.Jpvol} under the form
$$J^\prime(\Omega)(\theta)  = \sum\limits_{i=1}^4 I_i(\theta), $$
where $I_i(\theta)$, $i=1,\ldots,4$ are the four integral quantities at the right-hand side of \cref{eq.Jpvol}. 
Our derivation of the surface form of $J^\prime(\Omega)(\theta)$ follows the methodology exposed in \cite{allaire2020survey}, which we have already exemplified in the proof of \cref{lem.sdacad}: 
we integrate by parts all the terms featured in the $I_i(\theta)$
which involve derivatives of $\theta$, and among the resulting terms, we only retain those featuring boundary integrals on $\Gamma$ of the normal component $\theta\cdot n$, 
since the other contributions are expected to cancel according to the prediction of the structure theorem for shape derivatives.

Let us treat the first integral $I_1(\theta)$;  injecting the expression \cref{eq.expmtheta} of $m^\prime(0)(\theta)$, we see that:
$$ I_1(\theta) =  \int_Q  (\dv \theta)(\varphi_t^{-1}(x)) j(u_\Omega(t,x)) \:\d x \d t +r(\theta).$$
Here and throughout the proof, $r(\theta)$ is a remainder that may change from one line to the next, of the form
$$ r(\theta) = \int_D s_\Omega \cdot \theta \:\d x + \int_\Gamma t_\Omega \cdot \theta_\Gamma \:\d s,$$
for some vector fields $s_\Omega: D \to \R^d$, $t_\Omega: \Gamma \to \R^d$ 
that depend on $u_\Omega$, $p_\Omega$ and the motion $\varphi_t$. 
Decomposing the integral in the above right-hand side as a nested integral over $(0,T)$ and $D$, 
and using a change of variables by $\varphi_t$ in the innermost integral, we obtain:
$$ \begin{array}{>{\displaystyle}cc>{\displaystyle}l}
I_1(\theta) &=& \int_0^T \int_{D} |\det \nabla\varphi_t (x)| \dv \theta(x)  j(u_\Omega(t,\varphi_t(x))) \:\d x \d t +r(\theta)\\
&=& \int_0^T \int_{\partial D} |\det \nabla\varphi_t(x) |  j(u_\Omega(t,\varphi_t(x)))  \:(\theta\cdot n)(x)\:\d x \d t + r(\theta),
\end{array}$$
where we have used an integration by parts from the first line to the second one, together with the continuity of $u_\Omega(t,\cdot)$ across $\Gamma(t)$, see \cref{lem.jumpu}. Since $\theta \in \Tad$ vanishes on $\partial D$, it follows that: 
$$ I_1(\theta) = r(\theta).$$

Let us now turn to the second integral $I_2(\theta)$, which rewrites:
$$ 
\begin{array}{>{\displaystyle}cc>{\displaystyle}r}
 I_2(\theta) &=&    \int_Q \sigma_{\Omega(t)}  \left(  m^\prime(0)(\theta) \frac{\d u_{\Omega}}{\d t}  -\Big[F_{xx}^{\prime}(0)(\theta) \Big]v \cdot \nabla u_\Omega  + v_1(\theta) \cdot \nabla u_\Omega  - F_{xt}^\prime(0)(\theta) \cdot \nabla u_\Omega \right)   \: p_\Omega \:\d x \d t  \\[1em]
 &=& \int_0^T \int_D   \sigma_{\Omega(t)} (x) \Big( (\dv\theta)(\varphi_t^{-1}(x)) \frac{\d u_{\Omega}}{\d t}(t,x) -  \Big[\nabla \varphi_t(\varphi_t^{-1}(x))\Big] \Big[ \nabla \theta(\varphi_t^{-1}(x)) \Big] \Big[\nabla \varphi_t^{-1}(x)\Big] v(t,x) \cdot \nabla u_\Omega (t,x) \\ [1em]
&& -  \Big[\nabla \varphi_t(\varphi_t^{-1}(x))\Big] \Big[ \nabla \theta(\varphi_t^{-1}(x)) \Big] \frac{\partial}{\partial t}(\varphi_t^{-1}(x)) \cdot \nabla u_\Omega  \Big) p_\Omega(t,x) \:\d x \d t  + r(\theta)  ,
 \end{array}
$$
where we have used the derivatives \cref{eq.expmtheta,eq.Fxtt,eq.Fxxt}.
Again, the change of variables $x \mapsto \varphi_t(x)$ in the innermost integral yields: 
\begin{multline*}
I_2(\theta) = 
 \int_0^T \int_D   \sigma_{\Omega}(x) \lvert \det\nabla \varphi_t(x) \lvert \Big( (\dv\theta)(x) \frac{\d u_{\Omega}}{\d t}(t,\varphi_t(x)) -  \Big[\nabla \varphi_t(x)\Big] \Big[ \nabla \theta(x) \Big] \Big[\nabla \varphi_t^{-1}(\varphi_t(x))\Big] v(t,\varphi_t(x)) \cdot \nabla u_\Omega (t,\varphi_t(x))\\ 
 -  \Big[\nabla \varphi_t(x)\Big] \Big[ \nabla \theta(x) \Big] \frac{\partial \varphi_t^{-1}}{\partial t}(\varphi_t(x)) \cdot \nabla u_\Omega(t,\varphi_t(x))  \Big) p_\Omega(t,\varphi_t(x)) \:\d x \d t
+r(\theta),
\end{multline*}
and so 
\begin{multline*}
I_2(\theta) = 
 \int_0^T \int_D   \sigma_{\Omega} (x) |\det\nabla \varphi_t(x)| \Big( (\dv\theta)(x) \frac{\d u_{\Omega}}{\d t}(t,\varphi_t(x)) -  \Big[ \nabla \theta(x) \Big] \Big[\nabla \varphi_t^{-1}(\varphi_t(x))\Big] v(t,\varphi_t(x)) \cdot \Big[\nabla \varphi_t(x)\Big] ^T \nabla u_\Omega (t,\varphi_t(x))\\ 
 - \Big[ \nabla \theta(x) \Big] \frac{\partial \varphi_t^{-1}}{\partial t}(\varphi_t(x)) \cdot \Big[\nabla \varphi_t(x)\Big] ^T\nabla u_\Omega(t,\varphi_t(x))  \Big) p_\Omega(t,\varphi_t(x)) \:\d x \d t
+r(\theta).
\end{multline*}
Decomposing the innermost integral as the sum of its contributions on $\Omega$ and $D\setminus \overline\Omega$, 
then integrating by parts with \cref{lem.ippNtheta}, we obtain:
\begin{multline*}
I_2(\theta) = 
-(\sigma_2-\sigma_1) \int_0^T \int_{\Gamma}   |\det\nabla \varphi_t(x)| \frac{\d u_{\Omega}}{\d t}(t,\varphi_t(x))  p_\Omega(t,\varphi_t(x)) (\theta\cdot n)(x) \:\d s \d t \\
+ \int_0^T \int_{\Gamma}  |\det\nabla \varphi_t(x)| \Big(\Big[\nabla \varphi_t^{-1}(\varphi_t(x)) \Big] v(t,\varphi_t(x)) \cdot n(x) + \frac{\partial \varphi_t^{-1}}{\partial t}(\varphi_t(x)) \cdot n(x) \Big) \\ \Big(\sigma_2 \nabla u_\Omega^+(t,\varphi_t(x)) -\sigma_1  \nabla u_\Omega^-(t,\varphi_t(x)) \Big) \cdot \Big(\Big[\nabla \varphi_t(x)\Big]\theta(x) \Big)     \: p_\Omega(t,\varphi_t(x)) \:\d s \d t  
+r(\theta).
\end{multline*} 
Eventually, remarking that the definition \cref{eq.defvel} of the velocity $v(t,x)$ implies that 
 $$ \Big[\nabla \varphi_t^{-1}(\varphi_t(x)) \Big] v(t,\varphi_t(x)) + \frac{\partial \varphi_t^{-1}}{\partial t}(\varphi_t(x)) = \frac{\partial }{\partial t} \left( \varphi_t^{-1}(\varphi_t(x))\right) = 0, $$
we obtain: 
 \begin{equation*}
I_2(\theta) = 
-(\sigma_2-\sigma_1) \int_0^T \int_{\Gamma}   |\det\nabla \varphi_t(x)| \frac{\d u_{\Omega}}{\d t}(t,\varphi_t(x))  p_\Omega(t,\varphi_t(x)) (\theta\cdot n)(x) \:\d s \d t 
+r(\theta).
\end{equation*}

We next turn to the third integral $I_3(\theta)$ on the right-hand side of \cref{eq.Jpvol}. At first, injecting the derivatives \cref{eq.Fxxt,eq.Athetap} into the defining formula, we have:
$$\begin{array}{>{\displaystyle}cc>{\displaystyle}l}
 I_3(\theta)&=&  \int_Q \nu_{\Omega(t)} \Big[A^\prime(0)(\theta)\Big] \nabla u_\Omega \cdot \nabla p_\Omega \:\d x \d t  \\[1em]
 &=& \int_0^T \int_D  \nu_{\Omega(t)}\Big(m'(0)(\theta)\I - \Big[F_{xx}^\prime(0)(\theta)\Big] - \Big[F_{xx}^\prime(0)(\theta)\Big]^T \Big)  \nabla u_\Omega \cdot \nabla p_\Omega \:\d x \d t \\[1em]
&=&  \int_0^T \int_D  \nu_{\Omega(t)}(x)\Bigg( (\dv \theta)(\varphi_t^{-1}(x))  \I 
   -  \Big[\nabla \varphi_t(\varphi_t^{-1}(x))\Big] \Big[ \nabla \theta(\varphi_t^{-1}(x)) \Big] \Big[\nabla \varphi_t^{-1}(x)\Big] \\ 
  && -  \Big[\nabla \varphi_t^{-1}(x)\Big] ^T \Big[ \nabla \theta(\varphi_t^{-1}(x)) \Big] ^T \Big[\nabla \varphi_t(\varphi_t^{-1}(x))\Big] ^T \Bigg)  \nabla u_\Omega(t,x) \cdot \nabla p_\Omega(t,x) \:\d x \d t  + r(\theta).
    \end{array}$$
Hence, a change of variables based on the diffeomorphism $\varphi_t$ in the innermost integral produces: 
\begin{multline*}
I_3(\theta)  =  \int_0^T \int_D  \nu_{\Omega}(x) \lvert \det \nabla \varphi_t(x) \lvert \Bigg( (\dv \theta)(x)  \I 
   -  \Big[\nabla \varphi_t(x)\Big] \Big[ \nabla \theta(x) \Big] \Big[\nabla \varphi_t^{-1}(\varphi_t(x))\Big] \\ 
   -  \Big[\nabla \varphi_t^{-1}(\varphi_t(x))\Big] ^T \Big[ \nabla \theta(x) \Big] ^T \Big[\nabla \varphi_t(x)\Big] ^T \Bigg)  \nabla u_\Omega(t,\varphi_t(x)) \cdot \nabla p_\Omega (t,\varphi_t(x)) \:\d x \d t  + r(\theta), 
\end{multline*}
and so 
\begin{multline*} 
I_3(\theta)
  =  \int_0^T \int_D  \nu_{\Omega}(x) \lvert \det \nabla \varphi_t (x) \lvert   (\dv \theta)(x)   \nabla u_\Omega(t,\varphi_t(x)) \cdot \nabla p_\Omega (t,\varphi_t(x)) \:\d x \d t \\
  - \int_0^T \int_D  \nu_{\Omega}(x) \lvert \det \nabla \varphi_t(x) \lvert  \Big[ \nabla \theta(x) \Big] \Big[\nabla \varphi_t^{-1}(\varphi_t(x))\Big]   \nabla u_\Omega(t,\varphi_t(x)) \cdot  \Big[\nabla \varphi_t(x)\Big]^T \nabla p_\Omega (t,\varphi_t(x))  \:\d x \d t \\
  -\int_0^T \int_D  \nu_{\Omega}(x) \lvert \det \nabla \varphi_t(x) \lvert  \Big[ \nabla \theta(x) \Big] \Big[\nabla \varphi_t^{-1}(\varphi_t(x))\Big]   \nabla p_\Omega(t,\varphi_t(x)) \cdot  \Big[\nabla \varphi_t(x)\Big]^T \nabla u_\Omega (t,\varphi_t(x))  \:\d x \d t 
   + r(\theta).
\end{multline*}
Decomposing each integral over $D$ as the sum of two integrals over $\Omega$ and $D \setminus \overline \Omega$, 
then using the integration by parts formulas in \cref{lem.ippNtheta}, we obtain:
\begin{multline*}
I_3(\theta)
  =  - \int_0^T \int_{\Gamma} \lvert \det \nabla \varphi_t(x)\lvert \Big( \nu_2 \nabla u_\Omega^+(t,\varphi_t(x)) \cdot \nabla p_\Omega^+ (t,\varphi_t(x)) - \nu_1  \nabla u_\Omega^-(t,\varphi_t(x)) \cdot \nabla p_\Omega^- (t,\varphi_t(x))\Big)   (\theta\cdot n)(x)\:\d s(x) \d t \\
  + \int_0^T \int_{\Gamma}  \lvert \det \nabla \varphi_t(x) \lvert   \Bigg(\nu_2  \Big( \Big[\nabla \varphi_t^{-1}(\varphi_t(x))\Big]   \nabla u_\Omega^+(t,\varphi_t(x)) \cdot n(x) \Big)  \Big( \Big[\nabla \varphi_t(x)\Big]^T \nabla p_\Omega^+ (t,\varphi_t(x)) \cdot \theta(x) \Big)     \\ 
  - \nu_1  \Big( \Big[\nabla \varphi_t^{-1}(\varphi_t(x))\Big]   \nabla u_\Omega^-(t,\varphi_t(x)) \cdot n(x) \Big)  \Big( \Big[\nabla \varphi_t(x)\Big]^T \nabla p_\Omega^- (t,\varphi_t(x)) \cdot \theta(x) \Big)  \Bigg) \:\d s (x) \d t \\
 + \int_0^T \int_{\Gamma}  \lvert \det \nabla \varphi_t(x) \lvert  \Bigg( \nu_2  \Big( \Big[\nabla \varphi_t^{-1}(\varphi_t(x))\Big]   \nabla p_\Omega^+(t,\varphi_t(x)) \cdot n (x) \Big)  \Big( \Big[\nabla \varphi_t(x)\Big]^T \nabla u_\Omega^+ (t,\varphi_t(x)) \cdot \theta(x) \Big)     \\ 
  - \nu_1  \Big( \Big[\nabla \varphi_t^{-1}(\varphi_t(x))\Big]   \nabla p_\Omega^-(t,\varphi_t(x)) \cdot n(x) \Big)  \Big( \Big[\nabla \varphi_t(x)\Big]^T \nabla u_\Omega^- (t,\varphi_t(x)) \cdot \theta(x) \Big)  \Bigg) \:\d s (x) \d t    + r(\theta), 
\end{multline*}
and we now proceed to simplify each of the three integrals in the above right-hand side, which we denote by $J_1(\theta)$, $J_2(\theta)$ and $J_3(\theta)$, respectively. 

At first, we introduce the tangential and normal components of the gradients of $u_\Omega(t,\cdot)$ and $p_\Omega(t,\cdot)$; using the jump relations from \cref{lem.jumpu,rem.jumpp}, $J_1(\theta)$ rewrites:
\begin{multline*}
J_1(\theta) = 
  - (\nu_2 - \nu_1) \int_0^T \int_{\Gamma} \lvert \det \nabla \varphi_t(x) \lvert \nabla_{\Gamma(t)} u_\Omega(t,\varphi_t(x)) \cdot \nabla_{\Gamma(t)} p_\Omega (t,\varphi_t(x))    (\theta\cdot n)(x)\:\d s (x) \d t \\
   - \left(\frac{1}{\nu_2} - \frac{1}{\nu_1} \right) \int_0^T \int_{\Gamma} \lvert \det \nabla \varphi_t (x) \lvert \Big( \nu_{\Omega(t)} \nabla u_\Omega(t,\varphi_t(x)) \cdot n_{\Omega(t)}(\varphi_t(x)) \Big)  \Big( \nu_{\Omega(t)} \nabla p_\Omega(t,\varphi_t(x)) \cdot n_{\Omega(t)}(\varphi_t(x)) \Big)   (\theta\cdot n)(x)\:\d s (x)\d t.
\end{multline*}

Let us next deal with $J_2(\theta)$; to this end, we observe that
$$  \lvert \det \nabla\varphi_t (x) \lvert \nu_2  \Big( \Big[\nabla \varphi_t^{-1}(\varphi_t(x))\Big]   \nabla u_\Omega^+(t,\varphi_t(x)) \cdot n(x) \Big) = \nu_2 \Big(  \nabla u_\Omega^+(t,\varphi_t(x)) \cdot \left( \Big[\com(\nabla \varphi_t(x))\Big] n(x) \right) \Big),$$
where $\com(M)$ is the cofactor matrix of a $d \times d$ matrix. 
Hence, using \cref{lem.nortan} and the jump relations from \cref{lem.jumpu}, it follows that
\begin{equation}
\begin{array}{>{\displaystyle}cc>{\displaystyle}l}
 \lvert \det \nabla\varphi_t (x) \lvert \nu_2 \Big( \Big[\nabla \varphi_t^{-1}(\varphi_t(x))\Big]   \nabla u_\Omega^+(t,\varphi_t(x)) \cdot n(x) \Big) &=&   \lvert \det \nabla\varphi_t (x)\lvert \nu_1  \Big( \Big[\nabla \varphi_t^{-1}(\varphi_t(x))\Big]   \nabla u_\Omega^-(t,\varphi_t(x)) \cdot n(x) \Big) \\
 &=& \lvert \com(\nabla\varphi_t((x))) n(x) \lvert \Big( \nu_{\Omega(t)} \nabla u_\Omega \cdot n_{\Omega(t)} (t,\varphi_t(x)) \Big),
\end{array}
\end{equation}
and so: 
\begin{multline}
J_2(\theta)  = \\
    \int_0^T \int_{\Gamma}    \lvert \com(\nabla \varphi_t(x)) n(x) \lvert   \Big(  \nu_\Omega \nabla u_\Omega \cdot n_{\Omega(t)} (t,\varphi_t(x)) \Big)  \Big( \Big[\nabla \varphi_t(x)\Big]^T \nabla p_\Omega^+ (t,\varphi_t(x)) \cdot \theta (x) -  \Big[\nabla \varphi_t(x)\Big]^T \nabla p_\Omega^- (t,\varphi_t(x)) \cdot \theta (x)\Big)  \:\d s(x) \d t.
 \end{multline}
 We now bring into play the jump relations satisfied by $p_\Omega$, summarized in \cref{rem.jumpp} to simplify the second parenthesis in the above integrand: 
 \begin{multline*}
   \Big[\nabla \varphi_t(x)\Big]^T \nabla p_\Omega^+ (t,\varphi_t(x)) \cdot \theta (x) -  \Big[\nabla \varphi_t(x)\Big]^T \nabla p_\Omega^- (t,\varphi_t(x)) \cdot \theta(x) \\
 \begin{array}{>{\displaystyle}cc>{\displaystyle}l}
 &=&
   \Big( \nabla p_\Omega^+ (t,\varphi_t(x)) - \nabla p_\Omega^- (t,\varphi_t(x)) \Big) \cdot \Big(  \Big[\nabla \varphi_t(x)\Big]  \theta(x) \Big) \\[0.5em]
    &=& \left(\frac{1}{\nu_2} - \frac{1}{\nu_1}\right)   \Big( \nu_{\Omega(t)} \nabla p_\Omega(t,\varphi_t(x)) \cdot n_{\Omega(t)}(\varphi_t(x)) \Big)  \Big(  \Big[\nabla \varphi_t(x)\Big]  \theta(x)  \cdot  n_{\Omega(t)}(\varphi_t(x)) \Big) \\[1em]
        &=& \left(\frac{1}{\nu_2} - \frac{1}{\nu_1}\right)  \frac{\lvert \det \nabla \varphi_t  (x) \lvert }{ \lvert \com(\nabla \varphi_t(x)) n(x) \lvert} \Big( \nu_{\Omega(t)} \nabla p_\Omega(t,\varphi_t(x)) \cdot n_{\Omega(t)}(\varphi_t(x)) \Big) ( \theta\cdot n)(x).
  \end{array}
 \end{multline*}
 Hence, $J_2(\theta)$ equals: 
 \begin{equation*}
J_2(\theta)  = 
  \left(\frac{1}{\nu_2} - \frac{1}{\nu_1}\right)    \int_0^T \int_{\Gamma}   \lvert \det \nabla \varphi_t  (x) \lvert  \Big(  \nu_\Omega \nabla u_\Omega \cdot n_{\Omega(t)} (t,\varphi_t(x)) \Big)  \Big(  \nu_\Omega \nabla p_\Omega \cdot n_{\Omega(t)} (t,\varphi_t(x)) \Big) (\theta\cdot n)(x)  \:\d s(x) \d t.
 \end{equation*}
 
Applying a similar treatment to $J_3(\theta)$, it follows: 
\begin{multline*}
J_3(\theta) =
  - (\nu_2 -\nu_1) \int_0^T \int_{\Gamma} \lvert \det \nabla \varphi_t (x) \lvert  \nabla_{\Gamma(t)} u_\Omega(t,\varphi_t(x)) \cdot \nabla_{\Gamma(t)} p_\Omega (t,\varphi_t(x))    (\theta\cdot n)(x)\:\d s(x)  \d t \\
   + \left(\frac{1}{\nu_2} - \frac{1}{\nu_1} \right) \int_0^T \int_{\Gamma} \lvert \det \nabla \varphi_t (x) \lvert \Big( \nu_\Omega \nabla u_\Omega(t,\varphi_t(x)) \cdot n_{\Omega(t)}(\varphi_t(x)) \Big)  \Big( \nu_\Omega \nabla p_\Omega(t,\varphi_t(x)) \cdot n_{\Omega(t)}(\varphi_t(x)) \Big)   (\theta\cdot n)(x)\:\d s (x)\d t \\
   + r(\theta).
\end{multline*}

Eventually, the fourth and last integral $I_4(\theta)$ in the right-hand side of \cref{eq.refjuringu} equals:
$$I_4(\theta)=  \int_Q  \Big( m^\prime(0)(\theta) f + f_1(\theta)\Big) p_\Omega \:\d x \d t= \int_0^T \int_D  (\dv \theta)(\varphi_t^{-1}(x))  f (t,x) p_\Omega(t,x) \:\d x \d t  + r(\theta).$$
Since $p_\Omega(t,\cdot)$ is continuous across $\Gamma(t)$ (see again \cref{rem.jumpp}), it follows like in the treatment of $I_1(\theta)$ that
\begin{equation*}\label{eq.resI4}
 I_4(\theta) = r(\theta).
 \end{equation*}
 
Eventually, the desired surface expression \cref{eq.Jpsurf} of the shape derivative $J^\prime(\Omega)(\theta)$ is obtained by gathering the above expressions of the various integrals $I_i(\theta)$, $i=1,\ldots,4$ and $J_j(\theta)$, $j=1,\ldots,3$, 
and by checking that the integrals collected in the various remainders $r(\theta)$ do cancel, which is elementary, but tedious. 
This concludes the proof of the theorem.
\end{proof}

\subsection{Extensions of these results to the non linear setting} \label{sec.extder}

\noindent In this section, we outline how the shape differentiability of the functional $J(\Omega)$ in \cref{eq.defJOmpde}, proved in \cref{sec.magsd}, extends to the general situation 
where the function $\hat \nu : \mathbb R_+ \rightarrow \mathbb R_+$ characterizing the reluctivity of the ferromagnetic material is no longer a constant,
in which case the state equation \cref{eq.uOm} for $u_\Omega$ becomes a non linear evolution problem.

We content ourselves with the statement of the volume form of the shape derivative of $J(\Omega)$, 
as it is the only information used in our numerical implementation, see \cref{{sec_evasd}}. 
This formula results from a formal calculation similar to that involved in the proof of \cref{prop.sdJOmmag}, see also \cite{GanglLangerLaurainMeftahiSturm2015} for a related analysis.

\begin{proposition} \label{prop.sdJOmmagnl}
    The functional $J(\Omega)$ in \cref{eq.defJOmpde}, featuring the solution $u_\Omega$ to the non linear version of the evolution problem \cref{eq.uOm} with a reluctivity coefficient $\hat\nu$ as in \cref{eq.hypnu}, is shape differentiable at any bounded, Lipschitz shape $\Omega \subset \Drot$ and its shape derivative reads:
\begin{multline}\label{eq.Jpvolnl}
 \forall \theta \in \Tad, \quad J^\prime(\Omega)(\theta)   =  \int_Q m^\prime(0)(\theta) j (u_\Omega) \:\d x \d t \\
 +  \int_Q \sigma_{\Omega(t)}  \left(  m^\prime(0)(\theta) \frac{\mbox \d u_\Omega}{\mbox \d t}  - \Big[F_{xx}^{\prime}(0)(\theta)\Big] v \cdot \nabla u_\Omega  + v_1(\theta) \cdot \nabla u_\Omega + b^\prime(0)(\theta) \cdot \nabla u_\Omega \right)   \: p_\Omega \; \mbox \d x \mbox \d t \\
  + \int_Q \left( \nu_{\Omega(t)}(x,|\nabla u_\Omega|) \Big[A^\prime(0)(\theta)\Big] - \frac{\nu_{\Omega(t)}^\prime(x,|\nabla u_\Omega|)}{|\nabla u_\Omega|} \Big(\Big[F_{xx}^\prime(0)(\theta)\Big]^T \nabla u_\Omega \cdot \nabla u_\Omega \Big) \: \I \right) \nabla u_\Omega \cdot \nabla p_\Omega \; \mbox \d x \mbox \d t\\
- \int_{Q}  \Big( m^\prime(0)(\theta) f + f_1(\theta) \Big) p_\Omega \: \d x \mbox \d t.
\end{multline}
Here, $\nu^\prime_{\Omega(t)}(x,s)$ stands for the derivative of the partial mapping $s \mapsto \nu_{\Omega(t)}(x,s)$, and
 $p_\Omega$ is the solution to the adjoint evolution problem:
\begin{multline}\label{eq.varfpnl}
\text{Search for } p_\Omega \in \Wper \text{ s.t. } \forall w \in L^2(0,T;H^1_0(D)), \\
\int_Q  \left( - \sigma_{\Omega(t)} \frac{\d  p_\Omega}{\d t } w - \sigma_{\Omega(t)} (\dv v) p_\Omega w 
 + \left(\nu_{\Omega(t)}(x,|\nabla u_\Omega|) \I + \frac{\nu_{\Omega(t)}'(x,|\nabla u_\Omega|)}{|\nabla u_\Omega|} \nabla u_\Omega \otimes \nabla u_\Omega\right) \nabla p_\Omega \cdot \nabla w \right)\d x \d t  =\\
  -\int_Q j^\prime(u_\Omega) w \:\d x \d t.
\end{multline}
 The quantities $b^\prime(0)(\theta)$, $F_{xx}^\prime(0)(\theta)$, $f_1(\theta)$, $v_1(\theta)$ and $A^\prime(0)(\theta)$ are respectively given by \cref{eq.Athetap,eq.fthetap,eq.bthetap,eq.Fxxt,eq.vthetap}.
\end{proposition}

Again, note that the formula in \cref{eq.Jpvolnl} is supplemented by the terms \cref{eq.suppMagn} in the presence of permanent magnetization, see \cref{eq.usedvM}.
\begin{remark}
As usual in the practice of the adjoint method, the defining problem for $p_\Omega$ is (the transpose of) the linearized version of the state problem \cref{eq.uOm} for $u_\Omega$. 
In particular, the well-posedness of \cref{eq.varfpnl} is not an immediate consequence of that of \cref{eq.uOm}. In the present context, this property results from
similar arguments to those in the proof of \cref{th.wellposed}, once we have observed that the matrix field 
$$ A(x) :=  \left( \hat \nu(|\nabla u_\Omega|) \I+ \frac{\hat\nu^\prime(|\nabla u_\Omega|)}{|\nabla u_\Omega|} \nabla u_\Omega \otimes \nabla u_\Omega\right) $$
induces a linear and uniformly elliptic operator in $\Dmag$. To see this, we first note that, 
by taking $s_2 = s$, $s_1 = s+h$ in \cref{eq.hypnu} for $s > 0$ and sufficiently small $h \in \R$, then dividing both sides by $h^2$ and letting $h\to 0$, we obtain:
\begin{equation}\label{eq.minnuprime}
 s \hat\nu^\prime(s) + \hat\nu(s) \geq \underline \nu.
 \end{equation}
Now, for any unit vector $\xi \in \R^d$, the operator attached to $A(x)$ satisfies:
$$ A(x) \xi \cdot \xi = \hat \nu(|\nabla u_\Omega|) \lvert \xi \lvert^2+ \frac{\hat \nu'(|\nabla u_\Omega|)}{|\nabla u_\Omega|} \lvert \nabla u_\Omega \cdot \xi \lvert^2.$$
Then, if $\hat\nu^\prime(|\nabla u_\Omega|) \geq 0$, we obtain immediately that
$$ A(x)\xi \cdot \xi  \geq \underline \nu \lvert \xi \lvert^2.$$ 
If on the contrary $\hat \nu^\prime(|\nabla u_\Omega|) \leq 0$ the Cauchy-Schwarz inequality together with \cref{eq.minnuprime} yield:
$$\begin{array}{>{\displaystyle}cc>{\displaystyle}l}
A(x)\xi \cdot \xi &\geq &  \left(\hat\nu(|\nabla u_\Omega|) + \frac{\hat\nu^\prime(|\nabla u_\Omega|)}{|\nabla u_\Omega|} \lvert \nabla u_\Omega \lvert^2 \right)  \lvert \xi \lvert^2, \\
&\geq& \underline \nu  \lvert \xi \lvert^2, 
 \end{array} $$
 which is the expected conclusion.
\end{remark}

\section{Numerical methods}\label{sec.num}

\noindent The shape optimization problems of interest in this work are of the form
\begin{equation}\label{eq.minJnum}
 \min\limits_{\Omega \subset \Drot} J(\Omega),
 \end{equation}
where $J$ depends on the 2d ferromagnetic phase $\Omega$ within (the cross-section of) the rotor $\Drot$
via the solution $u_\Omega$ to the (linear or non linear) magneto-quasi-static problem \cref{eq.uOm}. 

In this section, we describe the main numerical methods employed in our shape optimization framework. 
After sketching the overall strategy in \cref{sec.numstrat}, 
we outline in \cref{sec_resstad} the practical resolution of the state and adjoint evolution problems \cref{eq.uOm} and \cref{eq.pOm}.
We eventually discuss in \cref{sec_evasd} how the formulas provided by \cref{prop.sdJOmmag,prop.sdJOmmagnl} are used in practice to identify 
a suitable descent direction for $J(\Omega)$. 

\subsection{Sketch of the shape optimization workflow} \label{sec.numstrat}

\noindent In a nutshell, we solve \cref{eq.minJnum} thanks to a gradient descent algorithm, 
which improves the shape $\Omega$ by successive deformations along descent directions for the objective function $J(\Omega)$. 
The numerical discretization is based on a space-time finite element framework: the 3d space-time cylinder $Q = (0,T) \times D$ is consistently equipped with a tetrahedral mesh $\calK$ which (approximately) resolves the moving phases $\Omega(t)$, $\Omega_{\text{a}}(t)$, $\Dmag(t)$ etc. for all times $t \in [0,T]$. In our experiments, this space-time mesh is constructed by defining the moving space-time geometry, e.g., using the Open Cascade Technology (OCCT) geometry kernel\footnote{\href{https://dev.opencascade.org/doc/refman/html/index.html}{https://dev.opencascade.org/doc/refman/html/index.html}}, and creating an unstructured tetrahedral mesh for this geometry using the mesh generator Netgen/NGSolve\footnote{\href{https://ngsolve.org/}{https://ngsolve.org/}}. 

The global strategy is outlined in \cref{algo.strat}. In there, and throughout the following, 
we denote with an $^n$ superscript the instances of the different objects at stake (the shape $\Omega$, the descent direction $\theta$, ...) at each iteration $n=0, \ldots$ of the process. 

\begin{algorithm}[ht]
\caption{Shape gradient algorithm for the solution of \cref{eq.minJnum}.}\label{algo.strat}
\begin{algorithmic}[0]
\STATE \textbf{Initialization:} 3d tetrahedral mesh $\calK^0$ of $Q$ where the moving phases $\varphi_t(\Omega^0)$, $\varphi_t(\Omega^0_a)$, etc. are discretized for $t \in (0,T)$. 
\FOR{$n=0,...,$ until convergence}
\STATE \begin{enumerate}
\item Solve the state \cref{eq.uOm} and adjoint \cref{eq.varfpnl} equations for $u_{\Omega^n}$, $p_{\Omega^n} : Q \to \R$.
\item Evaluate the volume form of the shape derivative $J^\prime(\Omega^n)(\theta)$ \cref{eq.Jpvolnl}.
\item Identify a descent direction $\theta^n: D \to \R^d$ for $J(\Omega)$ from $\Omega^n$ using the Hilbertian procedure \cref{eq.HilbertTrick}.
\item Select a suitably small descent step $\tau^n$ and calculate the space-time deformation $\Theta^n$ associated to $\tau^n \theta^n$ via \cref{eq.stdef}. 
\item Modify the space-time mesh $\calK^n$ according to $\Theta^n$ to obtain the new mesh $\calK^{n+1}$. 
\end{enumerate}
\ENDFOR
\RETURN Mesh $\calK^n$ of $Q$, where the deformations $\Omega^n(t^k)$ of the optimized 2d shape $\Omega^n$ are explicitly discretized.
\end{algorithmic}
\end{algorithm}

Briefly, the algorithm starts with the datum of a tetrahedral mesh of the 3d space-time cylinder where the configurations of the various moving or fixed phases (the ferromagnetic phase $\varphi_t(\Omega^0)$, the air phase $\varphi_t(\Omega_{\text{a}}^0)$, etc.) are approximately discretized for $t\in[0,T]$. 

At each iteration $n=0,\ldots$ of the process, the evolution problems for $u_{\Omega^n}$ and $p_{\Omega^n}$ are solved on the mesh $\calK^n$ of the space-time cylinder via the space-time finite element method. 
This task is described in \cref{sec_resstad}. This allows to evaluate the volume form of the shape derivative $J^\prime(\Omega^n)$ and to extract a suitable descent direction $\theta^n$ from the latter via the so-called Hilbertian procedure. 
This operation is detailed in \cref{sec_evasd}. Then, the space-time deformation mapping $\Theta^n$ induced by $\tau^n \theta^n$ is calculated, where $\tau^n$ is a ``small'' descent step, 
and it is used to deform the mesh $\calK^n$ of $Q$ into a new mesh $\calK^{n+1}$ where the updated shape $\Omega^{n+1} = (\Id + \tau^n \theta^n)(\Omega^n)$ (and the other constituent phases of the motor) is explicitly discretized.

\subsection{Resolution of the state and adjoint equations} \label{sec_resstad}

\noindent At any iteration $n$ of the execution of \cref{algo.strat}, 
the space-time cylinder $Q = (0,T)\times D$ is equipped with a tetrahedral mesh $\calK^n$ in which the deformed structures of the motor (in particular, the ferromagnetic phase $\Omega^n(t)$, the air phase $\Omega^n_{\text{a}}(t)$, etc.) for $t \in [0,T]$ are explicitly discretized. 
Hence, the motion of the structure of the motor induced by $\varphi_t$ is directly encoded in the space-time mesh.

This mesh serves as the support for the numerical solution of the state and adjoint equations \cref{eq.uOm,eq.varfpnl} for $u_{\Omega^n}$ and $p_{\Omega^n}$ via a space-time finite element method. 
These are discretized on the mesh $\calK^n$ by means of continuous, piecewise linear finite elements as suggested in \cite{steinbach2015}, 
see also \cite{GanglGobrialSteinbach2022} for an application to the parallel simulation of an electric machine. We mention that also higher order polynomial degrees could be used, see e.g. \cite{LangerSchafelner2022} for a related investigation in an optimal control setting. 
The non linearity in the magneto-quasi-static problem \cref{eq.uOm} for the state $u_{\Omega^n}$ is handled thanks to a Newton-Raphson method with damping in order to ensure a decrease of the residual. \black
The time-periodic conditions can be treated in a straightforward way by identifying the degrees of freedom lying on the bottom and top sections of the mesh $\calK^n$ of the space-time cylinder $Q$ provided that the bottom and top surfaces of the space-time cylinder feature an identical mesh. \black
 
Our numerical experiments rely on the finite element environment \texttt{NGSolve} \cite{Schoeberl2014}. 
Note that the precise setting at hand does not allow to take advantage of the capabilities of this package to compute automatically shape derivatives in the context of more ``standard'' shape optimization problems \cite{GanglSturmNeunteufelSchoeberl2020}.

\begin{remark}
\noindent \begin{itemize}
\item As we have mentioned in the introduction, in principle, space-time finite element approaches allow for parallelization of the numerical resolution of an evolution problem in space-time by means of classical domain decomposition approaches, see e.g. \cite{GanglGobrialSteinbach2022}. In the same spirit, adaptive mesh refinement in space-time is possible \cite{LangerSchafelner2022}.
Last but not least, the state and adjoint equations could actually be gathered into a single system and be solved in parallel, allowing for another level of parallelization when compared to more conventional time-stepping schemes.
While we do solve the associated systems of linear equations in parallel, as discussed in \cref{sec.primde}, adaptivity and the combination of state and adjoint are not exploited in the present study, but they are interesting leads for future work. 
\item The present use of a space-time finite element method is possible since the physical problem under scrutiny arises in two space dimensions, 
thus resulting in a 3d space-time mesh. This strategy could be extended to the case of three space dimensions, provided numerical method for meshing $(3+1)$-dimensional space-time domains and solving related variational problems be available \cite{LangerSchafelner2022}.
\end{itemize}
\end{remark}

\subsection{Evaluation of the shape derivative $J^\prime(\Omega)(\theta)$ and calculation of a descent direction} \label{sec_evasd}

\noindent At any iteration $n=0,\ldots$ of \cref{algo.strat}, the reference to which is omitted in this section for notational simplicity, 
we rely on the volume form of the shape derivative $J^\prime(\Omega)(\theta)$ to identify a descent direction for $J(\Omega)$.
As discussed in e.g. \cite{allaire2020survey,giacomini2017volumetric}, the use of the volume form is a priori less convenient to achieve this purpose than that of the surface form. 
Nevertheless, it presents several advantages, such as the less stringent regularity demanded from $\Omega$, $u_\Omega$, $p_\Omega$, a better consistency between the theoretical and discrete shape optimization settings \cite{hiptmair2015comparison}, and the possibility to conduct part of the evaluation in an automated way \cite{GanglSturmNeunteufelSchoeberl2020}.
 
To exploit the volume form of $J^\prime(\Omega)(\theta)$, we rely on the so-called ``Hilbertian method'', for which we refer to e.g. \cite{azegami1996domain,burger2003framework,de2006velocity}.
Let $V$ be a Hilbert space of vector fields on $D$, equipped with an inner product $b(\cdot,\cdot)$, which is continuously embedded in $\Tad$. 
We solve the following auxiliary variational problem for the shape gradient $\theta$ of the derivative $J^\prime(\Omega)$ associated to the Hilbert space $(V, b(\cdot,\cdot))$: 
\begin{equation}\label{eq.HilbertTrick}
 \text{Search for } \theta \in V \text{ s.t. for all } \eta \in V, \quad b(\theta, \eta) = J^\prime(\Omega)(\eta), 
 \end{equation}
If $\theta$ is the solution to \cref{eq.hilbert_trick}, then $- \theta$ is a descent direction for $J(\Omega)$, since
$$J'(\Omega)(-\theta) =-b(\theta, \theta) < 0.$$ 
Moreover, $\theta$ benefits from the features encoded in the space $V$, such as for instance a higher regularity than mere vector fields in $\Tad$, etc. 

Several choices are possible about $V$ and $b$, starting from Sobolev spaces $H^m(D)^d$, for $m$ large enough. 
One efficient and popular choice, albeit formal, is the following:
\begin{align} \label{eq.hilbert_trick}
V = H^1_0(D)^d, \text{ with the inner product }    b(\theta, \eta) := \int_D \Big( \alpha(x) (\nabla \theta + \nabla \theta ^\top) : \nabla \eta +  \beta(x) \theta \cdot \eta \Big) \:\d x,
\end{align}
where $\alpha, \beta: D \to \R$ are smooth, positive and non negative weights, respectively.
Intuitively, this choice, inspired by the linear elasticity system, produces shape gradient which ``resemble'' elastic displacements, and thereby cause as little compression as possible in the mesh. 
In the same spirit, one may also enrich $b(\cdot,\cdot)$ with a term attached to the Cauchy-Riemann equations, as a means to urge the solution $\theta$ to \cref{eq.HilbertTrick} to be (approximately) conformal, thus preserving (approximately) the angles of the mesh, see \cite{IglesiasSturmWechsung2018}.

The numerical solution of the identification problem \cref{eq.HilbertTrick} raises the need to calculate the volume form of $J^\prime(\Omega)(\theta)$ 
-- by this, we mean, for a given shape $\Omega$, to calculate the quantity $J^\prime(\Omega)(\theta)$ for a wide range of deformation fields $\theta$. 
To achieve this, we note that the shape derivatives \cref{eq.volumeadsd,eq.Jpvol,eq.Jpvolnl} of the functionals considered in this article can be written under the form
$$ J'(\Omega)(\theta) = \int_0^T \int_{\varphi_t(D)} \Big( \mathcal S_0(t,x) \cdot \theta (\varphi_t^{-1}(x)) + \mathcal S_1(t,x) : \nabla \theta ( \varphi_t^{-1}(x) )  \Big)\: \d x \:\d t,  $$
for some vector and matrix fields ${\mathcal S}_0 : Q \to \R^d$ and ${\mathcal S}_1 : Q \to \R^{d \times d}$,
 see \cite{LaurainSturm2016} for a general result in this spirit.
Hence, a change of variables yields:
\begin{align*}
    J'(\Omega)(\theta)  =& \int_0^T \int_{D} \lvert \mbox{det}\nabla \varphi_t(x) \lvert \Big(  \mathcal S_0(t , \varphi_t(x)) \cdot \theta(x) + \mathcal S_1(t,  \varphi_t(x)) : \nabla \theta(x) \Big) \; \mbox dx \: \mbox dt \\
    =&  \int_{D}  g_0(x) \cdot \theta(x) + g_1(x) : \nabla \theta(x)   \; \mbox{d}x,
\end{align*}
with 
\begin{equation}\label{eq.g0g1}
g_0(x) = \int_0^T \mathcal S_0(t, \varphi_t(x)) |\mbox{det}\nabla \varphi_t(x)|\; \d t , \text{ and } g_1(x) = \int_0^T \mathcal S_1(t, \varphi_t(x)) |\mbox{det}\nabla \varphi_t(x)| \; \d t.
\end{equation}
Hence, in practice, when it comes to evaluate $J^\prime(\Omega)(\theta)$, we first calculate the fields $g_0(x)$ and $g_1(x)$ on a triangular mesh $\calT$ of the spatial domain $D$. 
For instance, $\calT$ can be a mesh of the bottom surface of the space-time cylinder $Q$ enclosed in the total 3d mesh $\calK$. 
We use a piecewise constant approximation of the functions $g_i$ ($i=0,1$) on $\calT$: the constant value of $g_i$ on each triangle $T \in \calT$ is obtained by evaluating the ``vertical line integral'' featured in \cref{eq.g0g1} at the centroid of $T$; 
this is realized, in turn, by traveling the space-time mesh $\calK$ of $Q$ 
and using a composite trapezoidal rule. 

Eventually, once a descent direction $\theta$ is obtained for $J(\Omega)$ and a suitable descent step $\tau >0$ is selected, the corresponding space-time deformation mapping
$$\Theta(t,x) = (t, \varphi_t \circ (\Id  +\tau \theta) \circ \varphi_t^{-1}(x))$$
is used to update the space-time mesh $\calK$, see again \cref{fig_Qomega}.

\section{Numerical examples}\label{sec.numex}

\noindent In this section we discuss two numerical examples to illustrate the shape derivative formulas of \cref{prop.sdJOmmag,prop.sdJOmmagnl}
and the implementation of our numerical strategy dedicated to shape optimization problems involving moving domains. 
The first \cref{sec.1dex} deals with a rather academic situation where the one-dimensional optimized shape is subjected to a motion which is slightly more complicated than a rotation. 
We then turn in the \cref{sec.exrot} to the original motivation for this work, namely the optimization of the 2d structure of a rotating electric machine.

\subsection{An academic, spatially one-dimensional example}\label{sec.1dex}

\noindent 
This first example arises in one space dimension ($d=1$). 
The domain $D$ is the unit interval $(0,1) \subset \R$ and we consider a simplified version of the structure described in \cref{sec.2dgeomotor}: 
$D$ is solely made of two complementary phases $\Omega^1 := \Omega$ and $\Omega^2 := D \setminus \overline \Omega$, whose constituent materials have constant  
conductivity and reluctivity, set to $\sigma_1 = \nu_1 = 1$ and $\sigma_2 = \nu_2 = 5$, respectively.

This arrangement of $D$ is subjected to the deformation $\varphi_t(x) = x + t x^2$ over the time period $(0,T) = (0,1)$. 
The physical behavior of $\Omega$ is then accounted for by the solution $u_\Omega$ to the linear version of the magneto-quasi-static equation \cref{eq.uOm},
where the source term $f(t,x)$ reads: 
\begin{align*}
    f(t,\varphi_t(x)) = (x-0.4)(x-0.6)\sqrt{\varphi_t(x)}(1+t-\varphi_t(x))
\end{align*} 
see \cref{fig_1dinit} (a). 

Admittedly, this setting does not exactly comply with the assumptions of \cref{sec.setting,sec.simplefunc,sec_calcsdmqs}: 
 the mapping $t \mapsto \varphi_t(x)$ is not periodic, and so the domains $\Omega$, $D$ at initial time do not coincide with their respective final versions $\Omega(T)$, $D(T)$. However, a simple verification reveals that all the findings in there -- and notably, the well-posedness \cref{th.wellposed} about \cref{eq.uOm} and the calculation of the shape derivative in \cref{prop.sdJOmmag} -- remain valid in this context, provided the time-periodicity condition encompassed in $\Wper$ is replaced with the more general version: 
$$u(0, x) = u(T, \varphi_T(x)), \quad x \in D. $$

In this setting, we consider the shape optimization problem 
\begin{equation}\label{eq.sopbacad}
 \min\limits_{\Omega \subset D} J(\Omega), \text{ where } J(\Omega) = \int_Q u_\Omega(t,x) \:\d x\d t
 \end{equation}
 where $u_\Omega$ solves \cref{eq.uOm} with $\sigma_1=10$, $\sigma_2=0$, $\nu_1 = 1, \nu_2 = 10$, recalling the notation of \cref{sec.magsd}. 
 This problem \cref{eq.sopbacad} is solved by means of the shape gradient \cref{algo.strat}.
The inner product $b(\cdot,\cdot)$ featured in the practice \cref{eq.hilbert_trick} of the Hilbertian method is induced by the Laplace operator:
we solve \cref{eq.HilbertTrick} with $\alpha=0.5$, $\beta=0$.

The initial shape $\Omega^0$ is the interval $(0.4,0.6)$, see \cref{fig_1dinit} (b), and the solution $u_{\Omega^0}$ to \cref{eq.uOm}
is depicted in \cref{fig_1dinit} (c). The deformation $\Theta^0$ applied to the space-time mesh at the first iteration is represented on \cref{fig_1dfinal} (a). 
We apply $59$ iterations of the gradient-based \cref{algo.strat} until the norm of the calculated descent direction gets lower than a prescribed threshold of $10^{-9}$. 
The optimized design $\Omega^*$ as well as the corresponding potential $u_{\Omega^*}$ are depicted in \cref{fig_1dfinal} (b) and (c), respectively. The value of $J(\Omega)$ decreases from $5.091\cdot 10^{-4}$ to $4.231\cdot 10^{-4}$ in the process. 
\black

\begin{figure}[!ht]
\centering
\begin{tabular}{ccc}
\begin{minipage}{0.33\textwidth}
\begin{overpic}[width=1.0\textwidth]{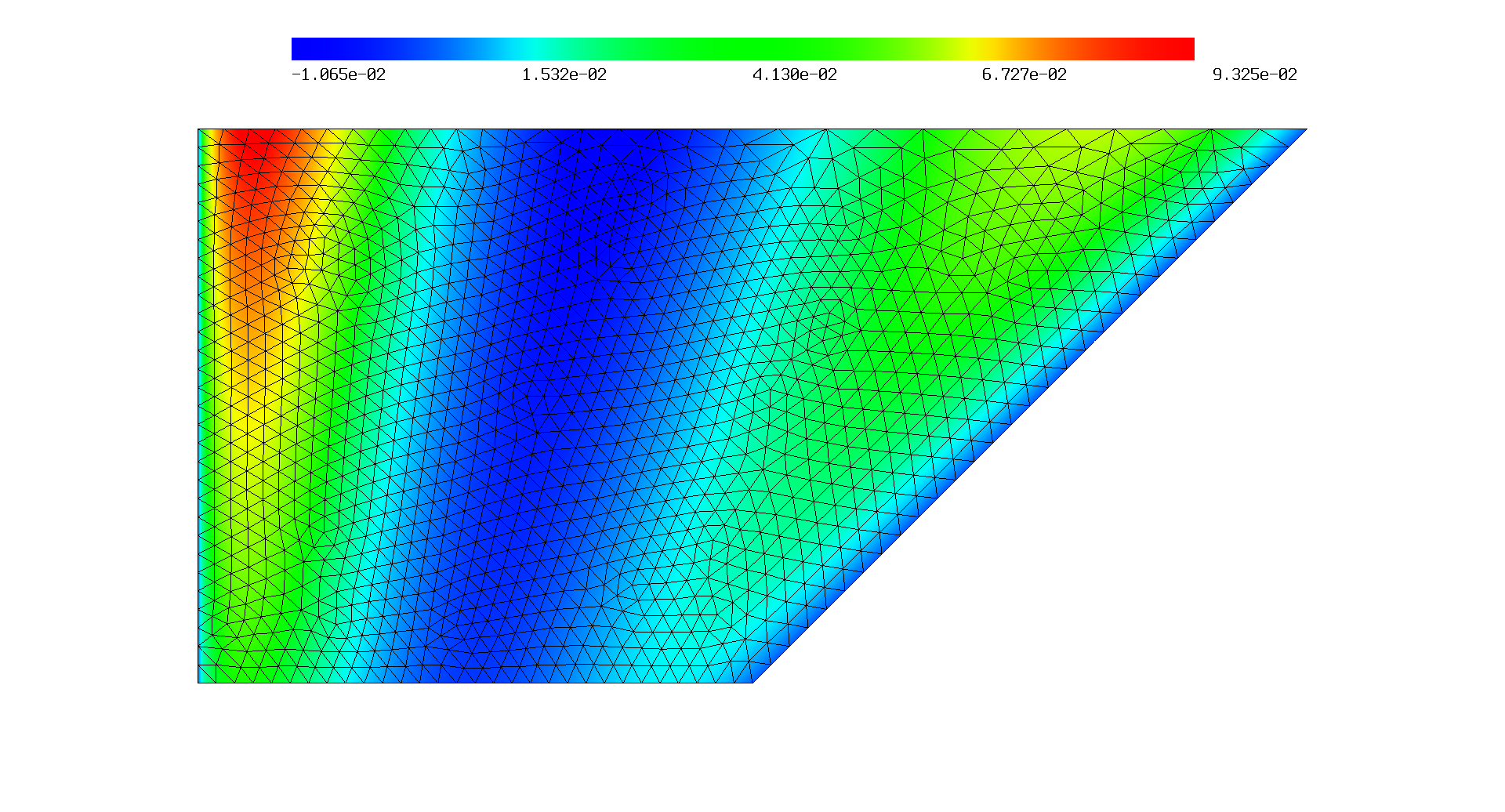}
\put(20,3){\fcolorbox{black}{white}{$a$}}
\end{overpic}
\end{minipage}& 
\begin{minipage}{0.33\textwidth}
\begin{overpic}[width=1.0\textwidth]{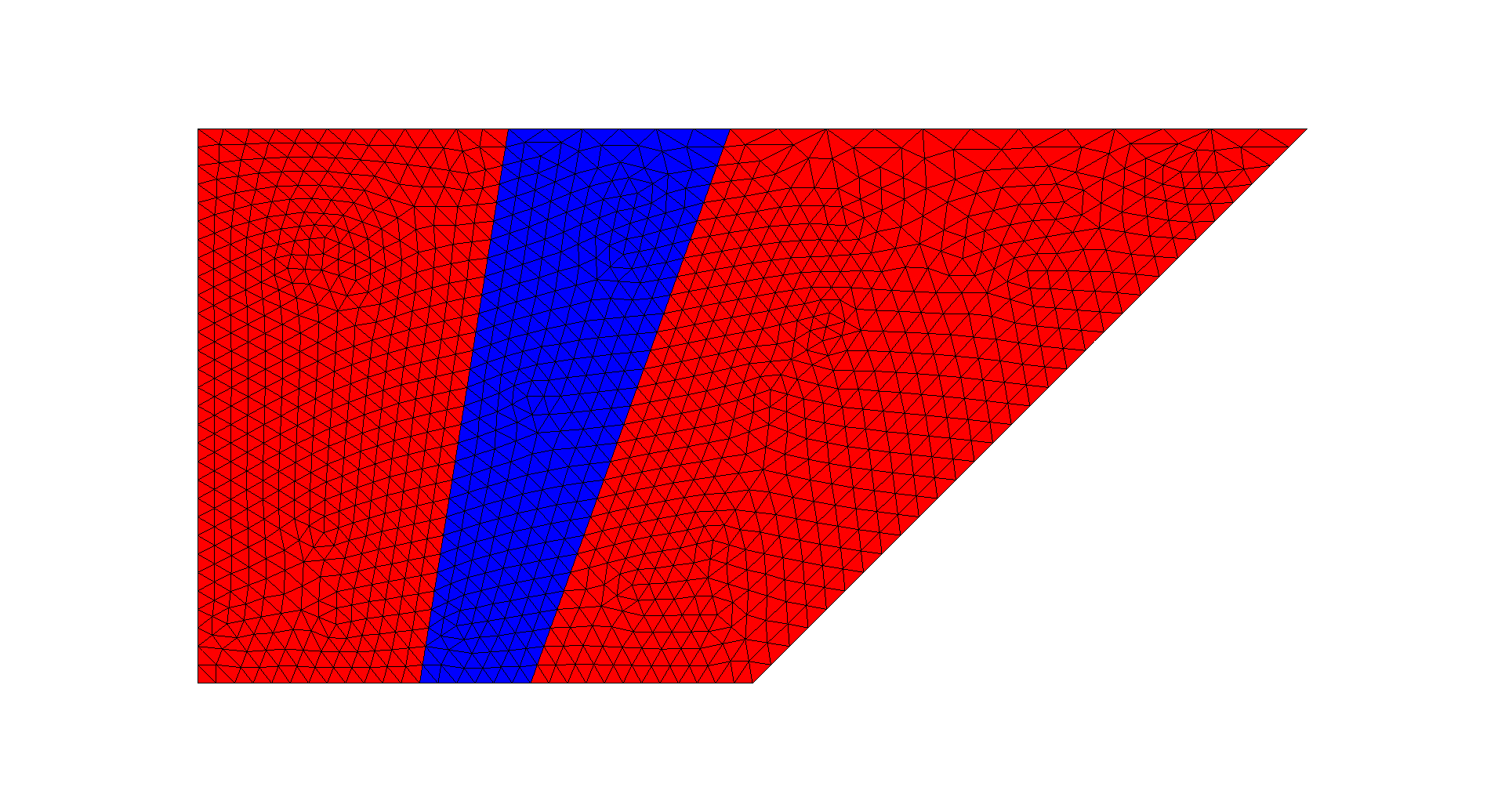}
\put(20,3){\fcolorbox{black}{white}{$b$}}
\end{overpic}
\end{minipage}&
\begin{minipage}{0.33\textwidth}
\begin{overpic}[width=1.0\textwidth]{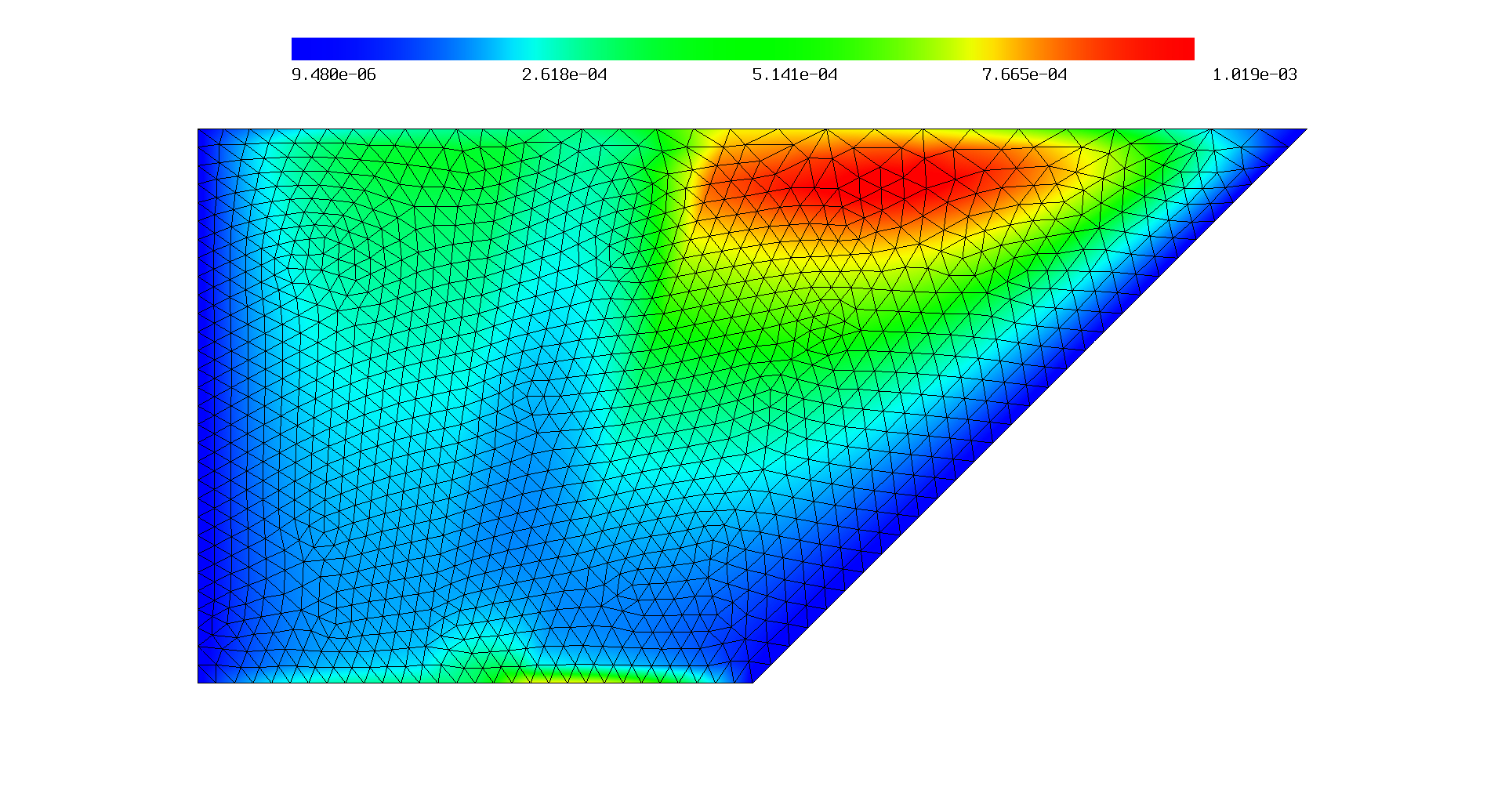}
\put(20,3){\fcolorbox{black}{white}{$c$}}
\end{overpic}
\end{minipage}
\end{tabular}
    \caption{\it (a) Right-hand side $f(t,x)$; (b) Space-time cylinder $Q$ associated to the initial design $\Omega^0$ in the 1d academic example of \cref{sec.1dex};  (c) Potential $u_{\Omega^0}$ associated to $\Omega^0$.}
\label{fig_1dinit}
\end{figure}

\begin{figure}[!ht]
\centering
\begin{tabular}{ccc}
\begin{minipage}{0.33\textwidth}
\begin{overpic}[width=1.0\textwidth]{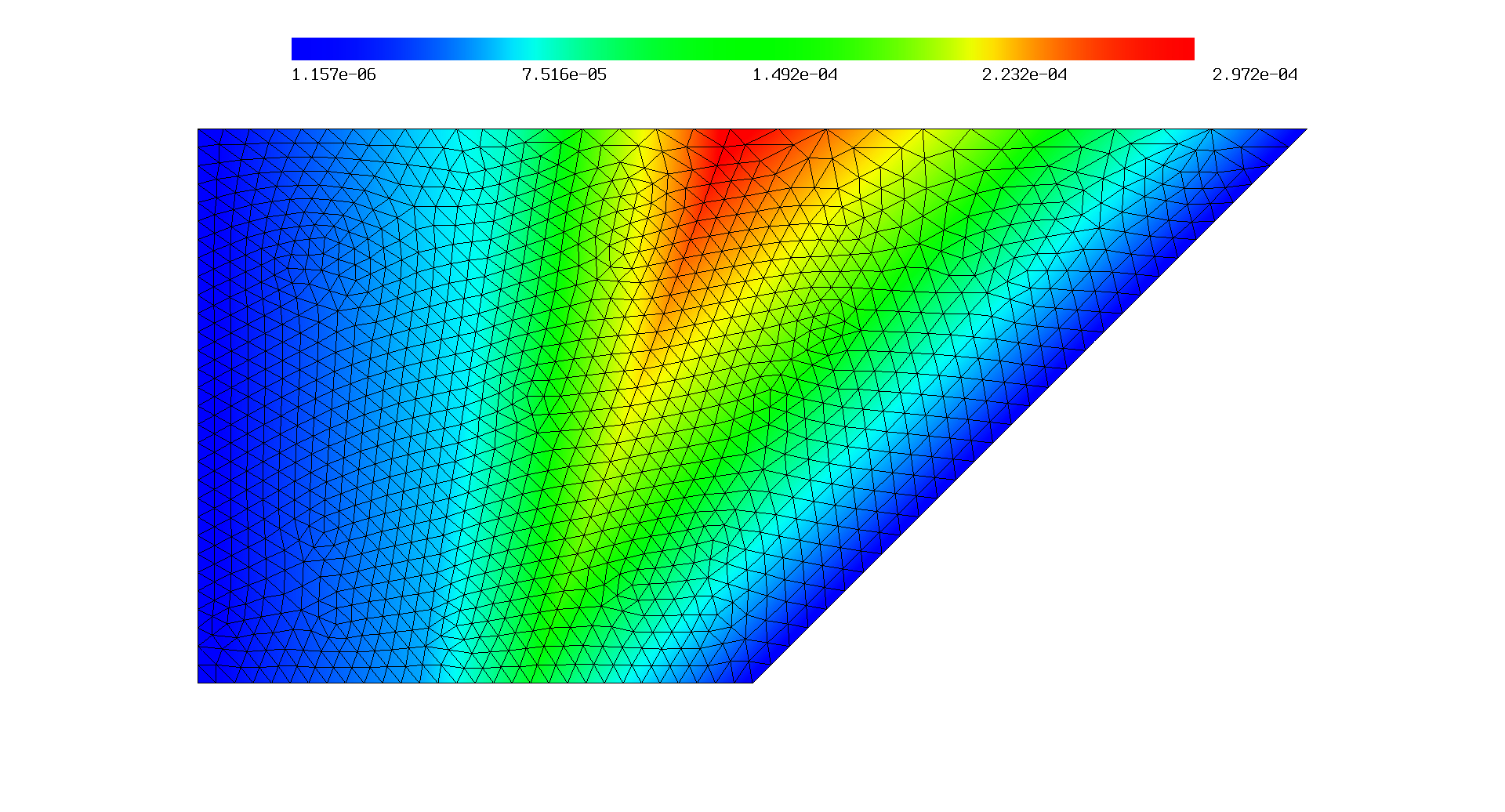}
\put(20,3){\fcolorbox{black}{white}{$a$}}
\end{overpic}
\end{minipage}&
\begin{minipage}{0.33\textwidth}
\begin{overpic}[width=1.0\textwidth]{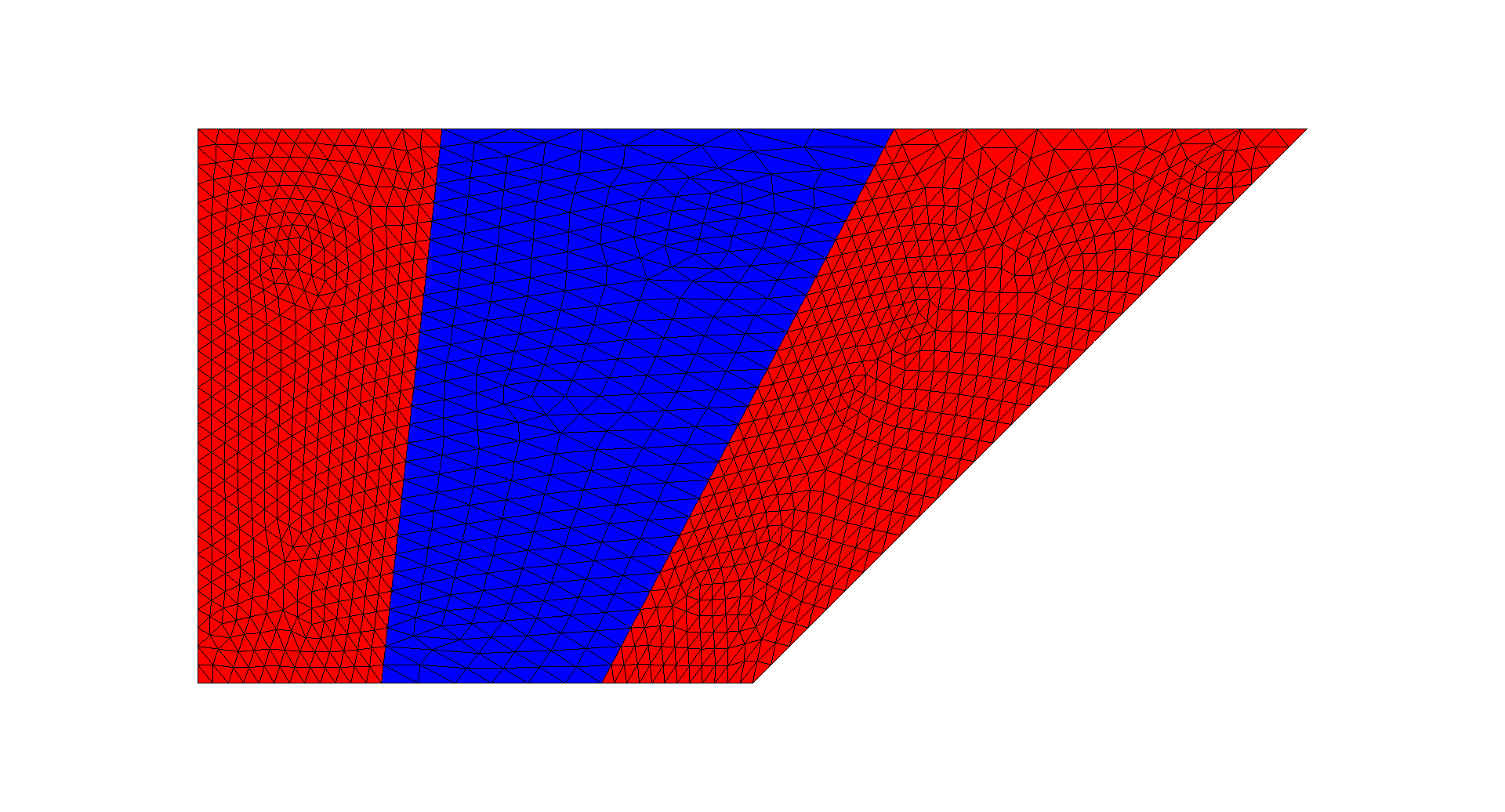}
\put(20,3){\fcolorbox{black}{white}{$b$}}
\end{overpic}
\end{minipage}& 
\begin{minipage}{0.33\textwidth}
\begin{overpic}[width=1.0\textwidth]{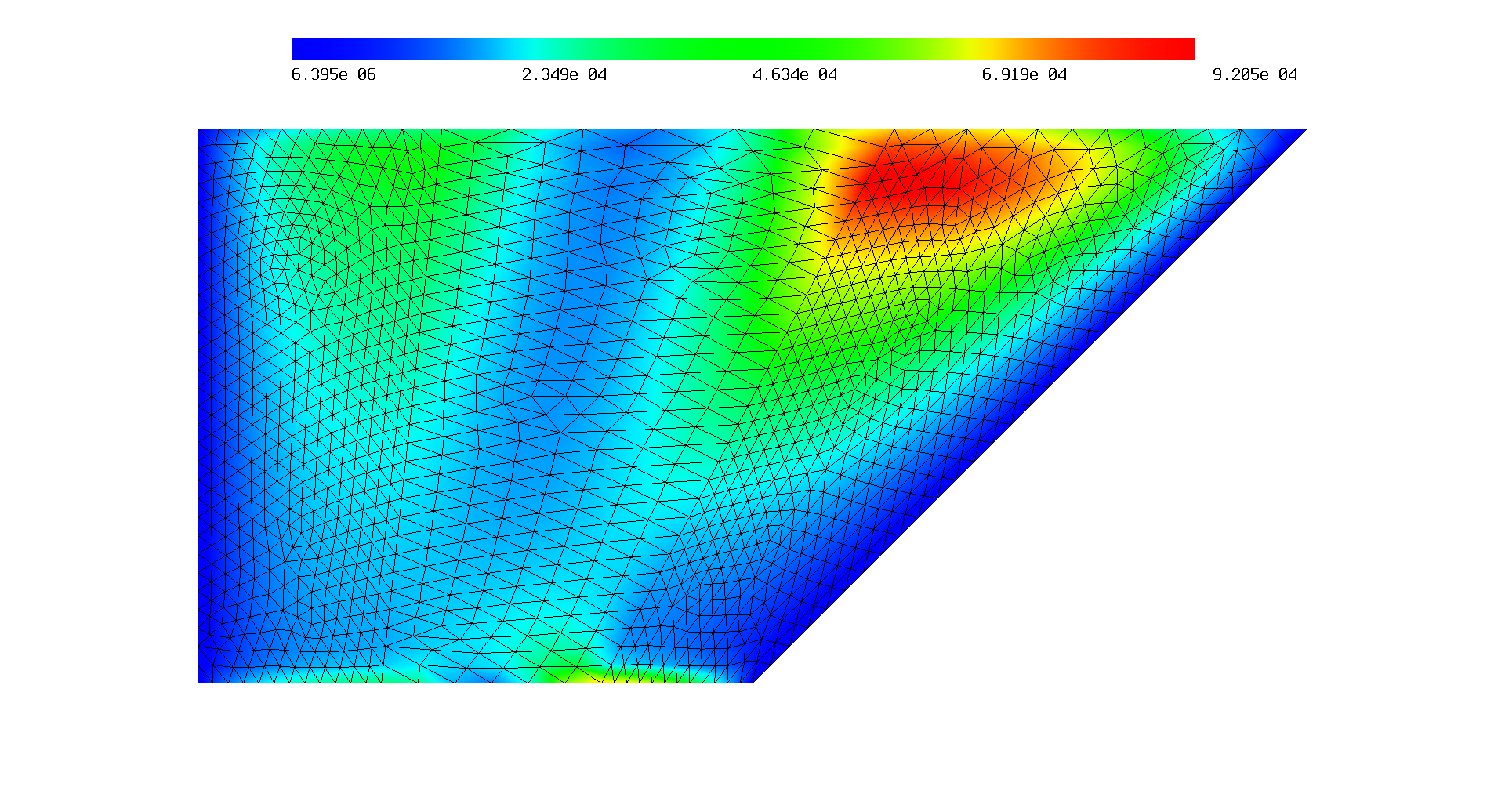}
\put(20,3){\fcolorbox{black}{white}{$c$}}
\end{overpic}
\end{minipage}
\end{tabular}
    \caption{\it (a) Descent direction $\Theta^0$ from the initial design $\Omega^0$, after extension of $\theta^0$ to the space-time cylinder; (b) Space-time cylinder $Q_{\Omega^*}$ associated to the optimized design $\Omega^*$; (c) Potential $u_{\Omega^*}$ for the optimized design $\Omega^*$.}
\label{fig_1dfinal}
\end{figure}

\subsection{Optimization of the structure of a rotating electric machine}\label{sec.exrot}

\noindent 
In this section, we consider the optimization of the structure of an electric motor. 

\subsubsection{Description of the geometric setting}

\noindent 
The considered machine operates at steady state with a constant rotational speed of 600 rotations per minute. 
The geometry of its 2d cross-section $D$ is symmetric with respect to both coordinates axes, see again \cref{fig_machine}; hence, 
the time horizon $T$ is chosen as the duration of a 90 degree rotation, i.e. $T = 0.025$ s.

The conductivities of the various materials at play take the values:
$$\sigma_f = \sigma_a = \sigma_c = 0, \quad \sigma_m = 10^6 \:\: \text{A}^2 \cdot  \text{s}^3 \cdot \text{kg}^{-1} \cdot \text{m}^{-3},$$ 
and the reluctivity of air, copper and magnet equal:
$$\nu_a = \nu_c = \frac{10^7} {4 \pi} \:\: \text{Am}\cdot \text{V}^{-1} \cdot \text{s}^{-1}, \text{ and } \nu_m = \frac{\nu_a }{1.05} \:\: \text{Am} \cdot \text{V}^{-1} \cdot\text{s}^{-1}.$$ 
The material behaviour of the ferromagnetic material is characterized by the nonlinear reluctivity function $\hat\nu$ defined by:
$$\hat \nu(s) = \nu_a-(\nu_a-c_1) \text{exp}(-c_2 s^{c_3}),$$ 
with the parameters $c_1 = 200$, $c_2=0.001$, $c_3=6$, see \cref{fig_nu} (a).
We also assume that the magnet region $\Dmag$ is the support of a permanent magnetization field $M$.
The direction of $M$ is depicted in \cref{fig_nu} (b) and its magnitude equals $\lvert M \lvert = \nu_m B_r  \: \text{A} \cdot \text{m}^{-1}$, where $B_r=1.216 \: \text{V}\cdot \text{s} \cdot\text{m}^{-2}$ is the magnetic remanence. 

The coil region $D_{\text{stat,c}}$ of the stator is occupied by a large number of thin copper wires wound around the iron core of the machine.  
This phase is subdivided into six regions $\Omega_{U_+},\Omega_{U_-},\Omega_{V_+},\Omega_{V_-},\Omega_{W_+},\Omega_{W_-}$ and the source current density $f$ is spatially homogeneous with smooth time variations in each of these:
\begin{align} \label{eq.imprcurrent}
    f(t,x) =& \Big(\mathds{1}_{\Omega_{U^+}}(x) - \mathds{1}_{\Omega_{U^-}}(x)\Big) I_U(t) + \Big( \mathds{1}_{\Omega_{V^+}}(x)  -  \mathds{1}_{\Omega_{V^-}}(x)  \Big)I_V(t) + \Big( \mathds{1}_{\Omega_{W^+}}(x) -  \mathds{1}_{\Omega_{W^-}}(x) \Big) I_W(t).
\end{align}
In this expression, the current densities read:
\begin{align*}
    I_U(t) = a \, \sin (4 \alpha(t) + \psi), &&
    I_V(t) = a \, \sin \left(4 \alpha(t) + \psi + \frac{4}{3} \pi \right), &&
    I_W(t) = a \, \sin \left(4 \alpha(t) + \psi + \frac{2}{3} \pi\right),
\end{align*}
where $\alpha(t) = \frac{\pi}{2} \frac t T$ and $\psi = \frac{\pi}{18}$ is called the current angle.
Here, the amplitude $a$ is given by $a = I \, c / \calA$ where $ I =  1555.64$ A is the product of the current intensity and the number of copper wires per coil\black, $c=2.75$ is the stacking factor and $\calA \approx 1.8\cdot 10^{-4}$ m$^2$ is the area of one coil.

In this application, the transformation $\varphi_t$ stands for a rotation of the inner part $\Drot$ at constant angular velocity $\alpha'(t) = \frac{\pi}{2T}$ s$^{-1}$, i.e.,
$$
    \varphi_t(x) = \left\{
    \begin{array}{cl}
                        R_{\alpha(t)} x & \text{if } x \in \Drot \\
                        x & \text{if } x \in \Dstat,
        \end{array}\right.         
                 \qquad \text{where } R_{\alpha(t)} = \left(
\begin{array}{cc}
\cos \alpha(t) & -\sin \alpha(t) \\
\sin\alpha(t) & \cos \alpha(t)
\end{array}
\right).
$$

In this setting, we aim to optimize the average torque of the device by acting on the shape of the phase $\Omega \subset \Drot$ containing ferromagnetic material 
-- and thereby on the shape of the air pockets near the magnets in the rotor domain -- while keeping fixed the region $\Dmag$ containing the permanent magnets:
\begin{equation}\label{eq.minmtorque}
 \min\limits_{\Omega \subset D} J(\Omega), \text{ where } J(\Omega) := -\Tor(u_\Omega), 
 \end{equation}
where the magneto-quasi-static potential $u_\Omega$ is the solution to the non linear evolution problem \cref{eq.uOm} and the torque $\Tor(u)$ is defined in \cref{eqn:torque}.
Note that, strictly speaking, $J(\Omega)$ is not of the form considered in \cref{sec_calcsdmqs}, 
as it involves the gradient of the potential $u_\Omega$. Nevertheless, as noted in \cref{sec.sopb}, our calculations performed can be adapted straightforwardly to handle this case. 

\subsubsection{Practical implementation details} \label{sec.primde}

\noindent At each iteration $n=0,\ldots$ of \cref{algo.strat}, the tetrahedral space-time mesh $\calK^n$ of the cylinder $Q$ consists of 320,597 vertices and 1,758,666 elements. 
As discussed in \cref{sec_resstad}, the non linear magneto-quasi-static problem \cref{eq.uOm} for $u_\Omega$ is solved by a Newton-Raphson method with damping, based on the space-time finite element framework. 
In the first optimization iteration $n=0$, this method is initialized with the function $0$, 
while at each subsequent iteration $n \geq 1$, the numerical solution $u_{\Omega^{n-1}}$ to \cref{eq.uOm} at the previous iteration is chosen instead,
which reduces the number of necessary Newton iterations. The systems of linear equations involved in the procedure are solved in parallel with \texttt{openMPI} \cite{open_mpi} and \texttt{MUMPS} \cite{mumps2000}, via the \texttt{PETSc}\cite{petsc-user-ref} library. We use \texttt{NGSolve} for both the meshing and the assembly of the finite element stiffness matrices and right-hand side vectors. 
Since, at the current state, \texttt{NGSolve} does not support MPI parallelization in combination with periodic boundary conditions, the meshing and finite element assembly operations can be carried out on one core only. 
We run the simulations on RICAM's RADON cluster\footnote{https://www.oeaw.ac.at/ricam/hpc} by using a single node with 40 cores and 1TB of RAM. 
The solution of the state problem \cref{eq.uOm} at the first optimization iteration $n=0$ requires 19 Newton iterations, for about one hour of computation. 
The needed CPU time to achieve this operation is reduced by half at the subsequent steps, as only 9 Newton iterations are needed. The (linear) adjoint problem \cref{eq.pOm} is solved in the same way within about two minutes. \black

The evaluation of the volume form of the shape derivative $J^\prime(\Omega)(\theta)$ is carried out along the lines of \cref{sec_evasd}.
Since only the interface $\Gamma$ between the ferromagnetic material and air phases $\Omega$ and $\Omega_a$ is optimized, we restrict 
the resulting descent directions $\theta$ considered in the shape derivative \cref{eq.Jpvolnl} and the auxiliary problem \cref{eq.hilbert_trick} to the rotor domain $\Drot$ deprived of the magnet region $\Dmag$, 
i.e. we actually solve the identification problem \cref{eq.HilbertTrick} with the Hilbert space $V= H^1_0(\Omega \cup \Omega_{a})^2$ and the inner product in \cref{eq.hilbert_trick} with the weights $\alpha(x) \equiv 1$ and $\beta(x) \equiv 0$. 
In addition, following \cite{IglesiasSturmWechsung2018} we add a term of Cauchy-Riemann type to the bilinear form $b(\cdot, \cdot)$ which helps in preserving the mesh quality.

\begin{remark}
Note that the smoothness assumption about the source function $f$ in the derivation of \cref{sec_calcsdmqs} is not satisfied by \cref{eq.imprcurrent}. 
However, as noted in \cref{eq.usedvM}, \cref{prop.sdJOmmag,prop.sdJOmmagnl} still hold in the present case where $f(t,\cdot)$ has support only in $\Dstat$.
\end{remark}

\begin{figure}
\centering
\begin{tabular}{cc}
\begin{minipage}{0.5\textwidth}
\begin{overpic}[width=1.0\textwidth]{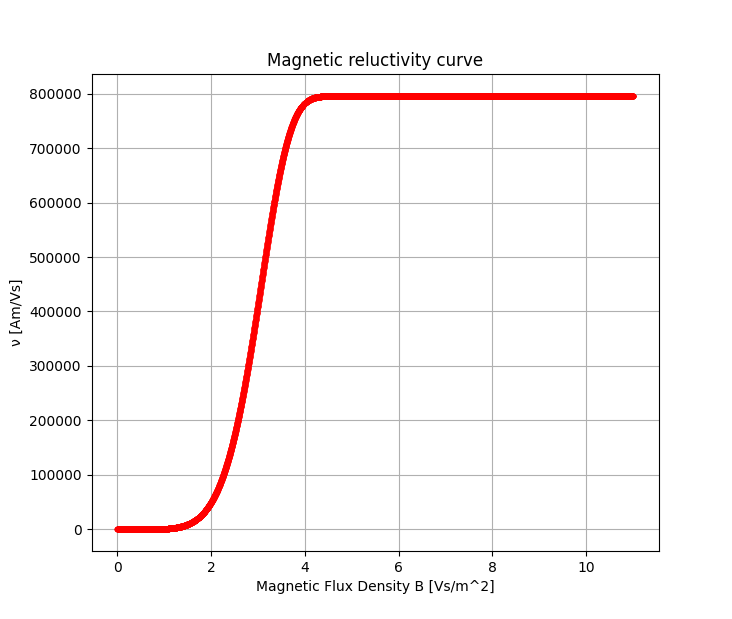}
\put(0,3){\fcolorbox{black}{white}{$a$}}
\end{overpic}
\end{minipage}&
\begin{minipage}{0.44\textwidth}
\begin{overpic}[width=1.0\textwidth]{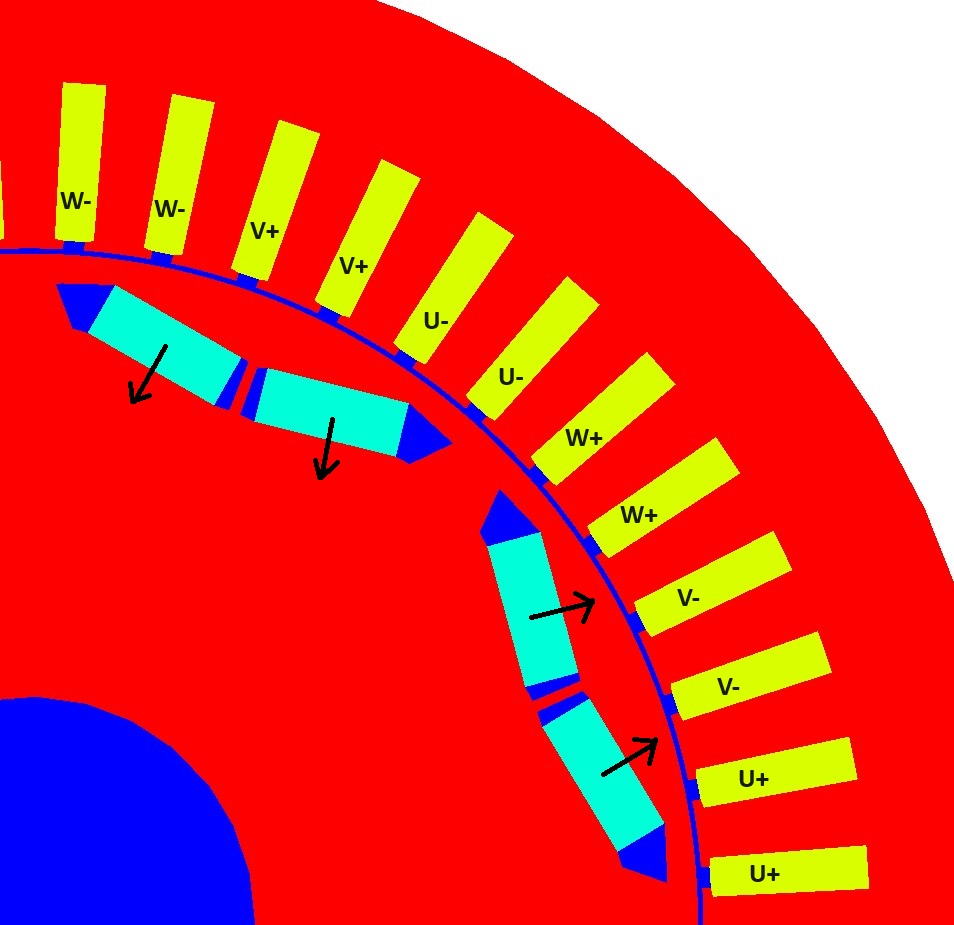}
\put(0,3){\fcolorbox{black}{white}{$b$}}
\end{overpic}
\end{minipage}
\end{tabular}
    \caption{\it Magnetic reluctivity curve $s \mapsto \hat \nu(s)$ in the ferromagnetic material considered in the motor optimal design example of \cref{sec.exrot}; (b) Zoom on the initial design $\Omega^0$.}
\label{fig_nu}
\end{figure}

\begin{figure}
\centering
\begin{tabular}{ccc}
\begin{minipage}{0.33\textwidth}
\begin{overpic}[width=1.0\textwidth]{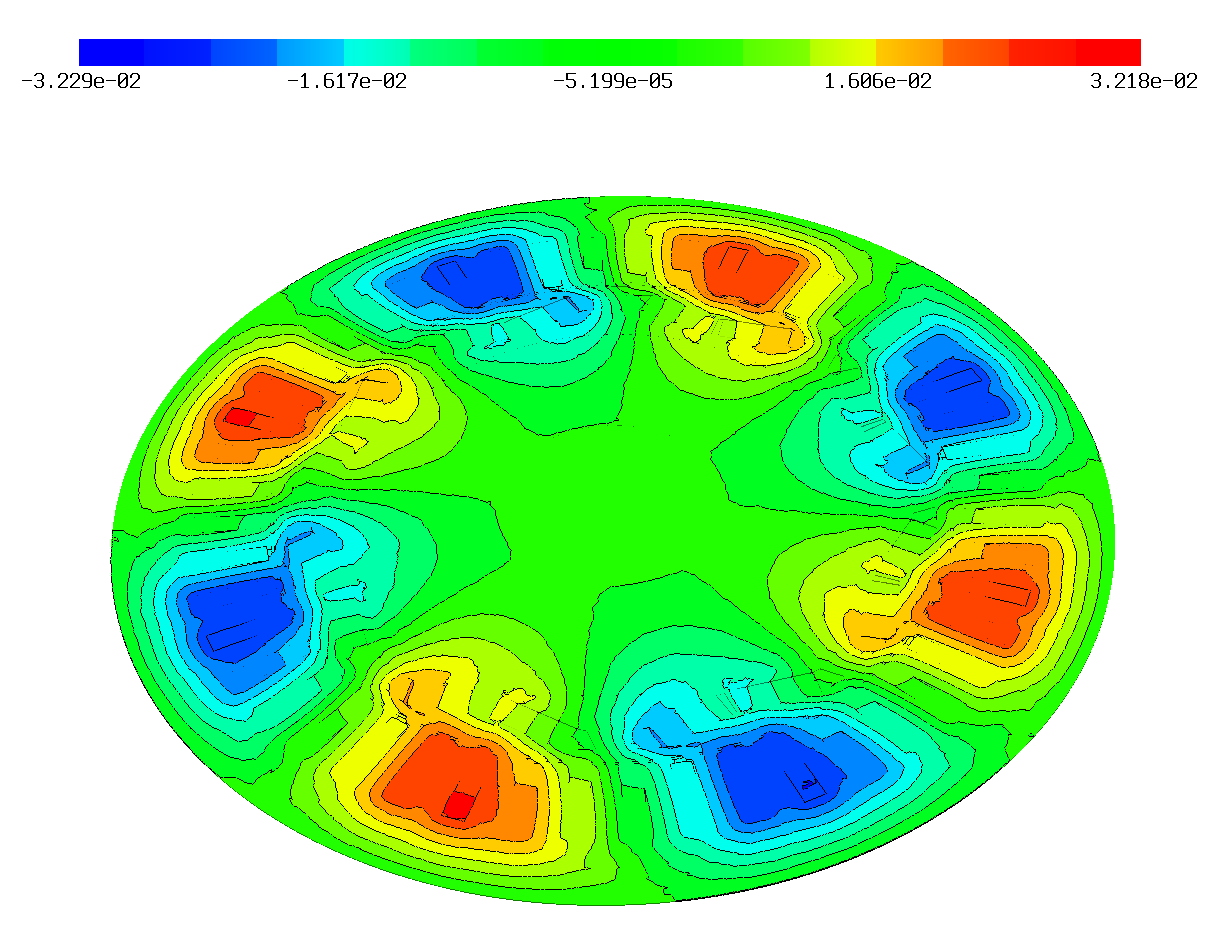}
\put(0,3){\fcolorbox{black}{white}{$a$}}
\end{overpic}
\end{minipage}&
\begin{minipage}{0.33\textwidth}
\begin{overpic}[width=1.0\textwidth]{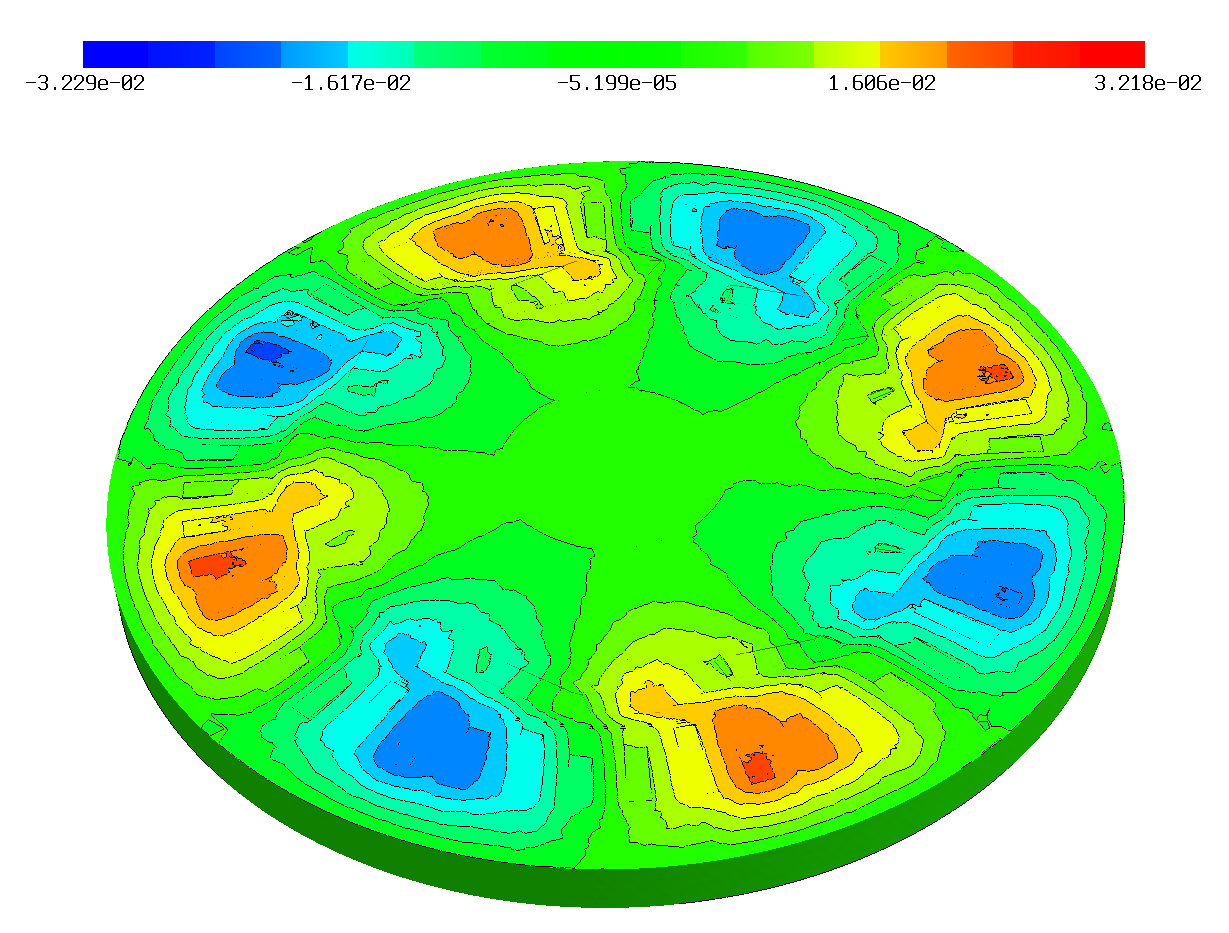}
\put(0,3){\fcolorbox{black}{white}{$b$}}
\end{overpic}
\end{minipage} & 
\begin{minipage}{0.33\textwidth}
\begin{overpic}[width=1.0\textwidth]{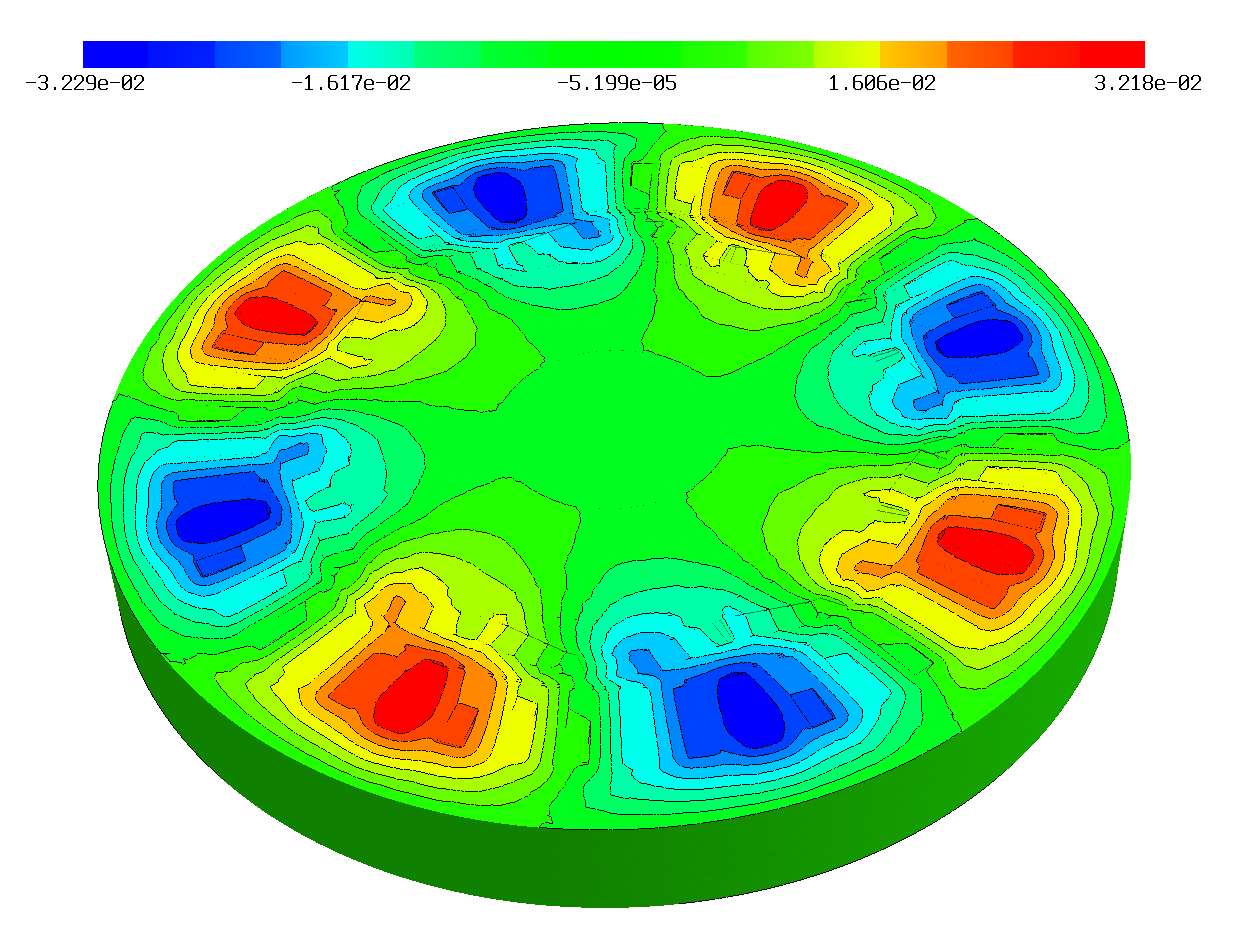}
\put(0,3){\fcolorbox{black}{white}{$c$}}
\end{overpic}
\end{minipage}
\end{tabular}
    \caption{\it Solution $u_{\Omega^0}$ to the version of the magneto-quasi-static state problem \cref{eq.uOm} considered in \cref{sec.exrot}, attached to the initial design $\Omega^0$ at times (a) $t=0$; (b) $t=T/2$; (c) $t=T$.}
\label{fig_stateinit}
\end{figure}

\begin{figure}
\centering
\begin{tabular}{cc}
\begin{minipage}{0.5\textwidth}
\begin{overpic}[width=1.0\textwidth]{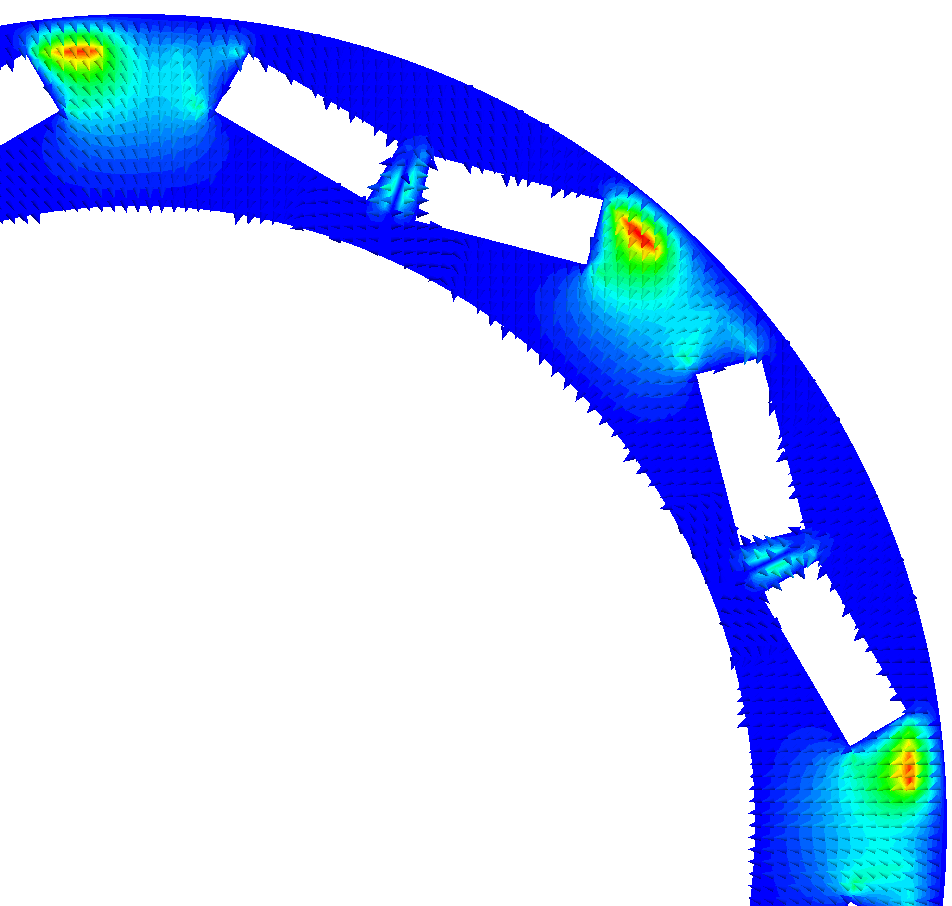}
\put(0,3){\fcolorbox{black}{white}{$a$}}
\end{overpic}
\end{minipage}&
\begin{minipage}{0.42\textwidth}
\begin{overpic}[width=1.0\textwidth]{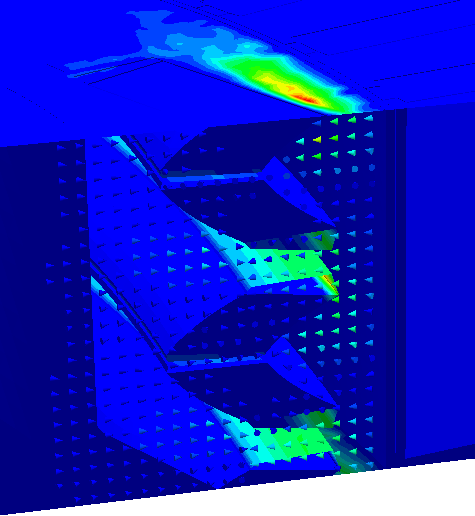}
\put(0,3){\fcolorbox{black}{white}{$b$}}
\end{overpic}
\end{minipage} 
\end{tabular}
    \caption{\it (a) Spatial descent direction $\theta^0$ at the initial iteration; (b) Corresponding space-time deformation $(0,\Theta^0(t,x)-x)$.}
\label{fig_shapegrad_init}
\end{figure}

\begin{figure}
\centering
\begin{tabular}{ccc}
\begin{minipage}{0.33\textwidth}
\begin{overpic}[width=1.0\textwidth]{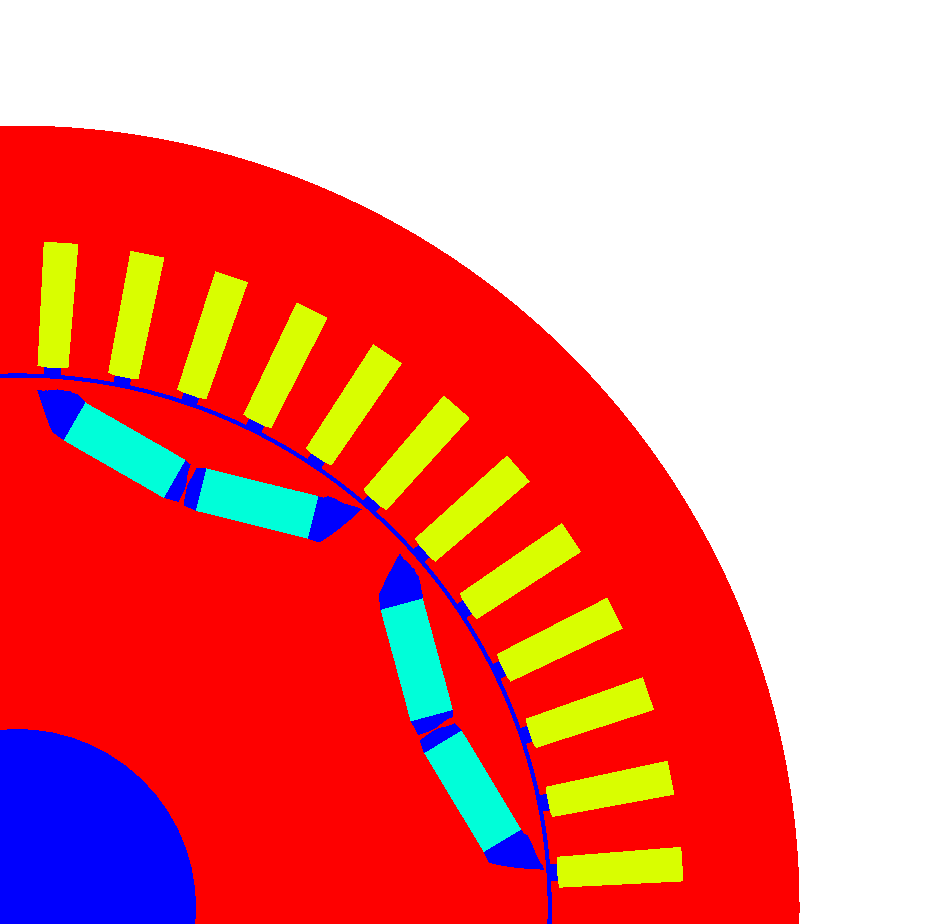}
\put(0,3){\fcolorbox{black}{white}{$a$}}
\end{overpic}
\end{minipage}&
\begin{minipage}{0.33\textwidth}
\begin{overpic}[width=1.0\textwidth]{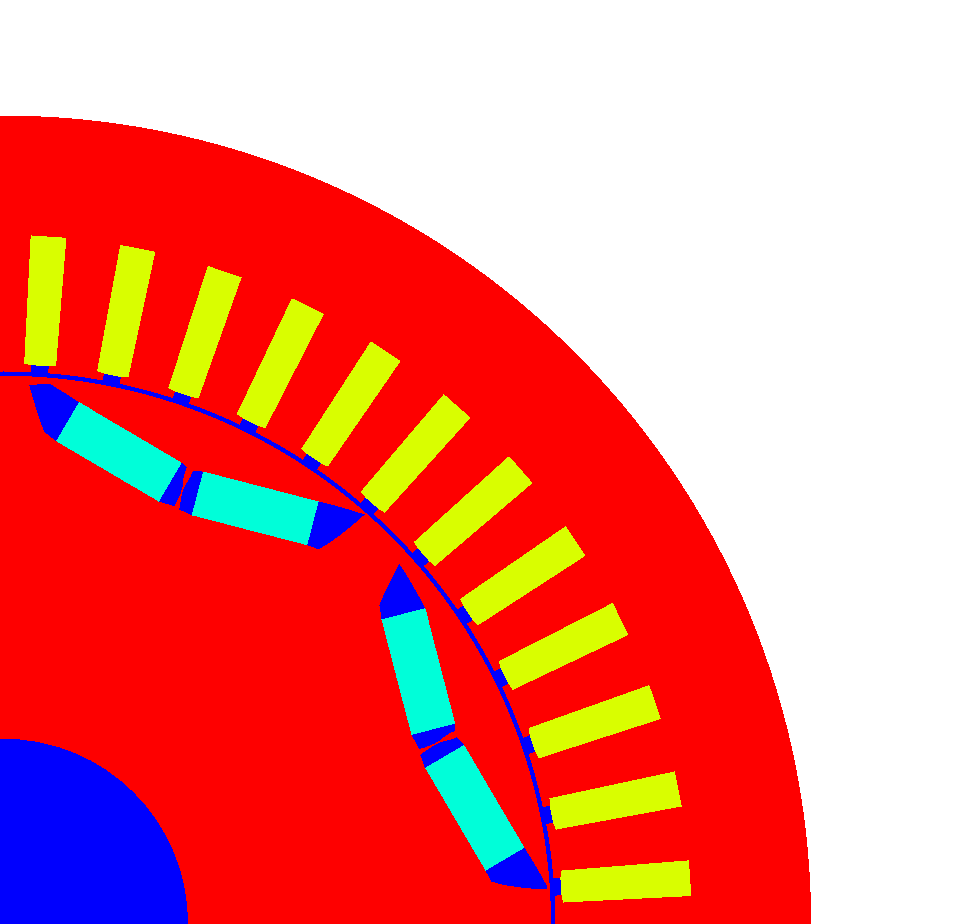}
\put(0,3){\fcolorbox{black}{white}{$b$}}
\end{overpic}
\end{minipage} & 
\begin{minipage}{0.33\textwidth}
\begin{overpic}[width=1.0\textwidth]{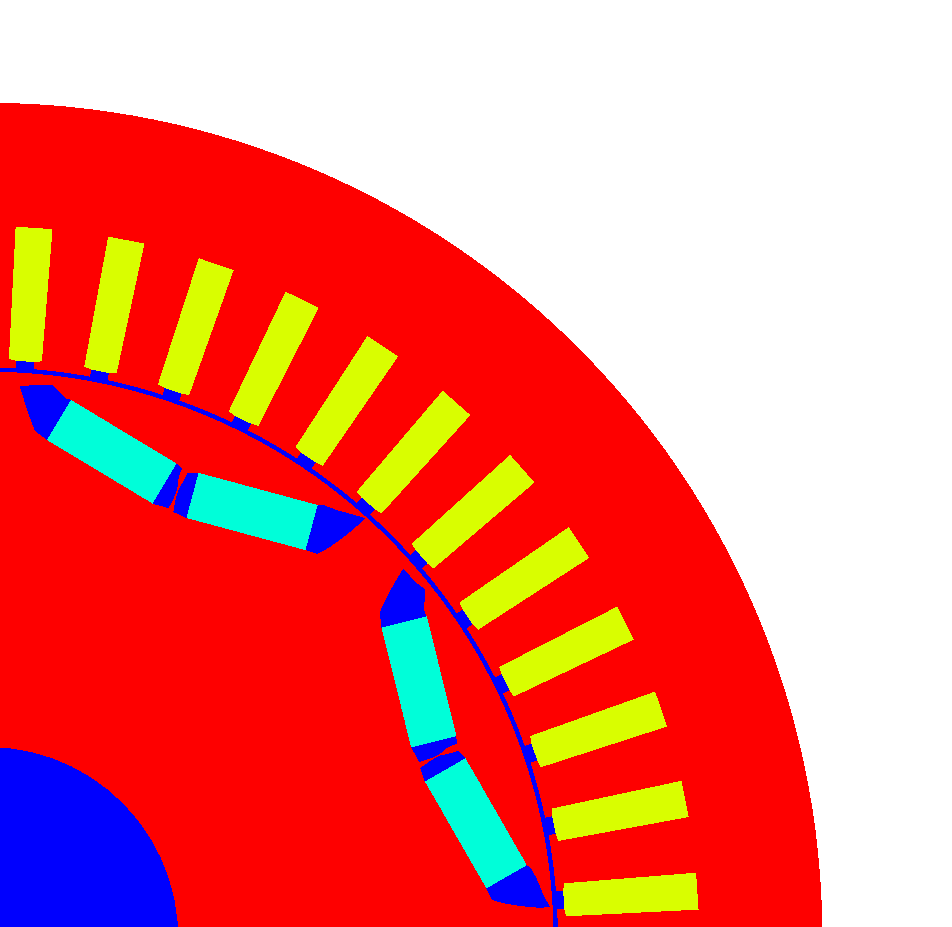}
\put(0,3){\fcolorbox{black}{white}{$c$}}
\end{overpic}
\end{minipage}
\end{tabular}
    \caption{\it Optimized design of the rotating machine considered in \cref{sec.exrot} at times (a) $t=0$; (b) $t = T/2$, and (c) $t=T$.}
\label{fig_designfinal}
\end{figure}

\subsubsection{Description of the numerical results}

\noindent The initial design $\Omega^0$ is sketched in \cref{fig_nu} (b), and the values of the potential $u_{\Omega^0}(t,\cdot)$ at times $t=0$, $t=T/2$ and $t=T$ are depicted in \cref{fig_stateinit} in three cross sections of the space-time mesh. The first descent direction $\theta^0 :D \to \R^d$ is shown in \cref{fig_shapegrad_init} (a) and the corresponding space-time vector field $(0, \Theta^0_x(t,x)-x)$ is represented at several times in \cref{fig_shapegrad_init} (b).

At each iteration of \cref{algo.strat}, a suitable descent step $\tau^n$ is calculated via a line search procedure 
in which $\tau^n$ is decreased by half as long as the scaled deformation $\tau^n \theta^n$ does not allow for a decrease of the objective function $J(\Omega)$.
This procedure is interrupted when $\tau^n$ becomes smaller than $10^{-10}$. 

The optimized design resulting from this procedure after 14 iterations is depicted in \cref{fig_designfinal} at the three different times $t=0$, $t=T/2$ and $t=T$; the torque $\Tor(u_\Omega)$ is increased from 522.94 N$\cdot$m to 587.79 N$\cdot$m in the process.
We note that, at the final iteration, the norm of the vector field $\theta^n$ does not vanish, but it no longer allows to decrease the value of $J(\Omega)$.  
One possible explanation for this fact is that the components of $\Omega_{\text{a}}$ made of the air pockets close to the magnet regions $\Dmag$ 
tend to get closer and eventually merge with the air gap region $\Dgap$, which is not allowed by the homogeneous Dirichlet boundary condition on $\partial \Dstat$ imposed in the identification problem \cref{eq.HilbertTrick} for $\theta$. In principle, this undesirable behavior could be prevented by 
adding a constraint about the mechanical stiffness of the device to the optimization problem \cref{eq.minmtorque}, 
as in e.g. \cite{brun2023magneto}. This issue will be addressed in a further work.

\begin{remark}
A careful look at \cref{fig_designfinal} shows that the rotated versions of the optimized design at the considered times $t^k$
do not have exactly the same look.
This numerical artifact is caused by the approximation of the extension of the deformation vector field $\theta$ into the space-time deformation mapping $\Theta$ in \cref{eq.stdef}, which is 
realized on an unstructured mesh of $Q$. 
\end{remark}

\section{Conclusion and perspectives}\label{sec.concl}

\noindent In this article, we have addressed a shape optimization problem related to the internal structure of an electric motor. 
This work departs from the related literature by the fact that the physical problem at stake is described by a version of the non linear magneto-quasi-static equation featuring a time evolving geometry -- a rich setting, which captures fine, realistic effects. 
After proving the well-posedness of the physical problem and calculating the shape derivative of a generic optimization criterion in this context, we have proposed a shape gradient workflow 
based on the space-time finite element method. We have validated our numerical strategy on an academic problem and tackled a more realistic example in the physical context of interest. 

The work in this article paves the way to multiple further investigations. On the one hand, the physical model, although already complicated enough, could be enriched to include multiphysics effects, 
for instance couplings between electromagnetic and thermal effects. 
This could require time horizons larger than one single rotation period, and could be addressed by combining the present space-time finite element method with time-stepping techniques. 
On a different note, constraints about the mechanical integrity and stability of the optimized design could be modeled with the help of structural mechanics, as suggested in \cite{wang2022topology}. From the numerical viewpoint, 
more robust methods could be implemented to deal with the update of the shape, such as the level set method \cite{allaire2004structural} or one of its avatars \cite{allaire2014shape,dapogny2022shape} retaining an exact, meshed description of the optimized phases. Eventually, optimal design strategies based on topological derivatives could also be considered in this context, see e.g. \cite{GanglKrenn2022} for a related investigation. \par
\bigskip

\paragraph{\textbf{Acknowledgments.}} The work of A. C. and P. G. is supported by the FWF funded project P32911 as well as the joint DFG/FWF Collaborative Research Centre CREATOR (CRC -- TRR361/F90) at TU Darmstadt, TU Graz, RICAM and JKU Linz.
This work was completed while C. D. was visiting the Laboratoire Jacques-Louis Lions from Universit\'e Paris-Sorbonne,
whose hospitality is thankfully acknowledged.

\appendix
\section{A few useful formulas in shape calculus}\label{app.tech}

\noindent The following integration by parts formula is a version of the usual Green's formula that is used repeatedly in the main part of the article. 

\begin{proposition}\label{lem.ippNtheta}
Let $\Omega \subset \R^d$ be a Lipschitz bounded domain. 
\begin{enumerate}[(i)]
\item Let $\theta \in \Winfty$ and $a : \R^d \to \R$ be a smooth scalar field. Then: 
$$ \int_\Omega \frac{\partial \theta_i}{\partial x_j} a \:\d x = \int_{\partial \Omega} \theta_i a n_j \:\d s - \int_\Omega \frac{\partial a}{\partial x_j} \theta_i \:\d x. $$
\item Let $\theta \in \Winfty$ and $a ,b : \R^d \to \R^d$ be smooth vector fields. 
The following integration by parts formula holds:
$$ \int_\Omega \nabla \theta a \cdot b \:\d x = \int_{\partial \Omega} (\theta \cdot b )(a \cdot n) \:\d s - \int_\Omega (\dv a )(\theta \cdot b) \:\d x - \int_\Omega \nabla b a \cdot \theta \:\d x $$
\end{enumerate} 
\end{proposition}

Let us also recall the following relation between the normal and tangent vectors to the boundary of a domain and those attached to a deformed configuration. 

\begin{lemma}\label{lem.nortan}
Let $\Omega \subset \R^d$ be a bounded Lipschitz domain, and let $\varphi : \R^d \to \R^d$ be a Lipschitz continuous homeomorphism with Lipschitz inverse. 
Then, for any tangential vector field $\tau : \partial \Omega \to \R^d$, the vector field $\nabla \varphi(\varphi^{-1}(y)) \tau(\varphi^{-1}(y))$ 
is tangential to $\partial (\varphi(\Omega))$. 
Besides, the unit normal vector fields $n_\Omega$ and $n_{\varphi(\Omega)}$ are related by:
$$ \forall x \in \Omega, \quad n_{\varphi(\Omega)}(\varphi(x)) = \frac{1}{|\com(\nabla \varphi(x)) n_\Omega(x)|} \com(\nabla \varphi(x)) n_\Omega(x),$$
where $\com(M) \in \R^{d\times d}$ is the cofactor matrix of a $d \times d$ matrix $M$.
\end{lemma}

\section{Main results about space-time variational problems}\label{app.varft}

\noindent This appendix is devoted to the proof of \cref{th.wellposed}, which states the well-posedness of the version \cref{eq.uOm} of the non linear magneto-quasi-static equation
featuring coefficients depending on the rotating distribution of materials within the domain $D$. 
Although the proof hinges on quite classical methods for the analysis of non linear evolution equations and mixed parabolic-elliptic equations, the exact setting considered in this article, featuring a time-dependent geometry, is new to the best of our knowledge (see however \cite{bachinger2005numerical,kolmbauer2012multiharmonic} for related investigations). 

After recalling a few facts about time-space functional spaces in \cref{app.vdistrib} and providing the statement of some needed results about the well-posedness of non linear boundary value and evolution problem in \cref{app.prelNL}, 
we detail the proof of \cref{th.wellposed}, properly speaking, in \cref{app.proofThmain}.

\subsection{A few words about time-space functional spaces}\label{app.vdistrib}

\noindent This section gathers a few definitions and basic facts about the time-space functional spaces which are naturally involved in the mathematical formulation of evolution problems. \par\medskip

Let $V \subset H \subset V^*$ be an evolution triple, that is: 
\begin{itemize}
\item $(V,\langle \cdot , \cdot \rangle_V)$ and $(H,\langle \cdot , \cdot \rangle_H)$ are real and separable Hilbert spaces, and $V^*$ is the topological dual of $V$; 
we denote by $\langle \cdot , \cdot \rangle_{V^*,V}$ the usual duality pairing between $V^*$ and $V$.
\item The inclusion $V \subset H$ is continuous and dense.
\end{itemize}
Given a fixed time interval $[0,T]$, we define the space $L^2(0,T; V)$ by:
\begin{equation}\label{eq.L20TV}
 L^2(0,T;V) = \left\{ u : [0,T] \to V, \:\: \int_0^T \lvert\lvert u(t) \lvert\lvert_V^2 \:\d t < \infty \right\}.
 \end{equation}
The latter is a Hilbert space when equipped with the natural inner product
$$ \langle u(t), v(t) \rangle_{ L^2(0,T;V) } := \int_0^T \langle u(t), v(t) \rangle_V \:\d t.$$
Its topological dual $L^2(0,T;V)^*$ is naturally identified to $L^2(0,T;V^*)$ via the pairing:
\begin{equation}\label{eq.iddualstspace}
 \langle f , u \rangle_{L^2(0,T;V^*),L^2(0,T;V)} = \int_0^T \langle f (t) , u(t) \rangle_{V^*,V} \:\d t.
\end{equation}

\begin{definition}
Let $V \subset H \subset V^*$ be an evolution triple, and let $Z$ be another Hilbert space. 
A function $u \in L^2(0,T;V)$ has a time derivative in the space $L^2(0,T; Z)$ if there exists $z \in L^2(0,T; Z)$ such that, for all real-valued test functions $\varphi \in \calC^\infty_c(0,T)$, it holds:
$$ \int_0^T u(t) \varphi^\prime(t) \:\d t = -\int_0^T z(t) \varphi(t) \:\d t.$$
This derivative is then denoted by $\frac{\partial u}{\partial t} = z$. 
\end{definition}

Let $V \subset H \subset V^*$ be an evolution triple. The time-dependent problems considered in this article bring into play Hilbert spaces of the form
\begin{equation}\label{eq.WOTVH}
W(0,T; V, H) = \Big\{ u \in L^2(0,T; V) \text{ s.t. } \frac{\partial u}{\partial t} \in L^2(0,T;V^*) \Big\},
\end{equation}
equipped with the natural norm
$$ || u ||_{W(0,T; V, H)}^2 := \int_0^T \lvert\lvert u(t) \lvert\lvert^2_{V}\:\d t + \int_0^T \left \lvert \left\lvert  \frac{\partial u}{\partial t}(t)\right\lvert\right\lvert^2_{V^*}  \:\d t.$$
We recall the following facts about such functional spaces, see Lemmas 1 and 2  of Chap XVIII, \S 1.2 of \cite{dautray1992evolution}.
\begin{lemma}\label{lem.densW}\noindent
\begin{enumerate}[(i)]
\item The subspace 
$$ \left\{ u(t) = \sum\limits_{k=0}^K t^k u_k , \:\: K \geq 0, \:\: u_k \in V \right\} \subset W(0,T;V,H)$$
of polynomial functions in the time variable with coefficients in $V$ is dense in $W(0,T; V,H)$.
In particular, the space $\calC^\infty([0,T],V)$ of smooth $V$-valued functions on $[0,T]$ is dense in $W(0,T; V,H)$.
\item The space $W(0,T; V, H)$ is continuously embedded in the space $\calC([0,T]; H)$ of continuous $H$-valued functions on $[0,T]$.
\end{enumerate}
\end{lemma} 
The last point of this lemma is of crucial importance in practice, as it gives a rigorous meaning to the time (e.g. initial, periodic) conditions usually imposed to the solution of an evolution problem when it is sought in $W(0,T;V,H)$.

\subsection{Preliminary results about the well-posedness of some non linear problems}\label{app.prelNL}

\noindent 
Let us first recall the following consequence of the Banach fixed point theorem, about the solution to a non linear stationary equation in a Hilbert space, whose operator satisfies suitable monotonicity and continuity conditions,
see Th. 25 B. from \cite{zeidler2013nonlinear}.

\begin{theorem}\label{th.wpnl}
Let $(V, \lvert\lvert\cdot \lvert\lvert_V)$ be a real Hilbert space and let $\calA: V \to V^*$ be a (possibly non linear) mapping satisfying the following properties:
\begin{enumerate}[(i)]
\item $\calA$ is strongly monotone, i.e. there exists $\alpha >0$ such that: 
$$ \forall u,v \in V, \quad \langle \calA u - \calA v , u -v \rangle_{V^*,V} \: \geq \: \alpha \lvert\lvert u - v \lvert\lvert_V ^2.$$
\item $\calA$ is Lipschitz continuous, i.e. there exists $L >0$ such that:
$$ \forall u , v \in V, \quad \lvert\lvert \calA u - \calA v \lvert\lvert_{V^*} \: \leq \: L \lvert\lvert u - v \lvert\lvert_V .$$
\end{enumerate}
Then, for each $b \in V^*$, the equation 
\begin{equation}\label{eq.NLabs}
 \calA u = b
 \end{equation}
has a unique solution $u \in V$. This solution has a Lipschitz continuous dependence with respect to the data: there exists a constant $C >0$ depending only on $\alpha$ and $L$ such that, if $u_1$, $u_2$ are the solutions to \cref{eq.NLabs} with respective right-hand sides $b_1$, $b_2 \in V^*$, it holds:
$$ \lvert\lvert u_1 - u_2 \lvert\lvert_V \: \leq\: C \lvert\lvert b_1 - b_2 \lvert\lvert_{V^*}.$$
\end{theorem} 

In particular, the above result is the pivotal ingredient in the proof of the following lemma about the well-posedness of a stationary version of the considered evolution problem \cref{eq.uOm}, see \cite{pechstein2006monotonicity} for a similar argument: 
\begin{lemma}\label{lem.wpbvpO}
Let $\calO$ be a bounded Lipschitz domain in $\R^d$, let $\hat\nu : \R_+ \to \R_+$ be a function satisfying \cref{eq.hypnu}, and let $B \in \calC^\infty(\overline{\calO}; \R^{d\times d})$ be a matrix-valued function such that $B(x)$ is invertible for all $x \in \overline{\calO}$. Then for all data $f \in L^2(\calO)$ and $g \in H^{1/2}(\partial \calO)$, the boundary value problem
\begin{equation}\label{eq.nlbvpcalO}
\left\{\begin{array}{cl}
-\dv (\hat \nu( \lvert B^T(x) \nabla u\lvert) B(x)B^T(x)\nabla u) = f & \text{in } \calO, \\
u = g & \text{on } \partial \calO,
\end{array} 
\right.
\end{equation}
has a unique solution $u \in H^1(\calO)$ which is Lipschitz continuous with respect to the data: the solutions $u_1$, $u_2 \in H^1(\calO)$ associated to different data $(f_1,g_1)$, $(f_2,g_2) \in L^2(\calO) \times H^{1/2}(\partial \calO)$ satisfy the following inequality: 
$$ \lvert\lvert u_1 - u_2 \lvert\lvert_{H^1(\calO)} \leq C \Big( \lvert\lvert f_1 - f_2 \lvert\lvert_{L^2(\calO)}  +  \lvert\lvert g_1 - g_2 \lvert\lvert_{H^{1/2}(\partial\calO)} \Big),$$
for a constant $C$ which depends only on the domain $\calO$, the matrix $B$ and the coefficients $\underline\nu$, $\overline\nu$ in \cref{eq.hypnu}. 
\end{lemma} 
\begin{proof}
Let $\widetilde g \in H^1(\calO)$ be such that $\widetilde g \lvert_{\partial \calO} = g$. 
Letting the change of unknown functions $w = u - \widetilde g \in H^1_0(\calO)$, the considered boundary value problem \cref{eq.nlbvpcalO} rewrites:
\begin{equation}\label{eq.nlbvpcalO2}
\left\{\begin{array}{cl}
-\dv (\hat\nu(\lvert B^T(x)(\nabla w + \nabla \widetilde g)\lvert) B(x)B^T(x) (\nabla w + \nabla \widetilde g) ) = f & \text{in } \calO, \\
w = 0 & \text{on } \partial \calO.
\end{array} 
\right.
\end{equation}
Let us equip $H^1_0(\calO)$ with the inner product
$$ \langle u , v \rangle_{H^1_0(\calO)} := \int_\calO \nabla u \cdot \nabla v \:\d x,$$
and let us introduce the operator $\calA : H^1_0(\calO) \to H^{-1}(\calO)$ defined by:
$$\calA w  = -\dv (\hat\nu(\lvert B^T(x)(\nabla w + \nabla \widetilde g) \lvert) B(x)B^T(x) (\nabla w + \nabla \widetilde g)). $$
We aim to show that $\calA$ is strongly monotone; to this end, let  $u_1,u_2 \in H^1_0(\calO)$ be arbitrary functions; introducing the shorthand $U_i = B^T(x)(\nabla u_i + \nabla \widetilde g) \in L^2(\calO)^d$,  $i=1,2$, an elementary calculation yields:
\begin{multline*} 
\left\langle \calA u_1 - \calA u_2 , u_1 - u_2 \right\rangle_{H^{-1}(\calO),H^1_0(\calO)} \\
\begin{array}{>{\displaystyle}cc>{\displaystyle}l}
&\hspace{-4.5cm} =&\hspace{-2cm}  \int_\calO  \Big( \hat\nu( \lvert U_1 \lvert) U_1 - \hat\nu(\lvert U_2 \lvert) U_2 \Big) \cdot (U_1 - U_2) \:\d x \\[1em]
&\hspace{-4.5cm} = & \hspace{-2cm} \underline\nu \int_\calO  \lvert U_1 - U_2 \lvert^2 \:\d x +  \int_\calO \Big( (\hat\nu( \lvert U_1 \lvert) -\underline\nu) U_1 -(\hat \nu(\lvert U_2 \lvert)-\underline\nu) U_2 \Big) \cdot (U_1 - U_2) \:\d x \\[1em]
&\hspace{-4.5cm}=&  \hspace{-2cm}  \underline\nu \int_\calO  \lvert U_1 - U_2 \lvert^2 \:\d x \\ [1em]
&\hspace{-4.5cm}&  \hspace{-2cm} +  \int_\calO \Big( (\hat\nu( \lvert U_1 \lvert) - \underline\nu) \lvert U_1 \lvert^2 + (\hat\nu(\lvert U_2 \lvert) - \underline\nu)\lvert U_2\lvert^2 - (\hat\nu(\lvert U_1 \lvert) + \hat\nu(\lvert U_2 \lvert) -2\underline\nu) U_1 \cdot U_2 \Big)  \:\d x \\[1em]
  &\hspace{-4.5cm} \geq& \hspace{-2cm}  \underline\nu \int_\calO  \lvert U_1 - U_2 \lvert^2 \:\d x \\ [1em]
  & \hspace{-4.5cm} &  \hspace{-2cm}  + \int_\calO  \Big( ( \hat\nu( \lvert U_1 \lvert) -\underline\nu) \lvert U_1 \lvert^2 + (\hat\nu(\lvert U_2 \lvert)-\underline\nu) \lvert U_2\lvert^2 - (\hat\nu(\lvert U_1 \lvert) + \hat\nu(\lvert U_2 \lvert) -2\underline\nu ) \lvert U_1 \lvert \lvert U_2 \lvert \Big)  \:\d x \\[1em]
    & \hspace{-4.5cm} =&  \hspace{-2cm} \underline\nu \int_\calO  \lvert U_1 - U_2 \lvert^2 \:\d x +  \int_\calO \Big( \hat\nu( \lvert U_1 \lvert) \lvert U_1 \lvert - \hat\nu(\lvert U_2 \lvert) \lvert U_2\lvert \Big)  (\lvert U_1 \lvert -\lvert U_2\lvert) \:\d x  \\[1em]
    & \hspace{-4.5cm} &\hspace{-2cm}  - \underline\nu  \int_{\calO }\lvert \lvert U_1\lvert - \lvert U_2 \lvert \lvert^2 \:\d x\\[1em]
    & \hspace{-4.5cm} \geq& \hspace{-2cm} \underline\nu \: \lvert\lvert U_1 - U_2 \lvert\lvert^2_{L^2(\calO)^d},
 \end{array}  
 \end{multline*}
 where we have used the Cauchy-Schwarz inequality to pass from the third to the fourth line and \cref{eq.hypnu} to obtain the last line. 
Now, by definition of $U_1$, $U_2$ it holds:
$$  \lvert\lvert U_1 - U_2 \lvert\lvert^2_{L^2(\calO)^d} = \int_{\calO} B(x)B^T(x) (\nabla u_1 - \nabla u_2) \cdot (\nabla u_1 - \nabla u_2) \:\d x.$$
Besides, since $B$ is a smooth matrix-valued function on $\overline{\calO}$ such that $B(x)$ is invertible for all $x \in \overline\calO$, there exists a constant $\underline\gamma > 0$ such that
$$\forall x \in \overline{\calO}, \: \forall \xi \in \R^d, \quad B(x) B^T(x) \xi \cdot \xi > \underline\gamma |\xi|^2. $$
We eventually infer from these facts that: 
$$  \left\langle \calA u_1 - \calA u_2 , u_1 - u_2 \right\rangle_{H^{-1}(\calO),H^1_0(\calO)} \geq \underline \nu \underline \gamma \lvert\lvert \nabla u_1 - \nabla u_2 \lvert\lvert_{L^2(\calO)^d}^2, $$
which precisely expresses that strong monotonicity of $\calA$. 

A similar calculation shows that $\calA : H^1_0(\calO) \to H^{-1}(\calO)$ is a Lipschitz continuous operator, with a Lipschitz constant depending only on $\overline \nu$ and the matrix $B(x)$. 
 
It follows from \cref{th.wpnl} that the problem \cref{eq.nlbvpcalO2} has a unique solution $w$, with Lipschitz continuous dependence on the data function $f$, 
which readily implies the desired statement about \cref{eq.nlbvpcalO}.
\end{proof}

Our study of \cref{eq.uOm} rests on results about the well-posedness of non linear evolution problems which stem from the general theory of maximal monotone operators.
For simplicity, we limit ourselves with a relatively simple statement which is enough for our purpose,  see e.g. Chap. 32 of \cite{zeidler2013nonlinear} (and notably Th. 32.D). 

\begin{theorem}\label{th.pbevolth}
Let $V \subset H \subset V^*$ be an evolution triple. 
For every $t \in (0,T)$, let $\calA(t) : V \to V^*$ be a strongly monotone and Lipschitz continuous operator, whose monotonicity and Lipschitz constants do not depend on $t$. Then for each $f \in L^2(0,T;V^*)$, the evolution equation 
$$ 
\left\{
\begin{array}{cl}
\frac{\partial u}{\partial t} + \calA u = f & \text{in } (0,T), \\ 
u(0) = u(T) ,
\end{array}
\right.
$$
has a unique solution $u$ in $W(0,T;V,H)$.
\end{theorem}

\begin{remark}
Strictly speaking, the statement of Th. 32.D in \cite{zeidler2013nonlinear} features a non linear operator $\calA$ which is independent of time. 
However, inspection of the proof of this result reveals that it carries over mutatis mutandis to the time-dependent case, provided the strong monotonicity and Lipschitz constants of $\calA$ are independent of time.
\end{remark}

\subsection{Proof of \cref{th.wellposed}}\label{app.proofThmain}

\noindent  We proceed in four steps. In the first three steps, we assume that there exists a solution $u \in \Wper$ to the variational problem \cref{eq.pbuOm}-\cref{eq.pbuOmae},
 and we construct an equivalent problem for $u$, see \cref{eq.pbevolDmag}; the well-posedness of the latter results from the application of \cref{th.pbevolth}.
This procedure implies, in particular, that a solution $u$ to \cref{eq.pbuOm} is unique, if it exists.
In the final step, we construct a solution to \cref{eq.pbuOm} from the solution to \cref{eq.pbevolDmag}, thus concluding about the existence part of the statement. \par\medskip 
 
\noindent \textit{Step 1: We use a change of variables to obtain an equivalent variational problem to \cref{eq.pbuOmae} featuring a fixed arrangement of phases within $D$.}\par\medskip

\noindent Let us assume that \cref{eq.pbuOm}-\cref{eq.pbuOmae} has a solution $u$ in the space $\Wper$ defined by \cref{eq.Xper}. 
Then, for a.e. $t \in (0,T)$, the following identity is satisfied for all test functions $w \in H^1_0(D)$:
\begin{multline}\label{eq.varfuOmt}
 \int_D \sigma_{\Omega(t)}(x)\left( \frac{\partial u}{\partial t}(t,x) + v(t,x) \cdot \nabla u(t,x) \right) w(x) \:\d x \\
 + \int_D \nu_{\Omega(t)}(x,\lvert \nabla u(t,x) \lvert) \nabla u(t,x) \cdot \nabla w(x) \:\d x = \int_D f(t,x) w(x) \:\d x.
 \end{multline}
We introduce the transported function
$$\overline u \in L^2(0,T; H^1_0(D)), \quad \overline u (t,x) := u(t,\varphi_t(x)) \text{ for } t \in [0,T] \text{ and } x \in D.$$
By considering test functions of the form $\overline w = w \circ \varphi_t^{-1}$, $w \in H^1_0(D)$, a change of variables based on $\varphi_t$ in \cref{eq.varfuOmt} yields, for a.e. $t \in (0,T)$:
\begin{multline}\label{eq.varfubarapp}
\forall w \in H^1_0(D), \:\:  \int_D \sigma_{\Omega}(x)  \frac{\partial \overline{u}}{\partial t}(t,x)  w(x) \:\d x \\
+ \int_D \nu_{\Omega}(x,\lvert  B_t^T(x) \nabla \overline u (t,x)\lvert) B_t(x) B^T_t(x) \nabla \overline{u}(t,x) \cdot \nabla w(x) \:\d x = \int_D f(t,\varphi_t(x)) w(x) \:\d x,
 \end{multline}
where, we have defined the smooth matrix-valued function: 
$$ B_t(x) = \nabla \varphi_t^{-1} (x),$$
and we have used the calculus identities \cref{eq.snut}.
The evolution problem \cref{eq.varfubarapp} is of mixed parabolic-elliptic problem, since $\sigma_\Omega$ vanishes outside the subset $\Dmag$ of $D$, where it takes the value $\sigma_m >0$.\par\medskip

\noindent \textit{Step 2: We reformulate \cref{eq.varfubarapp}  as a parabolic evolution problem posed on $\Dmag$.}\par\medskip

\noindent By using test functions $w$ with compact support inside $D \setminus \overline{\Dmag}$ in \cref{eq.varfubarapp}, we see that $\overline u$ satisfies:
$$ -\dv\Big(\nu_{\Omega}(x,\lvert  B_t^T(x) \nabla \overline u(t,x)\lvert) B_t(x) B^T_t(x) \nabla \overline{u}(t,x)\Big) = f(t,\varphi_t(x)) \text{ in } D \setminus \overline{\Dmag}.$$
On the other hand, an integration by parts in \cref{eq.varfubarapp} reveals that the spatial restriction of $\overline u(t,\cdot)$ to $\Dmag$, which we still denote by $\overline u$, satisfies the following problem:
\begin{multline}\label{eq.pbevolDmag}
\forall w \in H^1(\Dmag), \quad \sigma_m \int_{\Dmag}   \frac{\partial \overline{u}}{\partial t}(t,x)  w(x) \:\d x + \left\langle \calA \overline u, w \right\rangle_{H^1(\Dmag)^*,H^1(\Dmag)} \\
+ \left\langle \calB \overline u, w \right\rangle_{H^1(\Dmag)^*,H^1(\Dmag)}  = \int_{\Dmag} f(t,\varphi_t(x)) w(x) \:\d x.
\end{multline}
Here, we have introduced the linear operator $\calA : H^1(\Dmag) \to H^1(\Dmag)^*$ and the non linear operator $\calB : H^1(\Dmag) \to H^1(\Dmag)^*$ defined by, respectively:
$$ \forall u,w \in H^1(\Dmag), \:\: \left\langle \calA u, w \right\rangle_{H^1(\Dmag)^*,H^1(\Dmag)} = \nu_m \int_{\Dmag}  B_t(x) B_t^T(x) \nabla u(t,x) \cdot \nabla w(x) \:\d x ,$$
 and 
\begin{multline*}
\forall u,v \in H^1(\Dmag), \quad \left \langle \calB u , w \right \rangle_{H^1(\Dmag)^*,H^1(\Dmag)} = \\
- \int_{\partial \Dmag} \Big(\nu_{\Omega}(x,\lvert  B_t^T(x) \nabla \calL_t u (t,x)\lvert)  B_t(x)B_t^T(x) \nabla \calL_t u(t,x)\Big)\cdot n(x) \: w(x) \:\d s(x),
\end{multline*}
where $n$ is the unit normal vector to $\partial \Dmag$ pointing outward $\Dmag$ and for $z \in H^1(\Dmag)$, $\calL_t z \in H^1(D \setminus \overline{\Dmag})$ is the unique solution to the boundary value problem:
\begin{equation}\label{eq.defL}
 \left\{\begin{array}{cl}
-\dv (\nu(x,|B^T_t(x)\nabla \calL_t z(x)|) B_t(x) B_t^T(x) \nabla \calL_t z(x)) = f(t,x) & \text{in } D \setminus \overline{\Dmag}, \\
\calL_t z = z & \text{on } \partial \Dmag,\\
\calL_t z = 0 & \text{on } \partial D.
\end{array} 
\right.
\end{equation}
Let us then analyze $\calA$ and $\calB$. 
\begin{itemize}
\item The linear operator $\calA$ is continuous and monotone; more precisely, it holds:
\begin{equation*}\label{eq.monA}
 \forall u \in H^1(\Dmag),   \quad \left\langle \calA u,u \right\rangle_{H^1(\Dmag)^*,H^1(\Dmag)}  \geq C \left\lvert\left\lvert \nabla u \right\lvert\right\lvert^2_{L^2(\Dmag)^d}.
 \end{equation*}
Here and throughout the rest of the proof, $C$ stands for a positive constant which may change from one line to the other but is independent of time. 
The existence of such a constant in the above inequality is guaranteed by the smoothness and invertibility of the mapping $\varphi$.
\item The operator $\calB$ is Lipschitz continuous, with a Lipschitz constant independent of $t$, 
as readily follows from the application of \cref{lem.wpbvpO} to the non linear operator $\calL_t : H^1(\Dmag) \to H^1(D \setminus \overline{\Dmag})$.

Moreover, for all elements $u_1$, $u_2 \in H^1(\Dmag)$, several integrations by parts followed by a calculation identical to that conducted in the proof of \cref{lem.wpbvpO} yield:
$$
 \langle \calB u_1 - \calB u_2 , u_1 - u_2 \rangle_{H^1(\Dmag)^*,H^1(\Dmag)}  \geq C \lvert\lvert \nabla \calL_t u_1 - \nabla \calL_t u_2 \lvert\lvert^2_{L^2(D \setminus \overline \Dmag)^d}.
 $$
Since the mappings $ u\mapsto \lvert\lvert \nabla u \lvert\lvert_{L^2(D)^d}$ and $u \mapsto \lvert\lvert u \lvert\lvert_{H^1(D \setminus \overline{\Dmag})}$ are equivalent norms on the subspace of $H^1(D \setminus \overline \Dmag)$ functions with null trace on $\partial D$, we obtain:
 \begin{equation*}\label{eq.monB}
\begin{array}{ccl}
 \langle \calB u_1 - \calB u_2 , u_1 - u_2 \rangle_{H^1(\Dmag)^*,H^1(\Dmag)}  &\geq& C \lvert\lvert u_1 - u_2 \lvert\lvert^2_{H^1(D \setminus \overline{\Dmag})} \\
 &\geq& C \lvert\lvert u_1 - u_2 \lvert\lvert^2_{H^{1/2}(\partial \Dmag)},
 \end{array}
  \end{equation*}
 where the second line follows from the trace inequality. 
 \end{itemize}
The combination of both points reveals that the sum $\calA + \calB$ is Lipschitz continuous; it is also strongly monotone since, for all functions $u_1$, $u_2 \in H^1(\Dmag)$, it holds:
$$\begin{array}{>{\displaystyle}cc>{\displaystyle}l}
\left\langle (\calA + \calB) u_1 -(\calA + \calB) u_2, u_1- u_2 \right\rangle_{H^1(\Dmag)^*,H^1(\Dmag)} &\geq& C \left( \left\lvert\left\lvert \nabla (u_1-u_2) \right\lvert\right\lvert^2_{L^2(\Dmag)^d} +  \lvert\lvert u_1 - u_2 \lvert\lvert^2_{H^{1/2}(\partial \Dmag)}\right)  \\
 & \geq & C \left\lvert \left\lvert u_1 - u_2 \right\lvert\right\lvert_{H^1(\Dmag)}^2,
 \end{array}$$
as follows from (an avatar of) Poincar\'e's inequality. Moreover, the Lipschitz and strong monotonicity constants of $(\calA + \calB)$ are independent of time. 
 \par\medskip 

\noindent \textit{Step 3: We use the abstract \cref{th.pbevolth}.}\par\medskip

\noindent The application of \cref{th.pbevolth} reveals that the variational problem \cref{eq.pbevolDmag} has a unique solution $\overline u \in W(0,T;H^1(\Dmag),L^2(\Dmag))$ such that 
$\overline u(t=0, \cdot) = \overline u(t=T, \cdot)$.
This almost readily implies that the solution $u \in \Wper$ to the original evolution problem \cref{eq.pbuOm} is unique, when it exists. 
Indeed, if $u_i$, $i=1,2$ are two such solutions, let us define the functions $\overline{u_i} \in L^2(0,T; H^1_0(D))$ and $z_i \in L^2(0,T; H^1(\Dmag))$ by:
$$\overline{u_i}(t,x) = u_i(t,\varphi_t(x)), \text{ and } z_i(t,\cdot) = \overline{u_i}(t,\cdot) \lvert_{\Dmag}.$$ 
As a result of the previous steps, both functions $z_i$ belong to $W(0,T;H^1(\Dmag),L^2(\Dmag))$ and satisfy the evolution equation \cref{eq.pbevolDmag} with the time periodic condition $z_i(t=0,\cdot) = z_i(t=T, \cdot)$, $i=1,2$. As a result, $z_1 = z_2$, i.e. for a.e. $t \in [0,T]$, $\overline{u_1}(t,\cdot)$ and $\overline{u_2}(t,\cdot)$ coincide on $\Dmag$. Moreover, since $\overline{u_i}(t,\cdot)\lvert_{D \setminus \overline{\Dmag}} = \calL_t z_i$ on $D \setminus \overline{\Dmag}$, it immediately follows that $\overline{u_1} = \overline{u_2}$ in $L^2(0,T; H^1_0(D))$, so that $u_1 = u_2$. 

Summarizing, we have proved that there exists at most one solution $u \in \Wper$ to \cref{eq.pbevolDmag}.
\par\medskip

\noindent \textit{Step 4: We construct a solution $u \in \Wper$ to \cref{eq.pbevolDmag}.}\par\medskip

\noindent This task essentially relies on the previous three steps. 
Let $z \in W(0,T ; H^1(\Dmag),L^2(\Dmag))$ be the unique solution to the well-posed variational problem \cref{eq.pbevolDmag} equipped with time periodic boundary conditions:
$$ z(t=0, \cdot) = z(t=T,\cdot) \text{ in } L^2(\Dmag).$$
In particular, $z(t,\cdot)$ is an $H^1(\Dmag)$ function for a.e. $t \in [0,T]$, and so we may introduce the unique solution
$ \zext(t,\cdot) = \calL_t z(t,\cdot) \in H^1(D \setminus \overline{\Dmag})$ to the version of \cref{eq.defL} featuring Dirichlet data $z(t,\cdot)$ on $\partial \Dmag$. 
Let now $\overline u(t,\cdot) \in H^1_0(D)$ be defined by:
$$ \overline u(t,x) = \left\{ 
\begin{array}{cl}
z(t,x) & \text{if } x \in \Dmag, \\
\zext(t,x) & \text{if } x \in D\setminus \overline\Dmag. \\
\end{array}
\right.$$
By construction, $\overline u$ belongs to $L^2(0,T; H^1_0(D))$, and it is one solution to \cref{eq.varfubarapp}.

It is also easily seen that $\sigma_{\Omega} \frac{\partial \overline u}{\partial t} \in L^2(0,T; H^{-1}(D))$, since $\frac{\partial z}{\partial t} \in L^2(0,T; H^1(\Dmag)^*)$ and for a.e. $t\in [0,T]$ and for any test function $\psi \in H^1_0(D)$, 
$$ \left\langle \sigma_\Omega \frac{\partial \overline u}{\partial t}(t,\cdot), \psi \right\rangle_{H^{-1}(D), H^1_0(D)} = \sigma_m  \left\langle  \frac{\partial z}{\partial t}(t,\cdot), \psi \right\rangle_{H^1(\Dmag)^*, H^1(\Dmag)}.$$

Eventually, let $u \in \Wper$ be defined by 
$$u(t,x) = \overline u(t,\varphi_t^{-1}(x)), \quad t \in [0,T], \:\: x \in D.$$ 
By reusing the calculations conducted in the first step, we see that $u$ is one solution to \cref{eq.pbevolDmag}. The Lipschitz dependence of this function on the data function $f$ results immediately from the Lipschitz dependence of the solutions to the problems \cref{eq.defL,eq.pbevolDmag} with respect to their right-hand sides, and we omit the details for brevity. This terminates the proof.
\begin{flushright}
\qed
\end{flushright}

%
\bibliographystyle{siam}
\bibliography{genbib.bib}

\end{document}